\documentclass[a4paper,11 pt,twoside]{article}

\usepackage[T1]{fontenc}
\usepackage[english]{babel}
\usepackage{times}

\usepackage{amsmath}
\usepackage{amsfonts}
\usepackage{amssymb}
\usepackage{amsmath}
\usepackage{amsthm}
\usepackage{graphicx}
\usepackage{stmaryrd} 
\usepackage{enumerate}
\usepackage{multirow}
\usepackage{hyperref}
\usepackage[all]{hypcap}
\usepackage{fancyhdr}
\usepackage{upgreek}
\usepackage{wasysym} 

\usepackage{dsfont} 
\usepackage{mathrsfs} 
\usepackage{verbatim}
\usepackage[font=small,labelfont=bf]{caption}

\evensidemargin\oddsidemargin
\makeatletter
\g@addto@macro\th@plain{\thm@headpunct{}}
\makeatother
\newtheoremstyle{plain*}     {\medskipamount}{\medskipamount}{\itshape}{}{\bfseries\scshape}{}{5pt}{}
\theoremstyle{plain*}
\newtheorem*{theorem*}{Theorem}
\newtheorem*{proposition*}{Proposition}
\newtheoremstyle{plain}     {\medskipamount}{\medskipamount}{\itshape}{}{\bfseries\scshape}{.}{1em}{}
\theoremstyle{plain}
\newtheorem{theorem}{Theorem}
\newtheorem{lemma}{Lemma}

\newtheorem{proposition}[lemma]{Proposition}

\newtheorem{definition}[lemma]{Definition}
\newtheorem{remark}[lemma]{Remark}

\newtheorem{assumption}{Assumption}

\renewenvironment{proof}[1][Proof]{\begin{trivlist}
\item[\hskip \labelsep {\bfseries #1}]}{\end{trivlist}}

\def\runtitle#1{\gdef\@runtitle{\runninghead@case{#1}}}                      \def\@runtitle{}
\def\runauthor#1{{\def\etal{et al.}\gdef\@runauthor{\runninghead@case{#1}}}} \def\@runauthor{}

\makeatletter
\DeclareFontFamily{OMX}{MnSymbolE}{}
\DeclareSymbolFont{MnLargeSymbols}{OMX}{MnSymbolE}{m}{n}
\SetSymbolFont{MnLargeSymbols}{bold}{OMX}{MnSymbolE}{b}{n}
\DeclareFontShape{OMX}{MnSymbolE}{m}{n}{
    <-6>  MnSymbolE5
   <6-7>  MnSymbolE6
   <7-8>  MnSymbolE7
   <8-9>  MnSymbolE8
   <9-10> MnSymbolE9
  <10-12> MnSymbolE10
  <12->   MnSymbolE12
}{}
\DeclareFontShape{OMX}{MnSymbolE}{b}{n}{
    <-6>  MnSymbolE-Bold5
   <6-7>  MnSymbolE-Bold6
   <7-8>  MnSymbolE-Bold7
   <8-9>  MnSymbolE-Bold8
   <9-10> MnSymbolE-Bold9
  <10-12> MnSymbolE-Bold10
  <12->   MnSymbolE-Bold12
}{}

\let\llangle\@undefined
\let\rrangle\@undefined
\DeclareMathDelimiter{\llangle}{\mathopen}%
                     {MnLargeSymbols}{'164}{MnLargeSymbols}{'164}
\DeclareMathDelimiter{\rrangle}{\mathclose}%
                     {MnLargeSymbols}{'171}{MnLargeSymbols}{'171}
\makeatother

\def\cqfd{ \hfill $\blacksquare$}

\newcommand{\lbr}{\llbracket} 
\newcommand{\rbr}{\rrbracket}

\newcommand{\lambdaN}{\lambda_{\mbox{\tiny $N$}}}  
\newcommand{\lambdaNe}{\lambda_{\mbox{\tiny $N,\epsilon$}}}

\newcommand{\tiN}{\mbox{\tiny $N$}}
\newcommand{\tik}{\mbox{\tiny $k$}}

\newcommand{\Modun}{\textsc{Model 1}}
\newcommand{\Munw}{\textsc{Model 1-w}}
\newcommand{\Mdeux}{\textsc{Model 2}}

\newcommand{\fk}{f_{\tik}}
\newcommand{\gk}{g_{\tik}}
\newcommand{\gNk}{g_{\tik}^{\tiN}}

\newcommand{\frg}{\mathfrak{g}}  
\newcommand{\frf}{\mathfrak{f}}

\newcommand{\DN}{D_{\tiN}}

\newcommand{\DNMunw}{D_{\tiN}^{\mbox{\tiny\textsc{Mod $1$-w}}}}
\newcommand{\DNMde}{D_{\tiN}^{\mbox{\tiny\textsc{Mod $2$}}}}

\newcommand{\LN}{L_{\tiN}}

\newcommand{\LNMunw}{L_{\tiN}^{\mbox{\tiny\textsc{Mod $1$-w}}}}

\newcommand{\VNe}{V_{\mbox{\tiny $N,\epsilon$}}}

\newcommand{\nuN}{\nu_{\mbox{\tiny $N$}}}
\newcommand{\bbQN}{\bbQ^{\mbox{\tiny $N$}}}
\newcommand{\bbPN}{\bbP^{\mbox{\tiny $N$}}}
\newcommand{\bPMun}{\bP^{\mbox{\tiny\textsc{Mod $1$}}}}
\newcommand{\bPMunw}{\bP^{\mbox{\tiny\textsc{Mod $1$-w}}}}
\newcommand{\bPMde}{\bP^{\mbox{\tiny\textsc{Mod $2$}}}}
\newcommand{\bQMun}{\bQ^{\mbox{\tiny\textsc{Mod $1$}}}}
\newcommand{\bQMunw}{\bQ^{\mbox{\tiny\textsc{Mod $1$-w}}}}
\newcommand{\bQMde}{\bQ^{\mbox{\tiny\textsc{Mod $2$}}}}

\newcommand{\zetaun}{\zeta^{\mbox{\tiny $(1)$}}}
\newcommand{\zetaunn}{\zeta^{\mbox{\tiny $(1)$},n}}

\newcommand{\zetade}{\zeta^{\mbox{\tiny $(2)$}}}
\newcommand{\zetaden}{\zeta^{\mbox{\tiny $(2)$},n}}

\newcommand{\zetai}{\zeta^{\mbox{\tiny $(i)$}}}
\newcommand{\zetainftyun}{\zeta^{\mbox{\tiny $\infty,(1)$}}}
\newcommand{\zetainftyde}{\zeta^{\mbox{\tiny $\infty,(2)$}}}

\newcommand{\zetaNun}{\zeta^{\mbox{\tiny $N,(1)$}}}
\newcommand{\zetaNde}{\zeta^{\mbox{\tiny $N,(2)$}}}

\newcommand{\bmuN}{\mathbf{\mu_{\mbox{\tiny $N$}}}}  
\newcommand{\piN}{\mathbf{\pi_{\mbox{\tiny $N$}}}}   

\newcommand{\bmuNMde}{\mathbf{\mu^{\mbox{\tiny\textsc{Mod $2$}}}_{\mbox{\tiny $N$}}}}

\newcommand{\piNMun}{\mathbf{\pi^{\mbox{\tiny\textsc{Mod $1$}}}_{\mbox{\tiny $N$}}}}
\newcommand{\piNMunw}{\mathbf{\pi^{\mbox{\tiny\textsc{Mod $1$-w}}}_{\mbox{\tiny $N$}}}}
\newcommand{\piNMde}{\mathbf{\pi^{\mbox{\tiny\textsc{Mod $2$}}}_{\mbox{\tiny $N$}}}}

\newcommand{\pinMun}{\mathbf{\pi^{\mbox{\tiny\textsc{Mod $1$}}}_{\mbox{\tiny $n$}}}}
\newcommand{\piNnMunw}{\mathbf{\pi^{\mbox{\tiny\textsc{Mod $1$-w}}}_{\mbox{\tiny $N-n$}}}}
\newcommand{\pinMunw}{\mathbf{\pi^{\mbox{\tiny\textsc{Mod $1$-w}}}_{\mbox{\tiny $n$}}}}

\newcommand{\pN}{p_{\mbox{\tiny $N$}}}  
\newcommand{\qN}{q_{\mbox{\tiny $N$}}}

\newcommand{\cCN}{\cC_{\mbox{\tiny $N$}}}
\newcommand{\cCNMun}{\cC^{\mbox{\tiny\textsc{Mod $1$}}}_{\mbox{\tiny $N$}}}
\newcommand{\cCNMunw}{\cC^{\mbox{\tiny\textsc{Mod $1$-w}}}_{\mbox{\tiny $N$}}}
\newcommand{\cCnMun}{\cC^{\mbox{\tiny\textsc{Mod $1$}}}_{\mbox{\tiny $n$}}}

\newcommand{\cCNnMunw}{\cC^{\mbox{\tiny\textsc{Mod $1$-w}}}_{\mbox{\tiny $N-n$}}}
\newcommand{\cCNMde}{\cC^{\mbox{\tiny\textsc{Mod $2$}}}_{\mbox{\tiny $N$}}}
\newcommand{\cCNnMde}{\cC^{\mbox{\tiny\textsc{Mod $2$}}}_{\mbox{\tiny $N,n$}}}

\newcommand{\cCMun}{\cC^{\mbox{\tiny\textsc{Mod $1$}}}}
\newcommand{\cCMunw}{\cC^{\mbox{\tiny\textsc{Mod $1$-w}}}}
\newcommand{\cCMde}{\cC^{\mbox{\tiny\textsc{Mod $2$}}}}

\newcommand{\cEN}{\cE_{\mbox{\tiny $N$}}}
\newcommand{\cENMun}{\cE^{\mbox{\tiny\textsc{Mod $1$}}}_{\mbox{\tiny $N$}}}
\newcommand{\cENMunw}{\cE^{\mbox{\tiny\textsc{Mod $1$-w}}}_{\mbox{\tiny $N$}}}
\newcommand{\cENMde}{\cE^{\mbox{\tiny\textsc{Mod $2$}}}_{\mbox{\tiny $N$}}}

\newcommand{\cON}{\cO_{\mbox{\tiny $2N$}}}
\newcommand{\cOk}{\cO_{\mbox{\tiny $2k\!\!+\!\!1$}}}
\newcommand{\cOkl}{\cO_{\mbox{\tiny $2k\!\!+\!\!1\!,\!l$}}}
\newcommand{\cVN}{\cV_{\mbox{\tiny $N$}}}
\newcommand{\GNk}{G_{\mbox{\tiny $N$}}^{\mbox{\tiny $k$}}}

\newcommand{\tcBDG}{\mathtt{c}_{\textup{\tiny BDG}}}
\newcommand{\tcInterpo}{\mathtt{c}_{\textup{\tiny Interpo}}}

\newcommand{\cLN}{\cL^{\mbox{\tiny $N$}}}
\newcommand{\cRN}{\cR^{\mbox{\tiny $N$}}}
\newcommand{\cLNun}{\cL^{\mbox{\tiny $1\!,\!N$}}}
\newcommand{\cLNde}{\cL^{\mbox{\tiny $2\!,\!N$}}}
\newcommand{\cRNun}{\cR^{\mbox{\tiny $1\!,\!N$}}}
\newcommand{\cRNde}{\cR^{\mbox{\tiny $2\!,\!N$}}}
\newcommand{\Mun}{M^{\mbox{\tiny $(1)$}}}
\newcommand{\Mde}{M^{\mbox{\tiny $(2)$}}}
\newcommand{\Mi}{M^{\mbox{\tiny $(i)$}}}
\newcommand{\MNun}{M^{\mbox{\tiny $N,(1)$}}}
\newcommand{\MNde}{M^{\mbox{\tiny $N,(2)$}}}

\newcommand{\Minftyun}{M^{\mbox{\tiny $\infty,(1)$}}}

\newcommand{\Lun}{L^{\mbox{\tiny $(1)$}}}

\newcommand{\Li}{L^{\mbox{\tiny $(i)$}}}
\newcommand{\LNun}{L^{\mbox{\tiny $N,(1)$}}}

\newcommand{\rmhun}{\mathrm{h}^{\mbox{\tiny $(1)$}}}
\newcommand{\rmhde}{\mathrm{h}^{\mbox{\tiny $(2)$}}}
\newcommand{\rmhi}{\mathrm{h}^{\mbox{\tiny $(i)$}}}

\newcommand{\rmhN}{\mathrm{h}^{\mbox{\tiny $N$}}}
\newcommand{\rmhNun}{\mathrm{h}^{\mbox{\tiny $N,(1)$}}}
\newcommand{\rmhinfty}{\mathrm{h}^{\mbox{\tiny $\infty$}}}
\newcommand{\rmhinftyun}{\mathrm{h}^{\mbox{\tiny $\infty,(1)$}}}

\newcommand{\upetaun}{\upeta^{\mbox{\tiny $(1)$}}}

\newcommand{\etaun}{\eta^{\mbox{\tiny $(1)$}}}
\newcommand{\etade}{\eta^{\mbox{\tiny $(2)$}}}
\newcommand{\hun}{h^{\mbox{\tiny $(1)$}}}
\newcommand{\hde}{h^{\mbox{\tiny $(2)$}}}
\newcommand{\hi}{h^{\mbox{\tiny $(i)$}}}

\newcommand{\brmh}{\mathrm{\bar{h}}}  
\newcommand{\brmhi}{\mathrm{\bar{h}}^{\mbox{\tiny $(i)$}}}
\newcommand{\brmhun}{\mathrm{\bar{h}}^{\mbox{\tiny $(1)$}}}
\newcommand{\brmhde}{\mathrm{\bar{h}}^{\mbox{\tiny $(2)$}}}
\newcommand{\Hrmhi}{\hat{\mathrm{h}}^{\mbox{\tiny $(i)$}}}
\newcommand{\Hrmhun}{\hat{\mathrm{h}}^{\mbox{\tiny $(1)$}}}

\newcommand{\Hbrmhi}{\hat{\bar{\mathrm{h}}}^{\mbox{\tiny $(i)$}}}
\newcommand{\Hbrmhun}{\hat{\bar{\mathrm{h}}}^{\mbox{\tiny $(1)$}}}

\newcommand{\Wun}{W^{\mbox{\tiny $(1)$}}}
\newcommand{\Wde}{W^{\mbox{\tiny $(2)$}}}
\newcommand{\Wi}{W^{\mbox{\tiny $(i)$}}}

\newcommand{\Bun}{B^{\mbox{\tiny $(1)$}}}
\newcommand{\Bde}{B^{\mbox{\tiny $(2)$}}}

\newcommand{\rmh}{\mathrm{h}}

\newcommand{\tun}{\mathtt{1}}
\newcommand{\tc}{\mathtt{c}}

\newcommand{\tk}{\mathtt{k}}

\newcommand{\rL}{\mathrm{L}}

\newcommand{\rS}{\mathrm{S}}

\newcommand{\bP}{\mathbf{P}}
\newcommand{\bQ}{\mathbf{Q}}

\newcommand{\bbC}{\mathbb{C}}
\newcommand{\bbD}{\mathbb{D}}

\newcommand{\bbM}{\mathbb{M}}
\newcommand{\bbN}{\mathbb{N}}

\newcommand{\bbP}{\mathbb{P}}
\newcommand{\bbQ}{\mathbb{Q}}
\newcommand{\bbR}{\mathbb{R}}

\newcommand{\bbZ}{\mathbb{Z}}

\newcommand{\cA}{\mathcal{A}}

\newcommand{\cC}{\mathcal{C}}

\newcommand{\cE}{\mathcal{E}}
\newcommand{\cF}{\mathcal{F}}

\newcommand{\cH}{\mathcal{H}}
\newcommand{\cI}{\mathcal{I}}

\newcommand{\cL}{\mathcal{L}}

\newcommand{\cN}{\mathcal{N}}
\newcommand{\cO}{\mathcal{O}}

\newcommand{\cR}{\mathcal{R}}

\newcommand{\cV}{\mathcal{V}}
\newcommand{\cW}{\mathcal{W}}

\newcommand{\ccA}{\mathscr{A}}

\newcommand{\ccC}{\mathscr{C}}

\newcommand{\ccF}{\mathscr{F}}

\newcommand{\ccP}{\mathscr{P}}

\hypersetup{
citecolor=blue,
colorlinks=true, 
breaklinks=true, 
urlcolor= blue, 
linkcolor= black, 
bookmarksopen=true, 
pdftitle={Scaling limits}, 
}

\begin{document}

\title{ \textbf{\textsc{Scaling limits of weakly asymmetric interfaces}}}
\author{Alison M. {\sc Etheridge}
\thanks{Department of Statistics, University of Oxford, 1 South Parks Road, Oxford OX1 3TG, UK. \textsl{Email:} etheridg@stats.ox.ac.uk} \and Cyril {\sc Labb\'e}
\thanks{Mathematics Institute,
University of Warwick, Coventry CV4 7AL, UK.
\textsl{Email:} c.labbe@warwick.ac.uk} }
\vspace{2mm}

\pagestyle{fancy}
\fancyhead[LO]{}
\fancyhead[CO]{\sc{A.~Etheridge and C.~Labb\'e}}
\fancyhead[RO]{}
\fancyhead[LE]{}
\fancyhead[CE]{\sc{Scaling limits of weakly asymmetric interfaces}}
\fancyhead[RE]{}

\date{\small\today}

\maketitle

\begin{abstract}
We consider three models of evolving interfaces intimately related to the weakly asymmetric simple exclusion process with $N$ particles on a finite lattice of $2N$ sites. Our \Modun\ defines an evolving bridge on $[0,1]$, our \Munw\ an evolving excursion on $[0,1]$ while our \Mdeux\ consists of an evolving pair of non-crossing bridges on $[0,1]$. Based on the observation that the invariant measures of the dynamics depend on the area under (or between) the interface(s), we characterise the scaling limits of the invariant measures when the asymmetry of the exclusion process scales like $N^{-\frac{3}{2}}$. Then, we show that the scaling limits of the dynamics themselves are expressed in terms of variants of the stochastic heat equation. In particular, in \Munw\ we obtain the well-studied reflected stochastic heat equation introduced by Nualart and Pardoux~\cite{NualartPardoux92}.

\medskip

\end{abstract}

\maketitle

\section{Introduction}
Consider a collection of $N$ particles located on the linear lattice $\{1,2,\ldots,2N\}$ and subject to the exclusion rule that prevents two particles from sharing the same site. A particle configuration $\eta$ is therefore an element of $\{0,1\}^{2N}$ with $N$ occurrences of $1$, each $1$ encoding the presence of a particle. We denote by $\cENMun$ this state-space, the reason for the superscript will be made clear below. The \textit{simple exclusion process} consists of the following dynamics on $\cENMun$: each particle, independently of the others, jumps to its left (respectively its right) at rate $\pN$ (respectively $\qN$) if the target site is unoccupied. Notice that we do not consider periodic boundary conditions on our lattice so that a particle at site $1$ (respectively at site $2N$) cannot jump to its left (respectively to its right). When $\pN\ne\qN$ but $\pN/\qN \rightarrow 1$ as $N\rightarrow\infty$, the process is called the weakly asymmetric simple exclusion process (WASEP). In the present work, we introduce three models of interfaces intimately related to this process. Our \Modun\ defines an evolving interface which turns out to be the height function associated with a WASEP. Our \Munw\ is obtained from \Modun\ by adding the condition that the interface remains non-negative. Our \Mdeux\ consists of a pair of interfaces, each being associated to a WASEP, but with the condition that these interfaces cannot cross. We refer to Figure \ref{FigLattice1} for an illustration. In any of the three models, the area under the interface - or between the two interfaces - will play a central r\^ole. The main results of this paper consist of the characterisation of the scaling limits of these three dynamics via variants of the stochastic heat equation.

Our \Modun\ is related to evolutional (or dynamical) Young diagrams, we refer in particular to the works of Funaki and Sasada~\cite{FunakiSasada10} and Funaki, Sasada, Sauer and Xie~\cite{FunakiSasadaSauerXie13} where the authors study the scaling limits of Young diagrams conditioned on their area. We also refer to Dunlop, Ferrari and Fontes~\cite{DunlopFerrariFontes02} for the study of the long-time behaviour of a setting similar to our \Munw\ but on the infinite lattice $\bbZ$.

These interfaces can also be interpreted as polymers. In particular our \Munw, in the symmetric case $\pN=\qN$, coincides with the case $\lambda=1$ of the polymer model considered by Lacoin~\cite{Lacoin13} and Caputo, Martinelli and Toninelli~\cite{CaputoMartinelliToninelli08}. Indeed, in these references the authors consider the measure $\lambda^{\#\{x:h(x)=0\}}$ on the set of non-negative lattice paths $h$ (or polymers) starting at $0$ and ending at $0$ after $2N$ steps; therefore the case $\lambda =1$ yields the uniform measure. The dynamics considered by the authors is the corner flip dynamics with rates that can depend on $\lambda$ when the interface touches the wall: in the particular case $\lambda=1$ this is exactly our dynamics. In his paper, Lacoin studies the dynamical interface scaled by a factor $\frac{1}{2N}$ and shows that the scaling limit is given by the heat equation with Dirichlet boundary conditions: therefore, the hydrodynamical limit does not feel the effect of the wall. Notice that the invariant measure of this dynamics scales like $\sqrt{2N}$. In the present work, we look at this precise scaling, that is, we divide the interface by a factor $\frac{1}{\sqrt{2N}}$ and investigate the existence of a scaling limit. It turns out that under this scaling, the interface feels the effect of the wall so that we need to deal with some random reflecting measure at height $0$. We obtain the Nualart-Pardoux~\cite{NualartPardoux92} reflected stochastic heat equation in the limit, see the precise statement below. We also refer to Caravenna and Deuschel~\cite{CaravennaDeuschel08,CaravennaDeuschel09} for various results on the static behaviour of related models of polymers.

\begin{figure}[h!]
\centering
\includegraphics[width=6cm,angle=0]{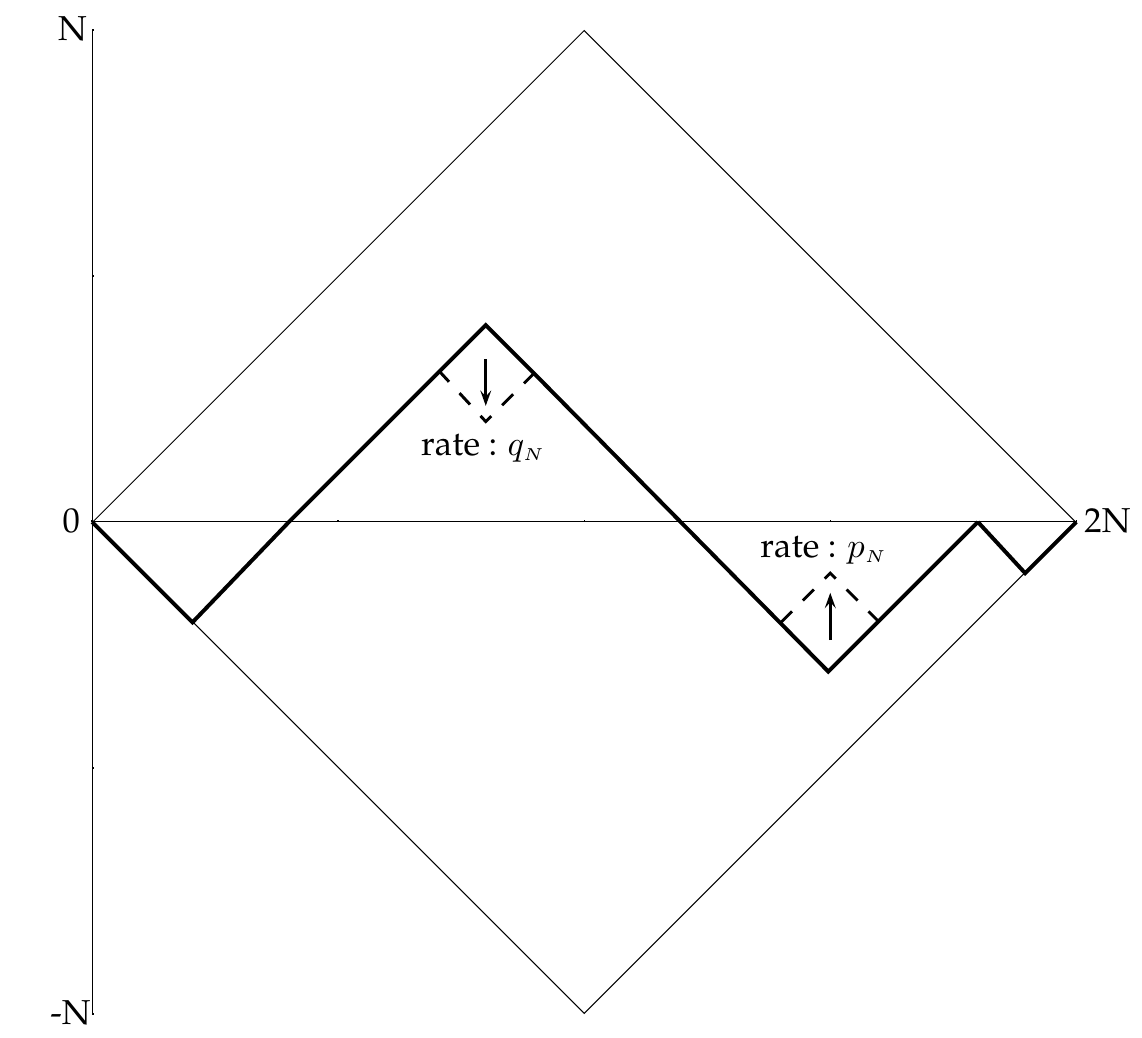}\hfill
\includegraphics[width=6cm,angle=0]{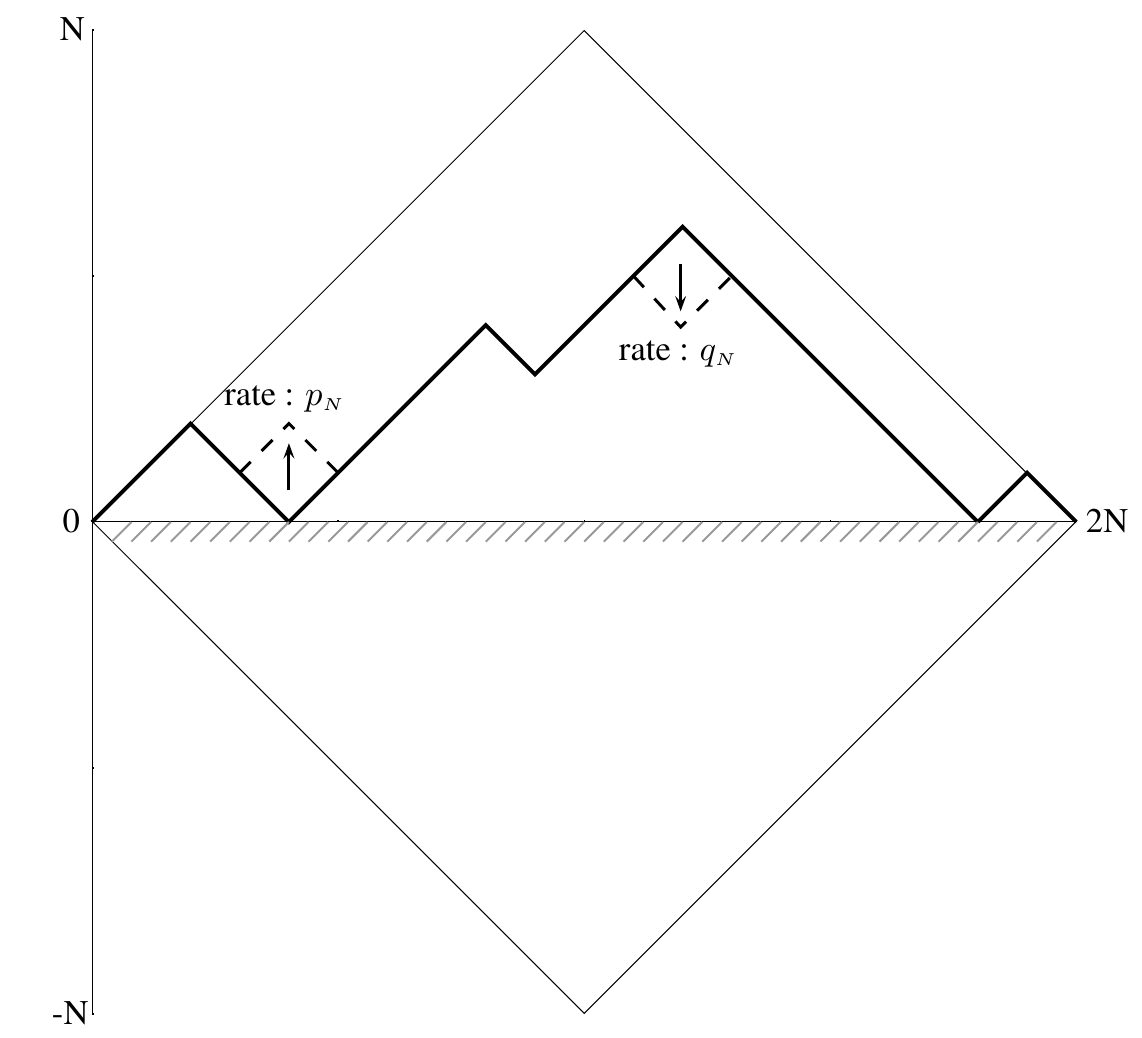}\hfill
\includegraphics[width=6cm,angle=0]{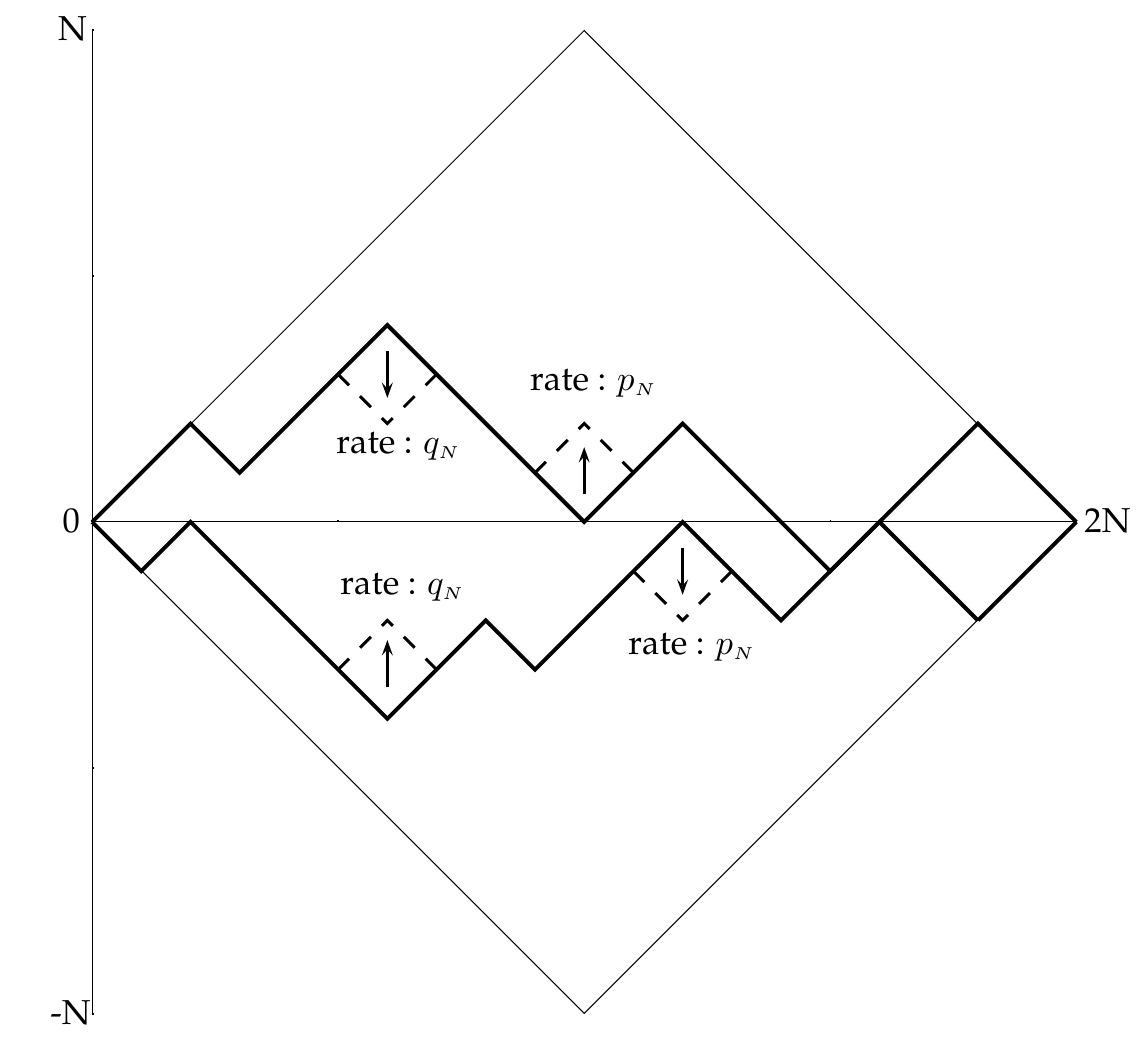}
\caption{Upper left \Modun, upper right \Munw, bottom \Mdeux. We have drawn the un-scaled interfaces; to obtain $\rmh$ one needs to rescale the interval $[0,2N]$ into $[0,1]$ and to divide the height of the interfaces by $\sqrt{2N}$.}\label{FigLattice1}
\end{figure}

Our models are discrete counterparts of the so-called $\nabla \varphi$ interface models. Let us recall that a $\nabla \varphi$ interface model is a finite system of coupled oscillators: each oscillator solves an SDE with a Brownian noise and a drift that depends on its position relative to its neighbours. In our models, one can interpret the collection of values $h(0),h(\frac{1}{N}),\ldots,h(\frac{N-1}{N}),h(1)$ as discrete oscillators which solve an SDE driven by a Poisson noise and a drift equal to the discrete Laplacian. We refer to Giacomin, Olla and Spohn~\cite{GiacominOllaSpohn01} for a setting similar to our \Modun\ but in higher dimension, to Funaki and Olla~\cite{FunaOll01} for a study of a $\nabla \varphi$ interface model constrained by a wall, and to Funaki~\cite{Funaki05} for a general review of $\nabla\varphi$ interface models.

Our motivation for \Mdeux\ came from the study of hybrid zones in population genetics. We suppose that each individual in a population undergoing biparental mating carries one of two forms (alleles) of a gene. Two parents of the same type have greater reproductive success than parents of different types. To caricature this situation we impose $\pN < \qN$ so that the two interfaces tend to move towards one another. The `hybrid zone' corresponds to the region between the two interfaces.

\medskip

\noindent Before we state our results, we need to describe our models more precisely. The underlying idea in any of the models is to consider lattice paths on $[0,2N]$ that start at $0$, make $+1/-1$ steps and come back to $0$ after $2N$ steps. In order to investigate potential scaling limits, we actually need to rescale these lattice paths suitably. Let us now provide the rigorous definitions.
\paragraph{\Modun.} Fix an integer $N\!\geq\! 1$ and set $k_{\tiN}:=\frac{k}{2N}$ for any $k\in\lbr 0,2N\rbr:=\{0,1,\ldots,2N\}$. Our state-space is the set 
\[\cCNMun:= \left\{
h:[0,1]\rightarrow\bbR\mbox{ s.t. }\;\;\;\begin{array}{l}
h(0)=h(1)=0,\\
h(k_{\tiN})=h((k-1)_{\tiN})\pm\frac{1}{\sqrt{2N}},\;\;\forall k\in\lbr 1,2N\rbr,\\
h\mbox{ is affine on every interval }[(k-1)_{\tiN},k_{\tiN}]
\end{array}
\right\}.\]
We write $\Delta$ for the discrete Laplacian on $\cCNMun$:
\[\forall k\in\lbr 1,2N-1\rbr,\;\;\;\Delta h(k_{\tiN}) := h\big((k\!+\!1)_{\tiN}\big)-2\,h\big(k_{\tiN}\big)+h\big((k\!-\!1)_{\tiN}\big).\]
Note that $\Delta$ implicitly depends on $N$ but this will never cause any confusion. The definition of $\cCNMun$ implies that $\Delta h(k_{\tiN})$ can only take the values $-\frac{2}{\sqrt{2N}}$, $0$ and $\frac{2}{\sqrt{2N}}$. Consequently we will write $\{\Delta h(k_{\tiN}) < 0\}$ and $\{\Delta h(k_{\tiN})>0\}$ to denote the first and third cases respectively.

For every $k\in\lbr 1,2N-1\rbr$, let $\pN(k_{\tiN})$ and $\qN(k_{\tiN})$ be two positive real numbers such that $\pN(k_{\tiN})+\qN(k_{\tiN})=(2N)^2$. We consider a probability space $(\Omega^{\tiN},\ccF^{\tiN},P^{\tiN})$ on which are defined two collections of independent Poisson processes $\cLN(k_{\tiN}),k\in \lbr 1,2N\!-\!1 \rbr$ and $\cRN(k_{\tiN}),k\in \lbr 1,2N\!-\!1 \rbr$ with jump rates $\pN(k_{\tiN})$ and $\qN(k_{\tiN})$ respectively. For a given initial condition $\rmh_0\!\in\!\cCNMun$, we define the $\cCNMun$-valued process $t\mapsto\rmh_t$ as the unique solution of the following finite system of stochastic differential equations:
\begin{equation}\label{Eq:DefMun}
\forall k\in\lbr 1,2N-1\rbr,\;\;\;d\rmh_t(k_{\tiN}) = \frac{2}{\sqrt{2N}}\Big(d\cLN_t(k_{\tiN})\, \tun_{\{\Delta \rmh_t(k_{\tiN}) >0\}} - d\cRN_t(k_{\tiN}) \,\tun_{\{\Delta \rmh_t(k_{\tiN}) <0\}}\Big).
\end{equation}
The process $\rmh$ can be informally described as follows. If at position $k_{\tiN}$ we have a local maximum, \textit{i.e.}~$\Delta\rmh_t(k_{\tiN}) < 0$, then at rate $\qN(k_{\tiN})$ the process $\rmh_t(k_{\tiN})$ jumps to $\rmh_t(k_{\tiN})-\frac{2}{\sqrt{2N}}$ so that it becomes a local minimum, \textit{i.e.}~$\Delta\rmh_t(k_{\tiN}) > 0$. The converse occurs at rate $\pN(\cdot)$. Recall the state-space $\cENMun$ introduced at the beginning of the article. Our process can be viewed as the evolving height function associated with a simple exclusion process. Indeed, there is a well-known correspondence between $\cENMun$ and $\cCNMun$: a positive/negative slope on $[(k-1)_{\tiN}, k_{\tiN}]$ corresponds to the presence/absence of a particle at the $k$-th site. The dynamics on $\cCNMun$, once translated in terms of $\cENMun$, defines the so-called simple exclusion process: flipping a local maximum downward corresponds to a jump of a particle to its right and vice-versa.

Let $\cCMun \!\supset\! \cCNMun$ be the space of continuous functions on $[0,1]$ that vanish at the boundaries. We denote by $\bbQN$ the distribution of $(\rmh_t,t\geq 0)$ on $\bbD([0,\infty),\cCMun)$ taken to be the Skorohod space of c\`adl\`ag $\cCMun$-valued functions. To emphasise the initial condition, we will write $\bbQN_{\nuN}$ when $\rmh_0$ is a random variable independent of the Poisson processes and distributed according to a given probability measure $\nuN$ on $\cCNMun$.\vspace{-10pt}
\paragraph{\Munw.} We define a modification of the first model by adding a reflecting wall for the interface at $0$. The state-space $\cCNMunw$ is the restriction of that of \Modun\ to the non-negative functions:
\[\cCNMunw := \left\{h\in \cCNMun\mbox{ s.t. }\;\; h(x)\geq 0,\;\; \forall x\in[0,1]\right\}.\]
All the previous definitions still hold except that the system of stochastic differential equations is now:
\begin{equation}\label{Eq:DefMunw}
\forall k\in\lbr 1,2N\!-\!1\rbr,\;\;\;d\rmh_t(k_{\tiN}) = \frac{2}{\sqrt{2N}}\Big(d\cLN_t(k_{\tiN})\, \tun_{\{\Delta \rmh_t(k_{\tiN}) >0\}} -  d\cRN_t(k_{\tiN}) \,\tun_{\{\Delta \rmh_t(k_{\tiN}) <0;\rmh_t(k_{\tiN})>\frac{1}{\sqrt{2N}}\}}\Big).
\end{equation}
The additional condition on the second term prevents the interface from becoming negative: if $\Delta \rmh_t(k_{\tiN})\! <0$ and $\rmh_t(k_{\tiN})=\frac{1}{\sqrt{2N}}$ then $\rmh_t((k-1)_{\tiN})=\rmh_t((k+1)_{\tiN})=0$ and a downward jump would make $\rmh_t(k_{\tiN})$ negative. We also set
\begin{equation}\label{EqZeta}
\forall t\geq 0,\forall x\in(0,1),\;\; \zeta(dt,dx) := \sum_{k=1}^{2N-1}\frac{2\qN(k_{\tiN})}{(2N)^{\frac{3}{2}}}\tun_{\{\Delta \rmh_t(k_{\tiN})<0 ; \rmh_t(k_{\tiN})=\frac{1}{\sqrt{2N}}\}}\delta_{\{k_{\tiN}\}}(dx)dt
\end{equation}
which is a random element of the space $\bbM$ of Borel measures on $[0,\infty)\!\times\!(0,1)$ satisfying $\int_{[0,t]\times(0,1)}x(1-x)\zeta(ds,dx)<\infty$ for every $t\geq 0$. We refer to Subsection \ref{Subsection:TightnessZeta} for the definition of the topology on this set of measures. The study of this random measure is necessary in order to characterise the scaling limit of $\rmh$. Indeed, the derivative in time of $\rmh$ in \Munw\ is the same as that in \Modun\ plus a reflection term involving the measure $\zeta$. At the limit $N\rightarrow\infty$, this random measure cannot be explicitly expressed in terms of $\rmh$ so that it needs to be obtained as a limit from the discrete setting.

The set $\cCMunw\!\supset\! \cCNMunw$ is taken to be the set of non-negative continuous functions on $[0,1]$ that vanish at the boundaries. Then we define $\bbQN_{\nuN}$ as the distribution of the pair $(\rmh,\zeta)$ on the product space $\bbD([0,\infty),\cCMunw)\times\bbM$ when $\rmh_0$ is a random variable independent of the Poisson processes and is distributed according to a given probability measure $\nuN$ on $\cCNMunw$.\vspace{-10pt}

\paragraph{\Mdeux.} The state-space $\cCNMde$ is the following set of pairs of interfaces:
\[\cCNMde := \left\{h=(\hun,\hde)\mbox{ s.t. }\;\;\hun,\hde \in \cCNMun\mbox { and }\; \hun(x) \geq \hde(x),\;\; \forall x\in[0,1]\right\}.\]
We call $\hun$ the upper interface and $\hde$ the lower interface. Let us describe the dynamics informally. The upper interface follows the same dynamics as in \Modun\ while the lower interface follows the opposite dynamics, that is, it jumps upward at rate $\qN(\cdot)$ and downward at rate $\pN(\cdot)$. Additionally, any jump that would break the ordering of the interfaces is erased. Formally, we define four collections of independent Poisson processes $\cLNun(k_{\tiN}),\cRNde(k_{\tiN}),k\in \lbr 1,2N\!-\!1 \rbr$ and $\cRNun(k_{\tiN}),\cLNde(k_{\tiN}),k\in \lbr 1,2N\!-\!1 \rbr$ with jump rates $\pN(k_{\tiN})$ for the first two and $\qN(k_{\tiN})$ for the last two. Then $t\mapsto\rmh_t:=(\rmhun_t,\rmhde_t)$ is the unique solution of the following system of stochastic differential equations:
\begin{equation}\label{Eq:DefMde}
\begin{split}
d\rmhun_t(k_{\tiN}) = \frac{2}{\sqrt{2N}}\Big(d\cLNun_t(k_{\tiN})\, \tun_{\{\Delta \rmhun_t(k_{\tiN}) >0\}} - d\cRNun_t(k_{\tiN}) \,\tun_{\{\Delta \rmhun_t(k_{\tiN}) <0;\rmhun_t(k_{\tiN}) > \rmhde_t(k_{\tiN})\}}\Big),\\
d\rmhde_t(k_{\tiN}) = \frac{2}{\sqrt{2N}}\Big(d\cLNde_t(k_{\tiN})\, \tun_{\{\Delta \rmhde_t(k_{\tiN}) >0;\rmhun_t(k_{\tiN})>\rmhde_t(k_{\tiN})\}} - d\cRNde_t(k_{\tiN}) \,\tun_{\{\Delta \rmhun_t(k_{\tiN}) <0\}}\Big).
\end{split}
\end{equation}
The condition $\rmhun_t(k_{\tiN}) > \rmhde_t(k_{\tiN})$ prevents the upper interface from passing below the lower interface, and vice-versa. We also introduce two random measures as follows:
\begin{equation}\label{EqZeta2}
\begin{split}
\zetaun(dt,dx) := \sum_{k=1}^{2N-1}\frac{2\qN(k_{\tiN})}{(2N)^{\frac{3}{2}}}\tun_{\{\Delta \rmhun_t(k_{\tiN})<0 ; \rmhun_t(k_{\tiN})=\rmhde_t(k_{\tiN})\}}\delta_{\{k_{\tiN}\}}(dx)dt,\\
\zetade(dt,dx) := \sum_{k=1}^{2N-1}\frac{2\qN(k_{\tiN})}{(2N)^{\frac{3}{2}}}\tun_{\{\Delta \rmhde_t(k_{\tiN})>0 ; \rmhun_t(k_{\tiN})=\rmhde_t(k_{\tiN})\}}\delta_{\{k_{\tiN}\}}(dx)dt.
\end{split}
\end{equation}
They are both random elements of the space $\bbM$ introduced above. Then we define $\bbQN_{\nuN}$ as the law of $(\rmh,\zetaun,\zetade)$ on $\bbD([0,\infty),\cCMde)\times\bbM\times\bbM$, under which $\rmh_0=(\rmhun_0,\rmhde_0)$ is a random variable with law $\nuN$ and independent of the Poisson processes. Here $\cCMde\!\supset\! \cCNMde$ denotes the space of continuous $\bbR^2$-valued functions $h=(\hun,\hde)$ on $[0,1]$ such that $\hun(x)\geq \hde(x)$ for every $x\in[0,1]$ and both $\hun,\hde$ vanish at the boundaries of $[0,1]$.\\

Let us emphasise our deliberate use of the same symbol $\bbQN$ in any of the three models in order to alleviate the notation. Moreover we will sometimes drop the superscript associated to the model and use the generic notation $\cCN$ and $\cC$ whenever a result applies indifferently to any of the three models. For any probability measure $\nu$ on $\cC$, we adopt the usual notation $\nu[F]:=\int_{\cC}F(h)\nu(dh)$ to denote the $\nu$-expectation of a measurable map $F:\cC\rightarrow\bbR$. Let us also introduce the notation
\[ \left\|h\right\|_{\cC}=\begin{cases}\sup_{x\in[0,1]}|h(x)|&\mbox{ in \Modun\ and \Munw},\\
\sup_{x\in[0,1]}|\hun(x)| + |\hde(x)|&\mbox{ in \Mdeux}.\end{cases}\]

\subsection{Main results}
We start with a result whose statement - in the case of \Modun\ - already appears in various forms in the literature, see for instance Janowsky and Lebowitz~\cite{JanowskyLebowitz92} or Funaki and Sasada~\cite{FunakiSasada10}.
\begin{proposition}\label{PropInvMeasure}
For every $N\geq 1$, the continuous-time Markov chain defined by any of the three models admits a unique invariant, reversible probability measure $\bmuN$ defined as follows:
\[\forall h\in\cCN,\;\; \bmuN(h) = \frac{1}{Z_{\tiN}}\exp\Big((2N)^{\frac{3}{2}}\ccA_{\tiN}(h)\Big)\]
where $Z_{\tiN}$ is a normalising constant and where $\ccA_{\tiN}(h)$ refers to the discrete weighted area under the interface
\[ \ccA_{\tiN}(h)=\begin{cases}\displaystyle\frac{1}{4N}\sum_{k=1}^{2N-1}\log\bigg(\cfrac{\pN(k_{\tiN})}{\qN(k_{\tiN})}\bigg)h(k_{\tiN})&\mbox{ in \Modun\ and \Munw},\\
\displaystyle\frac{1}{4N}\sum_{k=1}^{2N-1}\log\bigg(\cfrac{\pN(k_{\tiN})}{\qN(k_{\tiN})}\bigg)\big(\hun(k_{\tiN})-\hde(k_{\tiN})\big)&\mbox{ in \Mdeux}.\end{cases}\]
\end{proposition}
\noindent The area under the interface is a key quantity in the study of our models. Based on this observation, we investigate the scaling limits of this invariant measure when $N$ goes to infinity. We denote by $\bPMun$ the distribution on $\cCMun$ of the Brownian bridge and by $\bPMunw$ the distribution on $\cCMunw$ of the normalised Brownian excursion. Furthermore, $\bPMde$ is taken to be the distribution on $\cCMde$ of the $2$-dimensional Dyson Brownian bridge, which is also called the $2$-watermelon; this process is the unique solution of the following system of stochastic differential equations:
\begin{equation*}
\begin{cases}
d\hun(x) = -\frac{\hun(x)}{1-x}dx+ \frac{1}{\hun(x)-\hde(x)}dx + d\Bun(x),\;\;\;x\in(0,1),\\
d\hde(x) = -\frac{\hde(x)}{1-x}dx+ \frac{1}{\hde(x)-\hun(x)}dx + d\Bde(x),\;\;\;x\in(0,1),\\
\hun(0)=\hun(1)=\hde(0)=\hde(1)=0,\;\;\; \hun(x) \geq \hde(x),
\end{cases}
\end{equation*}
where $\Bun,\Bde$ are two independent standard Brownian motions. We refer to Dyson~\cite{Dyson62} and to Theorem 2.6 in Gillet~\cite{Gillet03} for details. The form taken by the invariant measure motivates an asymmetry that vanishes at rate $(2N)^{-\frac{3}{2}}$. In the following statement, $\bP$ and $\bQ$ will appear without superscript in order to alleviate notation.
\begin{theorem}\label{ThCVInvMeasure}
Let $\sigma$ be a Riemann-integrable function from $[0,1]$ into $\bbR$ and set
\begin{equation}\label{Eq:PnQn}
\pN(\cdot)+\qN(\cdot) := (2N)^2\;\;\;,\;\;\;\log\frac{\pN(\cdot)}{\qN(\cdot)} := 4\sigma(\cdot)(2N)^{-\frac{3}{2}}.
\end{equation}
Then $\bmuN \Rightarrow \bQ $ as $N\rightarrow\infty$, in the sense of weak convergence of probability measures on $\cC$, where $\bQ$ is defined via its Radon-Nikodym derivative with respect to $\bP$
$$ d\bQ(h) := \frac{\exp\big(\ccA_\sigma(h)\big)}{Z}\,d\bP(h).$$
Here $Z$ is a normalising constant and $\ccA_\sigma(h)$ is the weighted area defined as follows:
\[\ccA_{\sigma}(h)=\begin{cases}2\int_0^1 \sigma(x)h(x)dx&\mbox{ in \Modun\ and \Munw,}\\
2\int_0^1 \sigma(x)\big(\hun(x)-\hde(x)\big)dx&\mbox{ in \Mdeux.}\end{cases}\]
Moreover for every $\lambda > 0$, $\sup_{N\geq 1}\bmuN\big[e^{\lambda\left\|h\right\|_{\cC}}\big] < \infty$.
\end{theorem}
Although many results have been established on the WASEP when the asymmetry is of order $N^{-1}$ -see for instance G\"artner~\cite{Gartner88}, De Masi, Presutti and Scacciatelli~\cite{DeMasiPresuttiScacciatelli89}, Kipnis, Olla and Varadhan~\cite{KOV89} - the investigation of an asymmetry that scales like $N^{-3/2}$ seems to be new. We now turn our attention to the scaling limits of the dynamics itself.
\begin{assumption}\label{Assumption}
The asymmetry is given by (\ref{Eq:PnQn}) with $\sigma$ a $\frac{1}{2}$-H\"older function on $[0,1]$.
\end{assumption}
\noindent The H\"older condition on the map $\sigma$ is only needed in the proof of the large deviation result of Subsection \ref{SubsectionSuperExpo}. In the following statements, $\bbC([0,\infty),\cC)$ denotes the space of $\cC$-valued continuous maps endowed with the topology of uniform convergence on compact intervals of time.
\begin{theorem}\label{ThCVModel1}
Consider \Modun\ under Assumption \ref{Assumption}. Let $(\nuN)_{N\geq 1}$ be a sequence of probability measures on $\cCNMun$ that converges weakly toward a given probability measure $\nu$ on $\cCMun$ and such that there exists $\tc_{\textup{\tiny init}} > 0$ and $\upbeta_{\textup{\tiny init}} >0$ such that
\[\sup_{N\geq 1}\nuN\bigg[\sup_{x\ne y\in[0,1]}\frac{|h(x)-h(y)|}{|x-y|^{\upbeta_{\textup{\tiny init}}}}\bigg] \leq \tc_{\textup{\tiny init}}.\]
Then $\bbQN_{\nuN} \Rightarrow \bbQ_\nu$ as $N\rightarrow\infty$, in the sense of weak convergence of probability measures on the space $\bbD([0,\infty),\cCMun)$. Here $\bbQ_\nu$ is the distribution on $\bbC([0,\infty),\cCMun)$ under which $\rmh_0$ has law $\nu$ and $\rmh$ is the solution of the stochastic heat equation:
\begin{eqnarray}\label{EqSHE}\textsc{SHE}\begin{cases}
\displaystyle\partial_t \rmh_t(x) = \frac{1}{2}\partial^2_x \rmh_t(x) + \sigma(x) + \dot{W}(t,x),\\
\rmh_t(0)=\rmh_t(1)=0.
\end{cases}
\end{eqnarray}
Here $\dot{W}$ is a space-time white noise.
\end{theorem}
\noindent Recall the definition of the space $\bbM$ from above. Recall also that $\bmuN$ stands for the invariant probability measure.
\begin{theorem}\label{ThCVModel1w}
Consider \Munw\ under Assumption \ref{Assumption}. Then $\bbQN_{\bmuN} \Rightarrow \bbQ$ as $N\rightarrow\infty$, in the sense of weak convergence of probability measures on the product space $\bbD([0,\infty),\cCMunw)\times\bbM$ endowed with the product topology. Here $\bbQ$ is the distribution on $\bbC([0,\infty),\cCMunw)\times\bbM$ under which $\rmh_0$ has law $\bQMunw$ and $(\rmh,\zeta)$ is the solution of the Nualart-Pardoux reflected stochastic heat equation~\cite{NualartPardoux92}:
\begin{equation}\label{EqRSHE}\textsc{RSHE}\begin{cases}\begin{split}
&\displaystyle\partial_t \rmh_t(x) = \frac{1}{2}\partial^2_x \rmh_t(x) + \sigma(x) + \zeta(dt,dx) + \dot{W}(t,x),\\
&\displaystyle \rmh_t(x) \geq 0\;\;\;,\;\;\;\rmh_t(0)=\rmh_t(1)=0,\\
&\displaystyle\int_{[0,\infty)\times(0,1)}\!\!\!\!\rmh_t(x)\zeta(dt,dx)=0.\end{split}\end{cases}
\end{equation}
Here $\dot{W}$ is a space-time white noise.
\end{theorem}
\noindent Finally, we consider the most elaborate model.
\begin{theorem}\label{ThCVModel2}
Consider \Mdeux\ under Assumption \ref{Assumption}. Then the sequence $\bbQN_{\bmuN}$ is tight for the topology of weak convergence of probability measures on $\bbD([0,\infty),\cCMunw)\times\bbM\times\bbM$. Furthermore, if we let $\bbQ$ be the limit of a converging subsequence, then $\bbQ$ is supported by $\bbC([0,\infty),\cCMunw)\times\bbM\times\bbM$ and under $\bbQ$, $\rmh_0$ has law $\bQMde$ and $(\rmh,\zetaun,\zetade)$ satisfies:
\begin{equation}\label{Eq:SHEs}\textsc{Pair \scriptsize{of}\normalsize\ RSHEs}\begin{cases}\begin{split}
&\displaystyle\partial_t \rmhun_t(x) = \frac{1}{2}\partial^2_x \rmhun_t(x) + \sigma(x) + \zetaun(dt,dx) + \dot{W}^{\mbox{\tiny $(1)$}}(t,x),\\
&\displaystyle\partial_t \rmhde_t(x) = \frac{1}{2}\partial^2_x \rmhde_t(x) - \sigma(x) - \zetade(dt,dx) + \dot{W}^{\mbox{\tiny $(2)$}}(t,x),\\
&\displaystyle \rmhun_t(x) \geq \rmhde_t(x)\;\;\;,\;\;\;\rmhun_t(0)=\rmhun_t(1)=\rmhde_t(0)=\rmhde_t(1)=0,\\
&\displaystyle\int_{[0,\infty)\times(0,1)}\!\!\!\!\!\big(\rmhun_t(x)-\rmhde_t(x)\big)\big(\zetaun(dt,dx)+\zetade(dt,dx)\big)=0.\end{split}\end{cases}
\end{equation}
Here $\dot{W}^{\mbox{\tiny $(1)$}}$ and $\dot{W}^{\mbox{\tiny $(2)$}}$ are two independent space-time white noises.
\end{theorem}
\noindent Before proceeding to the proofs, we relate our results to the existing literature. The proof of Theorem \ref{ThCVModel1} is inspired by the convergence techniques used by Bertini and Giacomin~\cite{BertiniGiacomin97} in their celebrated paper on the KPZ equation. It seems that these techniques no longer work in the settings with reflection. Indeed the tightness of the random measure(s) that encodes the time spent at $0$ by the interface(s) needs specific work. Consequently the proofs of Theorems \ref{ThCVModel1w} and \ref{ThCVModel2} use different tools and depend strongly on the process being in the stationary regime. Funaki and Olla~\cite{FunaOll01} proved that the \textsc{RSHE} is the scaling limit of a system of oscillators which is similar to our \Munw. However in their case the oscillators take continuous values in $\bbR$ while our model is discrete; they mentioned in their paper that a discrete setting is probably more difficult to tackle. Also, the discreteness of the setting prevents us from applying the general method developped by Ambrosio, Savar\'e and Zambotti~\cite{AmbrosioSavareZambotti09}: indeed, our stationary measure fails to be log-concave. Let us also comment on the reason why we start from the invariant measure in the two more elaborate models. Actually, we first show tightness in a space of distributions and then, using estimates on the space regularity of the interface under the invariant measure, we obtain tightness in a space of continuous functions by interpolation arguments. Therefore, the initial condition being invariant appears as a technical assumption.\\
Both Nualart and Pardoux~\cite{NualartPardoux92} and Funaki and Olla~\cite{FunaOll01} used the penalisation method to deal with the reflecting measure. In the present paper, we instead show the convergence of $\zeta$ - or $\zetaun,\zetade$ - by martingale techniques; this approach seems to be new.\\
Let us also mention that the \textsc{RSHE} has been studied quite extensively in the recent years. In particular, Zambotti~\cite{Zambotti01} showed that the measure $\bQMunw$ is invariant for this stochastic PDE while Dalang, Mueller and Zambotti~\cite{DalMueZam06} obtained the following beautiful result: almost surely at any time $t>0$ the number of points $x\in(0,1)$ at which the interface vanishes is at most $4$. We also refer to Xu and Zhang~\cite{XuZhang09} for related equations.\\
Finally, let us mention that it would be interesting to investigate similar discrete models whose invariant measure converges to some distribution related to the Brownian motion (for instance, the reflected Brownian motion). The forms taken by the corresponding stochastic PDEs do not seem to be easy to guess.

\begin{remark}
We have not been able to decide whether $\zetaun=\zetade$ in the limit for \Mdeux, even though we believe that this equality holds. Let us point out that Theorem \ref{ThSuperExpo} does not provide such an equality since the functional $\Phi$ involved in the expression of the potential $V$ needs to depend on a finite number of sites of the lattice while the quantity $\zetaun-\zetade$ really depends on the whole interface. The equality $\zetaun=\zetade$ would ensure uniqueness of the limit in Theorem \ref{ThCVModel2} since, then, \textsc{Pair \scriptsize{of}\normalsize\ RSHEs} would just be a linear combination of \textsc{SHE} and \textsc{RSHE}.
\end{remark}

\vspace{-8pt}
\paragraph{Organisation of the paper.} In Section 2, we prove the results related to the invariant measure and we state a large deviation result on the local behaviour of the interfaces which will be necessary to identify the limits. The proof of this result is postponed to Appendix \ref{AppendixSuperExpo}. In Section 3, we present our general approach to proving tightness in any of the three settings. Then we provide the arguments when the processes start from the stationary measure, while the proof for \Modun\ starting from a more general initial condition, is postponed to Appendix \ref{SectionProofTightnessMun}. In Section 4, we identify the limit of the sequence $\bbQN$ and, therefore, complete the proof of Theorems \ref{ThCVModel1}, \ref{ThCVModel1w} and \ref{ThCVModel2}.

\section{The invariant measure}
\subsection{Proof of Proposition \ref{PropInvMeasure}}
We provide a proof that works for the three models; therefore $\cCN$ is any of the three state-spaces. Fix $N\geq 1$. Consider two configurations $h,h' \in \cCN$. Denote by $\lambda(h,h')$ the rate at which the process (in any of the three models) jumps from $h$ to $h'$. We have to prove that
\begin{equation}\label{DetailedBalance}
\exp\Big((2N)^{\frac{3}{2}}\ccA_{\tiN}(h)\Big)\lambda(h,h') = \exp\Big((2N)^{\frac{3}{2}}\ccA_{\tiN}(h')\Big)\lambda(h',h).
\end{equation}
By definition of the dynamics, $\lambda(h,h')\ne 0$ if and only if $h'$ is obtained from $h$ by flipping a local extremum into its counterpart without violating the non-crossing rules if any. By the symmetry of Equation (\ref{DetailedBalance}), we can assume that $\lambda(h,h')=\pN(k_{\tiN})$ for a given $k\in\lbr 1,2N-1\rbr$ so that $\lambda(h',h)=\qN(k_{\tiN})$. The key observation is that any jump that occurs at rate $\pN(\cdot)$ (respectively $\qN(\cdot)$) makes the area increase (respectively decrease). More precisely, we have $(2N)^{\frac{3}{2}}\ccA(h') = (2N)^{\frac{3}{2}}\ccA(h)+\log\Big(\frac{\pN(k_{\tiN})}{\qN(k_{\tiN})}\Big)$. Consequently (\ref{DetailedBalance}) follows.\cqfd

\subsection{Weak asymmetry and the area}
The expression for the invariant measure exhibits an interplay between the area and the ratio of the jump rates. This suggests that we should choose a weak asymmetry that scales consistently with the area. Let us first study the symmetric case $\pN(\cdot)\!=\!\qN(\cdot)\!=\!(2N)^2/2$. We denote by $\piNMun$, $\piNMunw$ and $\piNMde$ the corresponding invariant measures: from Proposition \ref{PropInvMeasure}, we deduce that they are the uniform measures on $\cCNMun$, $\cCNMunw$ and $\cCNMde$ respectively. Recall the definition of the probability measures $\bPMun$, $\bPMunw$ and $\bPMde$ introduced before the statement of Theorem \ref{ThCVInvMeasure}. Recall also that we drop the superscript associated with the model whenever a result can be stated indifferently for the three models.
\begin{lemma}\label{LemmaCVInvSymm}
As $N\rightarrow\infty$, $\piN$ converges weakly to the measure $\bP$ on $\cC$. Moreover for any $\lambda > 0$ we have $\sup_{N\geq 1}\piN\big[e^{\lambda\left\|h\right\|_{\cC}}\big] < \infty$.
\end{lemma}
\begin{proof}
The convergence of $\piNMun$ (respectively $\piNMunw$) towards $\bPMun$ (respectively $\bPMunw$) is a classical result, see~\cite{Kaigh76,Liggett68}. The uniform bounds for the exponential moments were obtained by Khorunzhiy and Marckert in~\cite{KhorunzhiyMarckert09}. Let us consider \Mdeux. Gillet proved the convergence result in~\cite{Gillet03}. Let us show the uniform bound for the exponential moments. The underlying idea of our proof is to study the paths $s:=\frac{\hun+\hde}{2}$ and $d:=\frac{\hun-\hde}{2}$. First, observe that on any interval $[(k-1)_{\tiN},k_{\tiN}]$ the pair $(\hun,\hde)$ has four possible increments: $\hun$ and $\hde$ both increase - we denote this event by $(\uparrow\uparrow)_k$; $\hun$ and $\hde$ both decrease $(\downarrow\downarrow)_k$; $\hun$ increases and $\hde$ decreases $(\uparrow\downarrow)_k$; $\hun$ decreases and $\hde$ increases $(\downarrow\uparrow)_k$. Fix a pair $(\hun,\hde)\in\cCNMde$. The non-crossing condition $\hun \geq \hde$ imposes that
\[ \sum_{i=1}^k\tun_{(\uparrow\downarrow)_i} \geq \sum_{i=1}^k\tun_{(\downarrow\uparrow)_i}\;\;\;,\;\;\;\forall k\in\lbr 1,2N\rbr.\]
Furthermore $\hun(1)=\hde(1)=0$ yields the existence of an integer $n\in\lbr 0,N\rbr$ such that:
\begin{equation*}
\sum_{i=1}^{2N}\tun_{(\uparrow\uparrow)_i} = \sum_{i=1}^{2N}\tun_{(\downarrow\downarrow)_i} = n\;\;\;\;,\;\;\;\;\sum_{i=1}^{2N}\tun_{(\uparrow\downarrow)_i} = \sum_{i=1}^{2N}\tun_{(\downarrow\uparrow)_i} = N-n.
\end{equation*}
We will denote by $\cCNnMde$ the subset of $\cCNMde$ restricted to the paths that fulfil these conditions for a given value $n$. For a given $(\hun,\hde)\in\cCNnMde$, let us denote by $\imath$ the subset of $\lbr 1,2N\rbr$ consisting of the indices of the increments of the form $(\uparrow\uparrow)$ or $(\downarrow\downarrow)$ in $(\hun,\hde)$. Plainly $\imath$ belongs to the collection $\cI(n)$ of subsets of $\lbr 1,2N\rbr$ with $2n$ elements; we will denote by $\imath(j),j\in \lbr1,2n\rbr$ the elements of $\imath$ in increasing order. Then we define the path $\tilde{s}$ as the following element of $\cCnMun$:
\[ \tilde{s}\Big(\frac{k}{2n}\Big) := \frac{1}{\sqrt{2n}} \sum_{j=1}^k\Big(\tun_{(\uparrow\uparrow)_{\imath(j)}}
-\tun_{(\downarrow\downarrow)_{\imath(j)}}\Big)\;\;\;,\;\;\;\forall k \in \lbr1,2n\rbr.\]
In words, $\tilde{s}$ makes $+\frac{1}{\sqrt{2n}}$ at any step $(\uparrow\uparrow)$, $-\frac{1}{\sqrt{2n}}$ at any step $(\downarrow\downarrow)$ and does not evolve at any other step of $(\hun,\hde)$. Similarly we define $\tilde{d}$ as the following element of $\cCNnMunw$:
\[ \tilde{d}\Big(\frac{k}{2n}\Big) := \frac{1}{\sqrt{2(N-n)}} \sum_{j=1}^k\Big(\tun_{(\uparrow\downarrow)_{\lbr 1,2N\rbr\backslash\imath(j)}}
-\tun_{(\downarrow\uparrow)_{\lbr 1,2N\rbr\backslash\imath(j)}}\Big)\;\;\;,\;\;\;\forall k \in \lbr1,2(N-n)\rbr,\]
where $\lbr 1,2N\rbr\backslash\imath(j)$ stands for the $j$-th element, in the increasing order, in the set $\lbr 1,2N\rbr\backslash\imath$. The map:
\begin{align*}
\cCNnMde &\rightarrow\cCnMun\times\cCNnMunw\times\cI(n)\\
(\hun,\hde)&\mapsto(\tilde{s},\tilde{d},\imath)
\end{align*}
is a bijection. Since $\piNMde$ is the uniform measure on $\cCNMde$, we deduce that $\piNMde(\,\cdot\, |\, \cCNnMde)$ is the uniform measure on $\cCNnMde$. Consequently, under $\piNMde(\,\cdot\, |\, \cCNnMde)$, the pair $(\tilde{s},\tilde{d})$ is distributed according to $\pinMun\otimes\piNnMunw$. The paths $s:=\frac{\hun+\hde}{2}$ and $d:=\frac{\hun-\hde}{2}$ are obtained from $\tilde{s}$ and $\tilde{d}$ by inserting constant steps and rescaling suitably so that
\[ \sup_{[0,1]} |s| = \frac{\sqrt{2n}}{\sqrt{2N}}\sup_{[0,1]} |\tilde{s}|\;\;\;,\;\;\;\sup_{[0,1]} |d| = \frac{\sqrt{2(N-n)}}{\sqrt{2N}}\sup_{[0,1]} |\tilde{d}|,\]
and, we deduce that
\[ \sup_{N\geq 1}\piNMde\big[e^{\lambda\sup|s|}\big] \!\leq\! \sup_{n\geq 1}\pinMun\big[e^{\lambda\sup|h|}\big] < \infty\,,\,\sup_{N\geq 1}\piNMde\big[e^{\lambda\sup|d|}\big] \!\leq\! \sup_{n\geq 1}\pinMunw\big[e^{\lambda\sup|h|}\big] < \infty.\]
Since $\hun=s+d$ and $\hde=s-d$, the Cauchy-Schwarz inequality yields for every $\lambda > 0$ and $i\in\{1,2\}$
\[ \piNMde\big[e^{\lambda\sup|\hi|}\big] \leq \sqrt{\piNMde\big[e^{2\lambda\sup|s|}\big]\piNMde\big[e^{2\lambda\sup|d|}\big]}.\]
Since $\left\|h\right\|_{\cC} \leq \sup_{[0,1]}|\hun| + \sup_{[0,1]}|\hde|$ another application of the Cauchy-Schwarz inequality completes the proof.\cqfd
\end{proof}
Before we proceed to the proof of Theorem \ref{ThCVInvMeasure}, let us state without proof a well-known result that we will use on several occasions.
\begin{lemma}\label{Lemma:CVUnifInteg}
Let $X_n,n\geq 1$ be a sequence of random variables that converges in distribution to a random variable $X$. Assume that there exists $p>1$ such that the expectation of $|X_n|^p$ is uniformly bounded in $n\geq 1$, then the first moment of $X_n$ converges to the first moment of $X$.
\end{lemma}
\noindent\textit{Proof of Theorem \ref{ThCVInvMeasure}}. Fix a Riemann-integrable function $\sigma$ and take $\log(\frac{\pN(\cdot)}{\qN(\cdot)}) = 4\sigma(\cdot)(2N)^{-\frac{3}{2}}$ so that
\[ (2N)^{\frac{3}{2}}\ccA_{\tiN}(h)=\begin{cases}\displaystyle\frac{2}{2N}\sum_{k=1}^{2N-1}\sigma(k_{\tiN})h(k_{\tiN})&\mbox{ in \Modun\ and \Munw,}\\
\displaystyle\frac{2}{2N}\sum_{k=1}^{2N-1}\sigma(k_{\tiN})\big(\hun(k_{\tiN})-\hde(k_{\tiN})\big)&\mbox{ in \Mdeux.}\end{cases}\]
We drop the superscript associated with the models since our proof works verbatim for the three models. From now on, we work on $\cC$ and we see $\piN$ and $\bmuN$ as measures on this space. We want to prove that for any bounded continuous map $F$ from $\cC$ to $\bbR$ we have
\begin{equation}\label{EqWeakCV}
\bmuN[F] \underset{N\rightarrow\infty}{\longrightarrow} \bQ[F].
\end{equation}
We observe that
\[ \bmuN[F] = \frac{\piN[F(h) \exp((2N)^{\frac{3}{2}}\ccA_{\tiN}(h))]}{\piN[\exp((2N)^{\frac{3}{2}}\ccA_{\tiN}(h))]}\;\;,\;\;\bQ[F] = \frac{\bP[F(h) \exp(\ccA_{\sigma}(h))]}{\bP[\exp(\ccA_{\sigma}(h))]}.\]
To prove (\ref{EqWeakCV}), we show that the numerator (resp.~denominator) of the expression on the left converges to the numerator (resp.~denominator) of the expression on the right. By continuity of $F$ and $\ccA_{\sigma}$, the pushforward of $\piN$ through $h\mapsto F(h)\exp(\ccA_{\sigma}(h))$ converges weakly to the pushforward of $\bP$ through the same map. Using the boundedness of $\sigma$ and the uniform exponential bound on $\left\|h\right\|_{\cC}$ obtained in Lemma \ref{LemmaCVInvSymm}, we deduce that 
\[\sup_{N\geq 1} \piN\big[F^2(h)\exp(2\ccA_{\sigma}(h))\big] < \infty.\]
Consequently, Lemma \ref{Lemma:CVUnifInteg} ensures that $\piN[F(h) \exp(\ccA_{\sigma}(h))]$ converges to $\bP[F(h) \exp(\ccA_{\sigma}(h))]$. It remains to show that
\[\piN\big[F(h) \big(\exp((2N)^{\frac{3}{2}}\ccA_{\tiN}(h))-\exp(\ccA_{\sigma}(h))\big)\big] \underset{N\rightarrow\infty}{\longrightarrow} 0.\]
The same argument as above shows that the second moment of this random variable is uniformly bounded in $N\geq 1$. Furthermore the Riemann-integrability of $\sigma$ and the convergence of $\piN$ towards $\bP$ ensure the convergence in probability of this random variable to $0$ so that the result follows from Lemma \ref{Lemma:CVUnifInteg}.\cqfd\vspace{4pt}\\
In the following proposition, we give a description of $\bQMun$.
\begin{proposition}
Consider \Modun\ and take $\sigma$ Riemann-integrable. Under $\bQMun$, the process $x\mapsto h(x)-(1-x)\int_0^x\frac{2}{(1-y)^2}\int_y^1\sigma(u)(1-u)du\,dy$ is a Brownian bridge on $[0,1]$. In the particular case where $\sigma$ is constant, we obtain that $x\mapsto h(x)- \sigma x(1-x)$ is a Brownian bridge.
\end{proposition}
\begin{proof}
We drop the superscript \Modun\ since there is no possible confusion here. We endow $\cC$ with the filtration $\cF_x,x\in[0,1]$ of the canonical process $x\mapsto h(x)$, and we introduce the $\bP$-martingale $D(x):=\frac{d\bQ}{d\bP}|_{\cF_x}, x\in[0,1]$. Since $\ccA_\sigma(h)=2\int_0^1 \sigma(x)h(x)dx$ we obtain that for all $x\in[0,1]$
$$ D(x) = \exp\Big(2\int_0^x \sigma(y)h(y)dy\Big)\, \frac{\bP\Big[e^{2\int_x^1 \sigma(y)h(y)dy}\, \big|\, \cF_x\Big]}{\bP\Big[e^{2\int_0^1 \sigma(y)h(y)dy}\Big]}. $$
Recall that, under $\bP$, $h$ is a Brownian bridge so that conditionally given $\cF_x$, the process $(h(y)-h(x)\frac{1-y}{1-x},y\in[x,1])$ is a Brownian bridge independent of $\cF_x$. We obtain:
$$ \bP\Big[e^{2\int_x^1 \sigma(y) h(y)dy} \big| \cF_x\Big] = e^{2h(x)\int_x^1\sigma(y)\frac{1-y}{1-x}dy}\bP\Big[e^{2\int_x^1 \sigma(y) \big(h(y)-h(x)\frac{1-y}{1-x}\big)dy}\Big]. $$
Moreover, there exists a $\bP$-Brownian motion $W$ such that for every $x\in[0,1]$, $h(x) = W(x) - \int_0^x h(y)(1-y)^{-1}dy$. A simple application of It\^o's formula to the process $x\mapsto h(x)\int_0^x\sigma(y)\frac{1-y}{1-x}dy$ ensures that
\[ \int_0^x \sigma(y)h(y)dy = h(x)\int_0^x\sigma(y)\frac{1-y}{1-x}dy-\int_0^x\int_0^y\sigma(u)\frac{1-u}{1-y}du\, dW_y.\]
Consequently, we have:
\[ \left\llangle \log D,W \right\rrangle(x)= 2\int_0^x\frac{1}{1-y}\int_y^1\sigma(u)(1-u)du\,dy.\]
Girsanov's theorem (see for instance Revuz and Yor~\cite{RevuzYor} Theorem VIII.1.7) ensures that under $\bQ$ the process
$$x\mapsto\tilde{W}(x):= W(x)-\left\llangle \log D,W \right\rrangle(x)$$
is a continuous martingale with the same bracket as $W$, and so, it is a $\bQ$ Brownian motion. Accordingly for every $x\in[0,1]$, $ h(x) = \tilde{W}(x) - \int_0^x h(y)(1-y)^{-1}dy +2\int_0^x\frac{1}{1-y}\int_y^1\sigma(u)(1-u)du\,dy $. A simple calculation ends the proof.\cqfd
\end{proof}
To end this subsection, we state a technical result useful for the proof of the tightness. For any $\eta \geq 0$ and $r\geq 1$, we introduce the Sobolev-Slobodeckij space:
\begin{equation}\label{Eq:SobolevSlobodeckij}
\cW^{\eta,r} := \Big\{f\in\rL^r([0,1]),\,\int_{[0,1]}|f(x)|^rdx + \int_{[0,1]^2}\frac{|f(x)-f(y)|^r}{|x-y|^{\eta r +1}}dx\,dy =: \left\|f\right\|_{\cW^{\eta,r}}^r < \infty\Big\}.
\end{equation}
\begin{lemma}\label{Lemma:SlobodeckijWalk}
Fix $r\geq 1$, $\eta \in (0,\frac{1}{2})$ and $p \geq 1$. In \Modun\ and \Munw, we have
\[\sup_{N\geq 1}\piN\Big[\left\|h\right\|_{\cW^{\eta,r}}^p\Big] <\infty.\]
In \Mdeux, the same holds true for both $\hun,\hde$.
\end{lemma}
\noindent This result can be seen as a uniform bound on the $\eta$-H\"older regularity of $h$ under $\piN$.
\begin{proof}
We start with \Modun. Fix $\epsilon \in (0,1/3)$ and set $D_\epsilon:=\{0\leq x,y \leq 1: |y-x| \leq \epsilon\}$. For any $h\in\cCMun$, we have
\begin{eqnarray}\label{Eq:IneqSloboWalk}
\int_{[0,1]}|h(x)|^rdx &\leq& \left\|h\right\|_{\cCMun}^r,\\\nonumber
\int_{[0,1]^2}\frac{|h(x)-h(y)|^r}{|x-y|^{\eta r +1}}dx\,dy &\leq& 2\left\|h\right\|_{\cCMun}^r\int_{[0,1]^2\backslash D_\epsilon}\frac{dx\,dy}{|x-y|^{\eta r +1}}+\int_{D_\epsilon}\frac{|h(x)-h(y)|^r}{|x-y|^{\eta r +1}}dx\,dy.
\end{eqnarray}
Since the exponential moments of the supremum norm of $h$ under $\piNMun$ are uniformly bounded in $N\geq 1$ thanks to Lemma \ref{LemmaCVInvSymm}, we only need to bound the moments of the second term on the r.h.s.~of the second line of (\ref{Eq:IneqSloboWalk}). Fix $\delta \in (0,1)$ such that $\eta+\frac{\delta}{r}<\frac{1}{2}$. Using Jensen's inequality in the first line and the existence of $c>0$ such that for all $x\in\bbR$, $|x|^{pr}\leq c e^{|x|}$ in the second line, we obtain
\begin{eqnarray*}
\Big(\int_{D_\epsilon}\frac{|h(x)-h(y)|^r}{|x-y|^{\eta r +\delta}}\frac{dx\,dy}{|x-y|^{1-\delta}}\Big)^p &\leq& C_{\epsilon,\delta}^{p-1}\int_{D_\epsilon}\Big(\frac{|h(x)-h(y)|^r}{|x-y|^{\eta r +\delta}}\Big)^p \frac{dx\,dy}{|x-y|^{1-\delta}}\\
&\leq& c\,C_{\epsilon,\delta}^{p-1}\int_{D_\epsilon}\exp\Big(\frac{|h(x)-h(y)|}{|x-y|^{\eta +\frac{\delta}{r}}}\Big)\frac{dx\,dy}{|x-y|^{1-\delta}}
\end{eqnarray*}
where $C_{\epsilon,\delta}:=\int_{D_\epsilon}\frac{dx\,dy}{|x-y|^{1-\delta}}$. Therefore we have
\[\piNMun\bigg[\Big(\int_{D_\epsilon}\frac{|h(x)-h(y)|^r}{|x-y|^{\eta r +1}}dx\,dy\Big)^p\bigg] \leq c\,C_{\epsilon,\delta}^{p-1}\int_{D_\epsilon}\piNMun\Big[\exp\Big(\frac{|h(x)-h(y)|}{|x-y|^{\eta  +\frac{\delta}{r}}}\Big)\Big] \frac{dx\,dy}{|x-y|^{1-\delta}}.\]
We need to bound the integrand in the right side. We denote by $\gamma_{\tiN}$ the probability measure induced on the space of continuous functions on $[0,1]$ by a simple random walk starting from $0$ and making $2N$ steps (and rescaled diffusively as usual). Notice that this random walk is not conditioned to come back to $0$ nor to stay non-negative. By the independence of the increments of the simple random walk and since $\eta+\frac{\delta}{r}<\frac{1}{2}$, one obtains easily:
\begin{equation}\label{Eq:Gamma}
\sup_{N\geq 1}\sup_{x\ne y}\gamma_{\tiN}\Big[\exp\Big(\frac{|h(x)-h(y)|}{|x-y|^{\eta +\frac{\delta}{r}}}\Big)\Big]<\infty.
\end{equation}
Now observe that for every $N\geq 1$, every $k\in\lbr 1,2N\rbr$ and every $h\in\cCNMun$, we have
\[ {\frac{d\piNMun}{d\gamma_{\tiN}}}\,_{|\cF_{\frac{k}{2N}}}(h) = \frac{\piNMun\Big(\big\{h':\forall i\in\lbr 1,k\rbr, h'(i_{\tiN})=h(i_{\tiN})\big\}\Big)}{\gamma_{\tiN}\Big(\big\{h':\forall i\in\lbr 1,k\rbr, h'(i_{\tiN})=h(i_{\tiN})\big\}\Big)}=2^{k}\binom{2N-k}{N-\frac{k +\sqrt{2N}h(k_{\tiN})}{2}}\binom{2N}{N}^{-1},\]
where $\cF_{x}$ is the sigma-algebra generated by $(h(y),y\in[0,x])$. The maximum of this quantity is reached when $|h(k_{\tiN})|$ equals $0$ or $1$ according as $k$ is even or odd. Stirling's formula then yields
\begin{equation}\label{Eq:BoundRadonNikodymWalks}
\sup_{N\geq 1}\sup_{h\in\cCNMun}{\frac{d\piNMun}{d\gamma_{\tiN}}}\,_{|\cF_{\frac{2}{3}}}(h) < \infty.
\end{equation}
For any $(x,y)\in D_\epsilon$, at least one of these two assertions is satisfied: $x,y\in[0,\frac{2}{3}]$ or $x,y\in[\frac{1}{3},1]$. Since $x\mapsto h(1-x)$ and $x\mapsto h(x)$ have the same distribution under $\piNMun$, and using (\ref{Eq:Gamma}) and (\ref{Eq:BoundRadonNikodymWalks}) we obtain
\[\sup_{N\geq 1}\sup_{(x,y)\in D_\epsilon}\piNMun\Big[\exp\Big(\frac{|h(x)-h(y)|}{|x-y|^{\eta +\frac{\delta}{r}}}\Big)\Big]<\infty.\]
Consequently the lemma is proved under $\piNMun$. Using the Vervaat transform~\cite{Vervaat79} that maps a bridge onto an excursion, this result can easily be extended to $\piNMunw$. Finally in \Mdeux\ the result follows by using the decomposition $s:=\frac{\hun+\hde}{2}$ and $d:=\frac{\hun-\hde}{2}$ introduced in the proof of Lemma \ref{LemmaCVInvSymm}.\cqfd
\end{proof}

\subsection{A large deviation result}\label{SubsectionSuperExpo}
For the proof of Theorems \ref{ThCVModel1}, \ref{ThCVModel1w} and \ref{ThCVModel2}, we will need a uniform estimate on the probability that the interface locally looks like an unconditioned simple random walk. This estimate is originally due to Kipnis, Olla and Varadhan~\cite{KOV89} (see also~\cite{VaradhanLectureNotes}) in the case where the lattice is the torus $\bbZ/\lbr 1,2N\rbr$. In order to state the estimate, we need to introduce some notation. We set $\cON:=\{0,1\}^{2N}$. For every $j\in \bbZ$, we denote by $\tau_j$ the shift by $j$ modulo $2N$ on $\cON$ which is defined as follows. For all $\eta\in\cON$ and all $i\in\lbr 1,2N\rbr$, $\tau_j\eta(i)=\eta(i+j)$ where $i+j$ is taken modulo $2N$. Consider an integer $l\geq 1$ and a map $\Phi:\{0,1\}^l\rightarrow\bbR$. Whenever $2N\geq l$ and for every $\eta\in\cON$, we extend $\Phi$ into a map from $\cON$ into $\bbR$ by setting $\Phi(\eta):=\Phi(\eta(1),\ldots,\eta(l))$. Consequently $\Phi$ is a map from $\cON$ into $\bbR$ that only depends on a fixed number of sites. We also introduce the map $\tilde{\Phi}$ as follows
\[ \tilde{\Phi}(a) = \sum_{\eta\in\cON}\Phi(\eta)a^{\#\{i:\eta(i)=1\}}(1-a)^{\#\{i:\eta(i)=0\}},\;\;\;a \in [0,1].\]
This can be viewed as the expectation of $\Phi$ under the product of $2N$ Bernoulli measures with parameter $a$. Recall from the introduction the definition of the space of particle configurations $\cENMun$ associated to $\cCNMun$. Similarly, we define $\cENMunw$ as the subset of $\cON$ whose elements $\eta$ have $N$ occurrences of $1$ and satisfy the following wall condition:
\[\forall k\in\lbr 1,2N\rbr,\;\;\; \sum_{i=1}^k\eta(i) \geq \frac{k}{2}.\]
Finally we set $\cENMde$ as the set of pairs $\etaun,\etade$ which both belong to $\cENMun$ and satisfy the following non-crossing condition:
\[\forall k\in\lbr 1,2N\rbr,\;\;\; \sum_{i=1}^k(\etaun(i)-\etade(i)) \geq 0.\]
Then we set for every $k\in\lbr 1,2N\!-\!1\rbr$ and every element $\eta$ of $\cENMun$ or $\cENMunw$,
\[ \nabla \eta(k) := \eta(k+1)-\eta(k) \in \{-1;0;1\}.\]
In \Mdeux, we define the same notation for $\etaun$ and $\etade$. Observe that $\nabla\eta$ is the counterpart of $\Delta h$. In any of the three models, the correspondence between $\cCN$ and $\cEN$ allows us to define a process $\upeta:=(\upeta_t,t\geq 0) \in \bbD([0,+\infty),\cEN)$ under the measure $\bbQN_{\nuN}$, where $\nuN$ is a distribution on $\cCN$.
\begin{theorem}\label{ThSuperExpo}(Large deviation)
For any initial distribution $\nuN$ and for every $\delta > 0,t >0$,
\[ \varlimsup_{\epsilon\downarrow 0}\varlimsup_{N\rightarrow\infty}\frac{1}{N}\log\bbQN_{\nuN}\Big(\frac{1}{N}\int_0^t\VNe(\upeta_s)ds > \delta\Big) = -\infty\]
where in \Modun\ and \Munw\ for all $\epsilon > 0$
\[ \VNe(\eta) = \sum_{i=1}^{2N}\bigg|\frac{1}{2\epsilon N + 1}\!\!\sum_{j:|i-j|\leq \epsilon N}\!\!\!\!\Phi(\tau_j\eta)-\tilde{\Phi}\Big(\frac{1}{2\epsilon N + 1}\!\sum_{j:|i-j|\leq \epsilon N}\!\!\eta(j)\Big)\bigg|\]
and in \Mdeux, $\VNe(\eta)$ is taken to be the sum of the same quantities for $\etaun$ and $\etade$.
\end{theorem}
\noindent In the statement of the theorem, all the integers are taken modulo $2N$.
\begin{remark}
It may seem surprising that we need such a super-exponential estimate, rather than just the convergence to $0$ of the probability of the event above. Actually the result is first established under the invariant measure, and then extended to the general case via the Radon-Nikodym derivative w.r.t.~the stationary case. Since this derivative is bounded by a term of order $e^{cN}$, a super-exponential decay allows us to compensate the derivative.
\end{remark}
The structure of the proof is very similar to that of~\cite{KOV89} but some key arguments need to be significantly modified since our state-space is no longer translation invariant and since we have added interaction with a wall in \Munw\ (resp.~between two interfaces in \Mdeux). Below, we describe the method of proof for the three models simultaneously. We denote the generator of our process by $\LN$. For instance, in \Munw\ this is the operator acting on maps $f$ from $\cENMunw$ into $\bbR$ as follows:
\[ \LNMunw f(\eta) := \sum_{k=1}^{2N-1}\big(f(\eta^{k,k+1})-f(\eta)\big)\big(\pN(k_{\tiN})\tun_{\{\nabla\eta(k)=1\}}-\qN(k_{\tiN})\tun_{\{\nabla\eta(k)=-1;\eta^{k,k+1}\in\cENMunw\}}\big),\]
where $\eta^{k,k+1}$ is obtained from $\eta$ by exchanging the values of $\eta(k)$ and $\eta(k+1)$. The condition $\eta^{k,k+1}\in\cENMunw$ in the formula expresses the wall condition.\\
We can associate to $\VNe$ a diagonal operator acting on maps $f$ from $\cEN$ into $\bbR$ as follows:
$$ \forall\eta\in\cEN,\;\; \VNe f(\eta) := \VNe(\eta) f(\eta).$$
Recall that $\bmuN$ is the reversible measure associated with the dynamics. We consider the Hilbert space $\rL^2(\cE_{\tiN},\bmuN)$ of square-integrable maps on $\cE_{\tiN}$ w.r.t. the measure $\bmuN$. Fix $a\in\bbR$. The operator $\LN+a\VNe$ is self-adjoint in $\rL^2(\cE_{\tiN},\bmuN)$; we denote by $\lambdaNe(a)$ its largest eigenvalue. The Feynman-Kac formula (see for instance Appendix 1 - Lemma 7.2 in the book of Kipnis and Landim~\cite{KipnisLandimBook}) ensures that for all $t\geq 0$
$$ \bbQN_{\bmuN}\bigg[\exp\Big(a\int_0^t\VNe(\upeta_s)ds\Big)\bigg] \leq \exp\big(t\lambdaNe(a)\big).$$
For $a>0$, the Markov inequality implies
$$ \frac{1}{N}\log \bbQN_{\bmuN}\Big(\frac{1}{N}\int_0^t\VNe(\upeta_s)ds \geq \delta\Big) \leq t\frac{\lambdaNe(a)}{N}-\delta a.$$
Consequently it suffices to show that for all $a>0$, $\varlimsup_{\epsilon\downarrow 0}\varlimsup_{N\rightarrow\infty}N^{-1}\lambdaNe(a)=0$ in order to prove the theorem under the stationary measure. Let us denote by $\DN$ the Dirichlet form associated with $\LN$. For instance, in \Munw\ this is the operator acting on maps $f\geq 0$ as follows:
\[ \DNMunw(f) \!:=\!\!\!\!\!\!\sum_{\eta\in\cENMunw}\!\!\!\!\frac{\bmuN(\eta)}{2}\sum_{k=1}^{2N-1}\!\!\Big(\sqrt{f(\eta^{k,k+1})}-\sqrt{f(\eta)}\Big)^2\!
\Big(\pN(k_{\tiN})\tun_{\{\nabla\eta(k)=1\}}+\qN(k_{\tiN})
\tun_{\Big\{\substack{\nabla\eta(k)=-1\\\eta^{k,k+1}\in\cENMunw}\Big\}}\Big).\]
The condition $\eta^{k,k+1}\in\cENMunw$ is due to the wall condition. Using the reversibility of $\bmuN$, we have the identity $\bmuN(\eta)\pN(k_{\tiN}) = \bmuN(\eta^{k,k+1})\qN(k_{\tiN})$ whenever $\nabla\eta(k)=+1$, thus we can rewrite the Dirichlet form in such a way that this wall condition becomes implicit:
\[\DNMunw(f) :=
\sum_{\eta\in\cENMunw}\bmuN(\eta)\sum_{k=1}^{2N-1}\Big(\sqrt{f\big(\eta^{k,k+1}\big)}-\sqrt{f(\eta)}\Big)^2\pN(k_{\tiN})\tun_{\{\nabla \eta(k)=1\}}.\]
The same trick can be applied in \Mdeux, see Formula (\ref{Eq:DirichletMdeux}). This is an important remark for the proof. Let us come back to the general case. A simple calculation together with the classical formula for the largest eigenvalue of a symmetric matrix yields
\begin{equation}\label{Eq:EigenValue}
\lambdaNe(a) = \sup_{f\geq 0,\bmuN[f]=1}\Big(a\sum_{\eta}\VNe(\eta)\bmuN(\eta)f(\eta)-\DN(f) \Big),
\end{equation}
where the supremum is taken over all non-negative maps $f$ on $\cEN$ such that $\sum_\eta \bmuN(\eta)f(\eta)=1$. From now on, $f$ will always be of this form. As $\VNe$ is uniformly bounded by $c'N$ for a certain constant $c'\! >\! 0$, it suffices to show that for all $c>0$
\begin{equation}\label{EqObjectiveSuperExpo}
\varlimsup\limits_{\epsilon\downarrow 0}\varlimsup\limits_{N\rightarrow\infty}\sup_{f:\DN(f)\leq cN}\frac{1}{N}\sum_{\eta\in\cEN}\VNe(\eta)\bmuN(\eta)f(\eta)=0
\end{equation}
in order to prove the theorem under the stationary measure. We have chosen to provide a complete proof in Appendix \ref{AppendixSuperExpo} that works both for \Modun\ and \Munw. It can be adapted easily to \Mdeux\ by adding extra terms.

\section{Tightness}\label{Section:Tightness}
The goal of this section is to show tightness of the sequence $\bbQN$ in order to prove Theorems \ref{ThCVModel1}, \ref{ThCVModel1w} and \ref{ThCVModel2}. Even though the state-spaces differ according to the models at stake, the methodology of proof is the same. In \Munw\ and \Mdeux, the definition of the topology on $\bbM$ and the tightness of the random measure(s) is postponed to Subsection \ref{Subsection:TightnessZeta}. Let us just note that we will define a metric on $\bbM$ that makes it a Polish space. Recall that in these two models, we consider the product topology on $\bbD([0,\infty),\cCMunw)\times\bbM$ (respectively on $\bbD([0,\infty),\cCMde)\times\bbM\times\bbM$), so that we can show separately the tightness of $\rmh$ and the tightness of $\zeta$ (respectively of $\zetaun$, $\zetade$).

\subsection{Tightness of \texorpdfstring{$\rmh$}{h}}
To alleviate the notation, we take $\nuN$ equal to the stationary measure $\bmuN$ whenever we deal with \Munw\ and \Mdeux. When we use the generic symbols $\cCN$, $\cC$ and $\bQ$ without superscript, we mean that our results apply indifferently to any model. Tightness of $\rmh$  will follow from the following two properties (see for instance Billingsley~\cite{Billingsley_CVPM}):
\begin{enumerate}[(i)]
\item\label{Tightness1} the sequence $(\nuN,N\geq 1)$ of measures on $\cC$ is tight; and
\item\label{Tightness2} for every $T>0$ we have $\lim\limits_{\beta\downarrow 0}\varlimsup\limits_{N\rightarrow\infty}\bbQN_{\nuN}\Bigg[\sup\limits_{\substack{s,t\in[0,T]\\|t-s|<\beta}}\left\|\rmh_t-\rmh_s\right\|_{\cC}^2\Bigg] = 0$.
\end{enumerate}
Property (\ref{Tightness1}) is actually an hypothesis in our theorems. To show Property (\ref{Tightness2}) we would like to prove that the process $t\mapsto\rmh_t$ is H\"older in space. As this process is not continuous in time, we actually consider its time interpolation $\brmh$ defined as
\begin{equation}\label{EqInterpol}
\brmh_t(x) := \Big(\left\lfloor t(2N)^2 \right\rfloor\!+\!1-t(2N)^2\Big)\rmh_{\frac{\left\lfloor t(2N)^2 \right\rfloor}{(2N)^2}}(x) + \Big(t(2N)^2-\left\lfloor t(2N)^2 \right\rfloor\Big)\rmh_{\frac{\left\lfloor t(2N)^2 \right\rfloor+1}{(2N)^2}}(x),
\end{equation}
which we prove to be H\"older continuous in time. First, we show that the difference between $\rmh$ and $\brmh$ vanishes as $N\rightarrow\infty$.
\begin{lemma}\label{LemmaBdelta}
For all $p>6$ we have $ \lim\limits_{N\rightarrow\infty}\bbQN_{\nuN}\bigg[\sup\limits_{t\in[0,T]}\left\|\brmh_t-\rmh_t\right\|_{\cC}^p\bigg]=0$.
\end{lemma}
\begin{proof}
Fix $p > 6$. We start with \Modun\ and \Munw. Suppose there exists $c > 0$ such that for all $N\geq 1$, $k\in\lbr 0, 2N\!-\!1\rbr$, $i\in\big[0,\lfloor T(2N)^2 \rfloor\big]$ we have
\begin{equation}\label{EqBoundB}
\bbQN_{\nuN}\Big[\sup_{x\in[k_{\tiN},(k+1)_{\tiN}]}\sup_{t(2N)^2\in[i,i+1]}\big|\brmh_t(x)-\rmh_{t}(x)\big|^p\Big]^{\frac{1}{p}} \leq \frac{c}{\sqrt{2N}},
\end{equation}
then we deduce that
\begin{eqnarray*}
\bbQN_{\nuN}\Big[\sup_{t\in[0,T]}\left\|\brmh_t-\rmh_t\right\|_{\cC}^p \Big] \!\!\!&\leq&\!\!\! \sum_{k=0}^{2N-1}\sum_{i=0}^{\lfloor T(2N)^2 \rfloor}\bbQN_{\nuN}\Big[\sup_{x\in[k_{\tiN},(k+1)_{\tiN}]}\sup_{t(2N)^2\in[i,i+1]}\big|\brmh_t(x)-\rmh_{t}(x)\big|^p\Big]\\
\!\!\!&\leq&\!\!\! c^pT(2N)^{3-\frac{p}{2}} \underset{N\rightarrow\infty}{\rightarrow} 0.
\end{eqnarray*}
We now prove (\ref{EqBoundB}). Fix $k,i$ as above. The very definition of $\brmh$ yields that for all $x\in[k_{\tiN},(k+1)_{\tiN}]$ and all $t\in[i(2N)^{-2},(i+1)(2N)^{-2}]$
\begin{eqnarray}\label{Eq:BoundInterpo}
|\brmh_t(x)-\rmh_t(x)|&\leq&|\rmh_{i(2N)^{-2}}(k_{\tiN})-\rmh_t(k_{\tiN})|+|\rmh_{i(2N)^{-2}}\big((k+1)_{\tiN}\big)-\rmh_t\big((k+1)_{\tiN}\big)|\\\nonumber
&&\!\!\!\!\!+|\rmh_{(i+1)(2N)^{-2}}(k_{\tiN})-\rmh_t(k_{\tiN})|+|\rmh_{(i+1)(2N)^{-2}}\big((k+1)_{\tiN}\big)-\rmh_t\big((k+1)_{\tiN}\big)|.
\end{eqnarray}
Now observe that $\sup_{t\in[i(2N)^{-2},(i+1)(2N)^{-2}]}|\rmh_{i(2N)^{-2}}(k_{\tiN})-\rmh_t(k_{\tiN})|$ is bounded by $2/\sqrt{2N}$ times the number of jumps of $\cR(k_{\tiN})+\cL(k_{\tiN})$ on the time interval $[i(2N)^{-2},(i+1)(2N)^{-2}]$, that is, $2/\sqrt{2N}$ times a Poisson random variable with parameter $1$. A similar bound holds true for the other three terms. Consequently (\ref{EqBoundB}) follows. For \Mdeux, the proof is almost identical: all the increments displayed in (\ref{Eq:BoundInterpo}) are taken in $\bbR^2$ rather than in $\bbR$ and the Poisson random variable has parameter $2$ rather than $1$, since there are four Poisson processes.\cqfd
\end{proof}
\noindent From now on, we write $\{|t-s|< \beta\}$ for the set $\{s,t \in [0,T]:|t-s|<\beta\}$. Then we observe that for all $p>6$
\begin{equation}\label{EqBoundUnifMoment}
\bbQN_{\nuN}\bigg[\!\sup\limits_{\{|t-s|< \beta\}}\left\|\rmh_t-\rmh_s\right\|_{\cC}^p\bigg]^{\frac{1}{p}}\leq 2\,\bbQN_{\nuN}\bigg[\sup\limits_{t\in[0,T]}\left\|\brmh_t-\rmh_t\right\|_{\cC}^p\bigg]^{\frac{1}{p}}+\bbQN_{\nuN}\Bigg[\bigg(\sup\limits_{\{|t-s|< \beta\}}\frac{\left\|\brmh_t-\brmh_s\right\|_{\cC}}{|t-s|^a} \bigg)^p\Bigg]^{\frac{1}{p}}\beta^a.
\end{equation}
The first term on the r.h.s.~vanishes as $N\rightarrow\infty$, thanks to Lemma \ref{LemmaBdelta}, while the second term on the r.h.s.~is finite whenever $a$ is small enough and $p$ is large enough, as the following result shows.
\begin{proposition}\label{PropHolderSpaceTimeBar}
There exists $p>8$ and $a>0$ such that
\[ \sup_{N\geq 1}\bbQN_{\nuN}\Bigg[\bigg( \sup_{s,t\in[0,T]}\frac{\left\|\brmh_t-\brmh_s\right\|_{\cC}}{|t-s|^a} \bigg)^p\Bigg]^{\frac{1}{p}} < \infty.\]
\end{proposition}
\noindent Letting $N$ tend to infinity and $\beta$ to $0$ in (\ref{EqBoundUnifMoment}), we deduce that Property (\ref{Tightness2}) is verified, so that the tightness of $\rmh$ under $\bbQN_{\nuN}$ now boils down to proving Proposition \ref{PropHolderSpaceTimeBar}. Below we provide the proof when $\rmh$ starts from the stationary measure $\bmuN$, while the specific proof for \Modun\ starting from a measure $\nuN$ that only satisfies the hypothesis of Theorem \ref{ThCVModel1} is postponed to Appendix \ref{SectionProofTightnessMun}, as it relies on different arguments.

\subsection{Proof of Proposition \ref{PropHolderSpaceTimeBar} under the stationary measure}\label{SubsectionTightnessh}
We now restrict ourselves to \Mdeux\ as this is the most involved setting. The arguments can easily be adapted to the other models. For any $\alpha\geq 0$ we define the Sobolev space of distributions:
\[\cH_{-\alpha}=\Big\{f\in \rS'([0,1]): \sum_{n\geq 0}\hat{f}(n)^2(1+n)^{-2\alpha} < \infty \Big\}.\]
Recall also the Sobolev-Slobodeckij spaces introduced in (\ref{Eq:SobolevSlobodeckij}). Fix $T>0$. In order to prove Proposition \ref{PropHolderSpaceTimeBar}, we first obtain a uniform bound on the Sobolev-Slobodeckij norm of the increments of $\brmh$, see Lemma \ref{Lemma:UnifBoundSobolev}, and we show tightness in $\cH_{-\alpha}$, see Proposition \ref{Prop:UnifBoundLyonsZheng}. The proof of Proposition \ref{PropHolderSpaceTimeBar} then relies on an interpolation argument between these two function spaces.
\begin{lemma}\label{Lemma:UnifBoundSobolev}
For any $\eta \in (0,\frac{1}{2})$, any $r\geq 1$ and any integer $p\geq 1$ we have
\[ \sup_{N\geq 1}\bbQN_{\bmuN}\bigg[\left\|\brmhi_t-\brmhi_s\right\|_{\cW^{\eta,r}}^p\bigg]^{\frac{1}{p}} < \infty.\]
\end{lemma}
\begin{proof}
By symmetry, it suffices to consider $i=1$. Observe that
\[ \left\|\brmhun_t\right\|_{\cW^{\eta,r}} \leq \left\|\rmhun_{\frac{\left\lfloor t(2N)^2 \right\rfloor}{(2N)^2}}\right\|_{\cW^{\eta,r}}+\left\|\rmhun_{\frac{\left\lfloor t(2N)^2 \right\rfloor+1}{(2N)^2}}\right\|_{\cW^{\eta,r}}.\]
Thus, by stationarity,
\[ \bbQN_{\bmuN}\bigg[\left\|\brmhun_t-\brmhun_s\right\|_{\cW^{\eta,r}}^p\bigg]^{\frac{1}{p}} \leq 4\,\bmuNMde\Big[\left\|\hun\right\|_{\cW^{\eta,r}}^p\Big]^{\frac{1}{p}}.\]
Recall that $\piNMde$ is the invariant measure in the symmetric case $\pN(\cdot)=\qN(\cdot)$. We have
\[\bmuNMde\Big[\left\|\hun\right\|_{\cW^{\eta,r}}^p\Big]^{\frac{1}{p}} \leq \piNMde\Big[\left\|\hun\right\|_{\cW^{\eta,r}}^{2p}\Big]^{\frac{1}{2p}}\piNMde\Big[\Big(\frac{d\bmuNMde}{d\piNMde}\Big)^2\Big]^{\frac{1}{2p}}.\]
For every $h\in\cCNMde$, we have
\[ \bmuNMde(h)=\frac{\exp((2N)^{\frac{3}{2}}\ccA_{\tiN}(h))}{\sum_{h'\in\cCNMde}\exp((2N)^{\frac{3}{2}}\ccA_{\tiN}(h'))}\;\;\;\;\;\;,\;\;\;\;\;\;\piNMde(h) = \frac{1}{\#\cCNMde},\]
so that the second moment of the Radon-Nikodym derivative can be written
\[\piNMde\Big[\Big(\frac{d\bmuNMde}{d\piNMde}\Big)^2\Big]=\frac{\piNMde\Big[\exp(2(2N)^{\frac{3}{2}}\ccA_{\tiN}(h))\Big]}{\piNMde\Big[\exp((2N)^{\frac{3}{2}}\ccA_{\tiN}(h))\Big]^2}.\]
The r.h.s.~is uniformly bounded in $N\geq 1$, as we showed in the proof of Theorem \ref{ThCVInvMeasure}. Moreover, Lemma \ref{Lemma:SlobodeckijWalk} ensures that $\sup_{N\geq 1}\piNMde\Big[\left\|\hun\right\|_{\cW^{\eta,r}}^{2p}\Big] <\infty$. This completes the proof.\cqfd
\end{proof}
The second result needed for the proof of Proposition \ref{PropHolderSpaceTimeBar} is the following control on the modulus of continuity of $\brmh$ in a Sobolev space of distributions.
\begin{proposition}\label{Prop:UnifBoundLyonsZheng}
For any $\alpha > \frac{1}{2}$ and any integer $p\geq1$ there exists $c > 0$ such that for every $0\leq s \leq t \leq T$
\[ \sup_{N\geq 1}\bbQN_{\bmuN}\bigg[\left\|\brmhi_t-\brmhi_s\right\|_{\cH_{-\alpha}}^{2p} \bigg]^{\frac{1}{2p}} \leq c(t-s)^{\frac{1}{4}}.\]
\end{proposition}
\noindent We postpone the proof of this result to the end of this subsection.\vspace{4pt}\\
\noindent\textit{Proof of Proposition \ref{PropHolderSpaceTimeBar} under the stationary measure.}
We use an interpolation argument inspired by the work of Debussche and Zambotti~\cite{DebusscheZambotti07} p.1721. Fix $b\in(0,\frac{1}{2})$ and set
\[ \eta := \frac{b+1/2}{2}\in (0,1/2) \;\;,\;\; r := \frac{8}{1-2b} \geq 1 \;\;,\;\;\kappa := \frac{14b+25}{12b+26}\in(0,1)\;\;,\;\;\alpha:=1 > 1/2.\]
Then we define $\delta := \kappa\eta-(1-\kappa)\alpha$ and $\frac{1}{q}:=\frac{\kappa}{r}+\frac{1-\kappa}{2}$. Notice that these parameters have been chosen such that we can apply Proposition \ref{Prop:UnifBoundLyonsZheng} and Lemma \ref{Lemma:UnifBoundSobolev}, and such that $(\delta-b)q > 1$. Theorem 1 of Section 4.3.1, Remark 2-b of Section 2.4.2 and Theorem-g of Section 1.3.3 in the book of Triebel~\cite{Triebel78} ensures the existence of a constant $\tcInterpo>0$ which only depends on the parameters of the function spaces at stake such that $\left\|f\right\|_{\cW^{\delta,q}}\leq \tcInterpo\left\|f\right\|_{\cW^{\eta,r}}^\kappa\left\|f\right\|_{\cH^{-\alpha}}^{1-\kappa}$ for every $f\in\cW^{\eta,r}\cap\cH^{-\alpha}$. Using H\"older's inequality we then obtain, for every $p\geq 1$
\[\bbQN_{\bmuN}\bigg[\left\|\brmhi_t-\brmhi_s\right\|_{\cW^{\delta,q}}^p\bigg] \leq \tcInterpo^p\bbQN_{\bmuN}\bigg[\left\|\brmhi_t-\brmhi_s\right\|_{\cW^{\eta,r}}^p\bigg]^\kappa\bbQN_{\bmuN}\bigg[\left\|\brmhi_t-\brmhi_s\right\|_{\cH^{-\alpha}}^p\bigg]^{1-\kappa}.\]
Since we chose the parameters such that $(\delta-b)q > 1$, the space $\cW^{\delta,q}$ is continuously embedded (see for instance Theorem 8.2 in~\cite{Hitchhiker12}) into the H\"older space:
\[ \ccC^{b}([0,1],\bbR):= \bigg\{f:[0,1]\rightarrow\bbR \;\;\text{ s.t. }\;\; f(0)=f(1)=0 \;\;\;,\;\;\; \sup_{x\ne y \in [0,1]}\frac{|f(x)-f(y)|}{|x-y|^b} < \infty\bigg\}.\]
From this observation, and using Proposition \ref{Prop:UnifBoundLyonsZheng} and Lemma \ref{Lemma:UnifBoundSobolev}, we deduce that for any given integer $p>\frac{2}{1-\kappa}$ there exists a constant $c>0$ such that
\[\forall t,s\in[0,T],\;\;\;\sup_{N\geq 1}\bbQN_{\bmuN}\bigg[\left\|\brmhi_t-\brmhi_s\right\|_{\ccC^{b}}^{2p}\bigg] \leq c|t-s|^{\frac{(1-\kappa)p}{2}}.\]
Using Kolmogorov's Continuity Theorem, we deduce the existence of a modification of $(\brmhi_t,t\in[0,T])$ which is $a$-H\"older continuous in time in the $\ccC^{b}$-norm for any $a\in(0,\frac{1-\kappa}{4}-\frac{1}{2p})$. Since $\brmh$ is already continuous in space and time by construction, we deduce that it coincides $\bbQN_{\bmuN}$-a.s.~with its modification. Consequently there exists $c>0$ such that for every $i\in\{1,2\}$
\[ \sup_{N\geq 1}\bbQN_{\bmuN}\Bigg[\bigg( \sup_{s\ne t\in[0,T]}\frac{\sup_{x\in[0,1]}|\brmhi_t(x)-\brmhi_s(x)|}{|t-s|^a} \bigg)^p\Bigg]^{\frac{1}{p}} \leq c\]
and thus
\[ \sup_{N\geq 1}\bbQN_{\bmuN}\Bigg[\bigg( \sup_{s\ne t\in[0,T]}\frac{\sup_{x\in[0,1]}\Big(|\brmhun_t(x)-\brmhun_s(x)|+|\brmhde_t(x)-\brmhde_s(x)|\Big)}{|t-s|^a} \bigg)^p\Bigg]^{\frac{1}{p}} \leq 2c.\]
This completes the proof of Proposition \ref{PropHolderSpaceTimeBar}.\cqfd\vspace{4pt}\\
We now proceed to the proof of Proposition \ref{Prop:UnifBoundLyonsZheng}, which relies on the Fourier decomposition of $\rmh$. Consider the orthonormal basis of $\rL^2([0,1],dx)$ defined by $\varepsilon_0(x)=1$ and $\varepsilon_n(x)=\sqrt{2}\cos(n\pi x)$ for every $n\geq 1$. For every $n\geq 0$ and any tempered distribution $f\in \rS'([0,1])$, we define the $n$-th Fourier coefficient $\hat{f}(n):=\langle f,\varepsilon_n\rangle$. A simple calculation ensures that
\[ \Hrmhi_t(n) = \sum_{k=1}^{2N-1}c_{n,k}\rmhi_t(k_{\tiN}),\;\;\;i\in\{1,2\},\]
where $c_{0,k} := \frac{1}{2N}$ and $c_{n,k} := \frac{4\sqrt{2}N}{(n\pi)^2}\cos(n\pi k_{\tiN})\big(1-\cos(\frac{n\pi}{2N})\big)$ for all $n\geq 1$. Observe that
\begin{equation}\label{Eq:BoundCnk}
\sup_{n,k} |c_{n,k}| \leq \frac{\sqrt{2}}{2N}.
\end{equation}
We deduce from (\ref{Eq:DefMde}) that the Fourier coefficients satisfy, for all $0 \leq s \leq t$,
\begin{equation}\label{Eq:Hrmh}
\Hrmhi_t(n)-\Hrmhi_s(n) = \int_{s}^{t}\hat{d}^{\mbox{\tiny $(i)$}}_u(n)du + \hat{M}^{\mbox{\tiny $(i)$}}_{s,t}(n),
\end{equation}
where for $i=1$ we have
\begin{eqnarray*}
\hat{d}^{\mbox{\tiny $(1)$}}_s(n) &:=& \frac{2}{\sqrt{2N}}\sum_{k=1}^{2N-1}c_{n,k}\big(\pN(k_{\tiN})\tun_{\{\Delta \rmhun_s(k_{\tiN}) >0\}}-\qN(k_{\tiN})\tun_{\{\Delta \rmhun_s(k_{\tiN}) <0;\rmhun_s(k_{\tiN}) >\rmhde_s(k_{\tiN})\}}\big),\\
\hat{M}^{\mbox{\tiny $(1)$}}_{s,t}(n)&:=& \frac{2}{\sqrt{2N}}\int_{s}^{t}\sum_{k=1}^{2N-1}c_{n,k}\Big((d\cLNun_u(k_{\tiN})-\pN(k_{\tiN}) du)\tun_{\{\Delta \rmhun_u(k_{\tiN}) >0\}}\\
&&\hspace{3.0cm}- (d\cRNun_u(k_{\tiN})-\qN(k_{\tiN}) du)\tun_{\{\Delta \rmhun_u(k_{\tiN}) <0;\rmhun_u(k_{\tiN})>\rmhde_u(k_{\tiN})\}}\Big).
\end{eqnarray*}
The expressions for the corresponding processes for $i=2$ follow via obvious modifications. The proof of Proposition \ref{Prop:UnifBoundLyonsZheng} actually relies on three preliminary lemmas.
\begin{lemma}\label{LemmaFKDrift}
$\sup_{n\geq 0, N\geq 1}\sup_{s\leq t}\bbQN_{\bmuN}\bigg[\exp\Big((t-s)^{-\frac{1}{2}}\big|\int_s^t\hat{d}^{\mbox{\tiny $(i)$}}_r(n)dr\big|\Big)\bigg] <\infty.$
\end{lemma}
\noindent The proof of this lemma is similar to that of Lemma 11.3.9 in Kipnis and Landim~\cite{KipnisLandimBook}.
\begin{proof}
We restrict to $i=1$ for simplicity. Until the end of the proof, we use the notations of Subsection \ref{SubsectionSuperExpo} and we work with the canonical process $\upeta$ on $\bbD([0,T],\cENMde)$. We define the following operator $\hat{V}_n$:
\[ \hat{V}_n(\eta) = \frac{2}{\sqrt{2N}}\sum_{k=1}^{2N-1}c_{n,k}\big(\pN(k_{\tiN})\tun_{\{\nabla \etaun(k)=1\}}-\qN(k_{\tiN})\tun_{\{\nabla \etaun(k)=-1;\eta^{(1),k,k+1}\in\cENMde\}}\big),\]
where $\eta^{(1),k,k+1}$ is obtained from $\eta=(\etaun,\etade)$ by exchanging the values $\etaun(k)$ and $\etaun(k+1)$. For any $a \in\bbR$, one can apply the methodology presented in Section \ref{SubsectionSuperExpo} to the operator $\LN + a\hat{V}_n$. Let $\lambdaN(a)$ be its largest eigenvalue which satisfies Formula (\ref{Eq:EigenValue}) where $\DN$ is the Dirichlet form defined for all $f\geq 0$ by
\begin{equation}\label{Eq:DirichletMdeux}
\begin{split}
\DNMde(f) &:= \sum_{\eta\in\cENMde}\bmuNMde(\eta)\sum_{k=1}^{2N-1}\pN(k_{\tiN})\bigg(\Big(\sqrt{f\big(\eta^{(1),k,k+1}\big)}-\sqrt{f(\eta)}\Big)^2\tun_{\{\nabla \etaun(k)=1\}}\\
&\hspace{5.3cm}+\Big(\sqrt{f\big(\eta^{(2),k,k+1}\big)}-\sqrt{f(\eta)}\Big)^2\tun_{\{\nabla \etade(k)=-1\}}\bigg).
\end{split}
\end{equation}
Observe that, using the same argument as in Subsection \ref{SubsectionSuperExpo} for \Munw, we have written the Dirichlet form in such a way that the interaction between the interfaces does not appear. Similarly, for all $a \in \bbR$ the quantity $a\sum_{\eta}\hat{V}_n(\eta)\bmuNMde(\eta)f(\eta)$ can be written:
\[a\frac{2}{\sqrt{2N}}\sum_{k=1}^{2N-1}c_{n,k}\sum_{\eta}\bmuNMde(\eta)\pN(k_{\tiN})\tun_{\{\nabla\etaun(k)=1\}}\Big(f(\eta)-f\big(\eta^{(1),k,k+1}\big)\Big).\]
Fix $f\geq 0$ such that $\bmuNMde[f]=1$. Since for all $\gamma > 0$ we have
\[ \Big|f(\eta)-f\big(\eta^{(1),k,k+1}\big)\Big| \leq \frac{1}{2\gamma}\Big(\sqrt{f(\eta^{(1),k,k+1})}-\sqrt{f(\eta)}\Big)^2 + \frac{\gamma}{2}\Big(\sqrt{f(\eta)}+\sqrt{f(\eta^{(1),k,k+1})}\Big)^2, \]
using (\ref{Eq:BoundCnk}) we see that $\big|a\sum_{\eta}\hat{V}_n(\eta)\bmuNMde(\eta)f(\eta)\big|$ is bounded by
\begin{eqnarray*}
&&\frac{|a|\sqrt{2}}{\gamma(2N)^{\frac{3}{2}}}\sum_{k=1}^{2N-1}\sum_{\eta}\bmuNMde(\eta)\pN(k_{\tiN})\tun_{\{\nabla\etaun(k)=1\}}\Big(\sqrt{f\big(\eta^{(1),k,k+1}\big)}-\sqrt{f(\eta)}\Big)^2\\ &+&\, \frac{|a|\sqrt{2}\gamma}{(2N)^{\frac{3}{2}}}\sum_{k=1}^{2N-1}\sum_{\eta}\bmuNMde(\eta)\pN(k_{\tiN})
\tun_{\{\nabla\etaun(k)=1\}}\Big(\sqrt{f\big(\eta^{(1),k,k+1}\big)}+\sqrt{f(\eta)}\Big)^2.
\end{eqnarray*}
Taking $\gamma =|a|\sqrt{2}(2N)^{-3/2}$, the last expression is bounded above by
\[\DNMde(f) + \frac{2a^2}{(2N)^{3}}\sum_{k=1}^{2N-1}\sum_{\eta}\bmuNMde(\eta)\pN(k_{\tiN})\tun_{\{\nabla\etaun(k)=1\}}\Big(\sqrt{f\big(\eta^{(1),k,k+1}\big)}+\sqrt{f(\eta)}\Big)^2\\
\leq\DNMde(f) + 8a^2,\]
where we use the fact that the $\rL^1$ norm of $f$ equals $1$. The Feynman-Kac formula (see for instance Appendix 1 - Lemma 7.2 in~\cite{KipnisLandimBook}) ensures that for all $t\geq 0$ and all $a > 0$
\[ \bbQN_{\bmuN}\Big[e^{a\big|\int_0^t\hat{d}^{\mbox{\tiny $(1)$}}_s(n)ds\big|}\Big]\leq \bbQN_{\bmuN}\Big[e^{a\int_0^t\hat{d}^{\mbox{\tiny $(1)$}}_s(n)ds}\Big]+\bbQN_{\bmuN}\Big[e^{-a\int_0^t\hat{d}^{\mbox{\tiny $(1)$}}_s(n)ds}\Big] \leq 2e^{8a^2t}.\]
The value $a=1/\sqrt{t}$ and a stationarity argument yield the asserted result.\cqfd
\end{proof}
\begin{lemma}\label{Lemma:BoundHrmh}
For every integer $m\geq 1$ there exists $\tc(m) > 0$ such that for every $0 \leq s \leq t$ and every $N\geq 1$
\[ \sup_{n\geq 0}\bbQN_{\bmuN}\Big[\big(\Hrmhi_t(n)-\Hrmhi_s(n)\big)^{2m}\Big]^{\frac{1}{2m}} \leq \tc(m)\Big(\sqrt{t-s}+\Big(\frac{t-s}{N^3}\Big)^{\frac{1}{4}}+N^{-\frac{5}{4}}\Big).\]
\end{lemma}
\noindent Observe that one cannot expect to have a bound of the form $(t-s)^a$ since otherwise the process $\Hrmhi$ would have a continuous modification by the Kolmogorov continuity criterion. However, the extra terms vanish as $N$ tends to infinity so that any limiting process will be continuous.
\begin{proof}
We restrict to $i=1$ for simplicity. Using (\ref{Eq:Hrmh}) we write
\[\bbQN_{\bmuN}\Big[\big(\Hrmhun_t(n)-\Hrmhun_s(n)\big)^{2m}\Big]^{\frac{1}{2m}} \leq \bbQN_{\bmuN}\bigg[\Big(\int_s^t\hat{d}^{\mbox{\tiny $(1)$}}_r(n)dr\Big)^{2m}\bigg]^{\frac{1}{2m}} + \bbQN_{\bmuN}\Big[\big(\hat{M}^{\mbox{\tiny $(1)$}}_{s,t}(n)\big)^{2m}\Big]^{\frac{1}{2m}}.\]
The bound for the first term on the right hand side is a direct consequence of Lemma \ref{LemmaFKDrift}. To bound the second term, we define the martingale $\hat{D}^{\mbox{\tiny $(1)$}}_{s,t}(n):=\big[\hat{M}^{\mbox{\tiny $(1)$}}_{s,\cdot}(n)\big]_t - \big\llangle \hat{M}^{\mbox{\tiny $(1)$}}_{s,\cdot}(n)\big\rrangle_t$, and we use the Burkholder-Davis-Gundy inequality twice (we refer to Formula (\ref{Eq:BDGTwice}) in Appendix \ref{SectionProofTightnessMun}) to obtain
\[\bbQN_{\bmuN}\Big[\big(\hat{M}^{\mbox{\tiny $(1)$}}_{s,t}(n)\big)^{2m}\Big]^{\frac{1}{2m}} \leq \tcBDG(2m)\Big(\bbQN_{\bmuN}\Big[\big\llangle \hat{M}^{\mbox{\tiny $(1)$}}_{s,\cdot}(n)\big\rrangle_t^{m}\Big]^{\frac{1}{2m}}+\sqrt{\tcBDG(m)}\bbQN_{\bmuN}\Big[\big[\hat{D}^{\mbox{\tiny $(1)$}}_{s,\cdot}(n)\big]_t^{\frac{m}{2}}\Big]^{\frac{1}{2m}}\Big).\]
It is elementary to check that $\bbQN_{\bmuN}$-a.s.~$\big\llangle \hat{M}^{\mbox{\tiny $(1)$}}_{s,\cdot}(n)\big\rrangle_t \leq 8(t-s)$ so that the bound for the corresponding term is immediate. We turn to the quadratic variation and write
\begin{eqnarray*}
&&\big[\hat{D}^{\mbox{\tiny $(1)$}}_{s,\cdot}(n)\big]_t=\sum_{k=1}^{2N-1}\Big(\frac{2c_{n,k}}{\sqrt{2N}}\Big)^4\#\big\{\tau\in(s,t]:\rmhun_\tau(k_{\tiN})\ne \rmhun_{\tau-}(k_{\tiN})\big\}\\
&\leq&N^{-6}\sum_{k=1}^{2N-1}\sum_{j=0}^{\lfloor (t-s)(2N)^2\rfloor}\Big(\cLNun_{s+\frac{j+1}{(2N)^2}}(k_{\tiN})-\cLNun_{s+\frac{j}{(2N)^2}}(k_{\tiN})+\cRNun_{s+\frac{j+1}{(2N)^2}}(k_{\tiN})-\cRNun_{s+\frac{j}{(2N)^2}}(k_{\tiN})\Big).
\end{eqnarray*}
Observe that for each $j$ and each $k$ the random variable on the right hand side of the last equation has a Poisson distribution with mean $1$. Consequently there exists a constant $c'>0$ such that for every $N\geq 1$
\[ \sup_{n\geq 0}\bbQN_{\bmuN}\Big[\big[\hat{D}^{\mbox{\tiny $(1)$}}_{s,\cdot}(n)\big]_t^{\frac{m}{2}}\Big]^{\frac{2}{m}} \leq \frac{c'}{N^5}\big(\lfloor (t-s)(2N)^2\rfloor+1\big),\]
and the asserted bound follows.\cqfd
\end{proof}
\begin{lemma}\label{Lemma:BoundUnifHatBarh}
For every integer $m\geq 1$ there exists $\bar{\tc}(m)>0$ such that for all $0\leq s \leq t \leq T$
\[ \sup_{n\geq 0,N\geq 1}\bbQN_{\bmuN}\Big[\big(\Hbrmhi_t(n)-\Hbrmhi_s(n)\big)^{2m}\Big]^{\frac{1}{2m}} \leq \bar{\tc}(m)(t-s)^{\frac{1}{4}}.\]
\end{lemma}
A more technical proof would yield a bound of order $(t-s)^{\frac{1}{2}}$ which is more intuitive since the Fourier modes in the limiting stochastic PDE are Brownian like. However we will not need such an accurate bound.
\begin{proof}
We restrict to $i=1$ for simplicity. Assume first that $t-s < (2N)^{-2}$. We set
\[ \forall k\in\lbr 1,2N-1\rbr,\;\;\; J_k:=\cLNun_{\frac{\lfloor t(2N)^2\rfloor+1}{(2N)^2}}(k_{\tiN})-\cLNun_{\frac{\lfloor s(2N)^2\rfloor}{(2N)^2}}(k_{\tiN})+\cRNun_{\frac{\lfloor t(2N)^2\rfloor+1}{(2N)^2}}(k_{\tiN})-\cRNun_{\frac{\lfloor s(2N)^2\rfloor}{(2N)^2}}(k_{\tiN}).\]
Each $J_k$ is a Poisson random variable with mean at most $2$. Recall that $\brmhun$ is the time interpolation of $\rmhun$, so that $\bbQN_{\bmuN}$-a.s.
\[ \big|\brmhun_t(k_{\tiN})-\brmhun_s(k_{\tiN})\big|\leq\frac{2}{\sqrt{2N}}(t-s)(2N)^2J_k.\]
This implies, together with (\ref{Eq:BoundCnk}), that the Fourier coefficients of $\brmhun$ satisfy $\bbQN_{\bmuN}$-a.s.,
\[ \sup_{n\geq 0}\big|\Hbrmhun_t(n)-\Hbrmhun_s(n)\big|\leq\frac{2\sqrt{2}}{(2N)^{\frac{3}{2}}}(t-s)(2N)^2\sum_{k=1}^{2N-1}J_k.\]
Since $(t-s)(2N)^{2}<1$, we have $(t-s)(2N)^{2}<(t-s)^{\frac{1}{4}}\sqrt{2N}$ and thus
\[ \sup_{n\geq 0}\bbQN_{\bmuN}\Big[\big(\Hbrmhun_t(n)-\Hbrmhun_s(n)\big)^{2m}\Big]^{\frac{1}{2m}} \leq 2\sqrt{2}\,a(2m)(t-s)^{\frac{1}{4}},\]
where $a(2m)$ is the $\rL^{2m}$-norm of a Poisson random variable with mean $2$. The asserted uniform bound follows. Assume now that $t-s \geq (2N)^{-2}$ and write
\[|\Hbrmhun_t(n)-\Hbrmhun_s(n)| \leq |\Hbrmhun_t(n)-\Hbrmhun_{\frac{\lfloor t(2N)^2\rfloor}{(2N)^2}}(n)|+ |\Hbrmhun_{\frac{\lfloor t(2N)^2\rfloor}{(2N)^2}}(n)-\Hbrmhun_{\frac{\lfloor s(2N)^2\rfloor}{(2N)^2}}(n)| +|\Hbrmhun_s(n)-\Hbrmhun_{\frac{\lfloor s(2N)^2\rfloor}{(2N)^2}}(n)|.\]
The bound already obtained applies to the first and third terms, while we use the fact that $\brmhun$ and $\rmhun$ coincide at times of the form $\frac{\lfloor t(2N)^2\rfloor}{(2N)^2}$ to bound the second term using Lemma \ref{Lemma:BoundHrmh} as follows:
\begin{eqnarray*}
&&\sup_{n\geq 0}\bbQN_{\bmuN}\Big[\big(\Hbrmhun_t(n)-\Hbrmhun_s(n)\big)^{2m}\Big]^{\frac{1}{2m}}\\
\!\!\!&\leq&\!\!\! 4\sqrt{2}\,a(2m)(2N)^{-\frac{1}{2}}+ \tc(m)\Big(\frac{\sqrt{\lfloor t(2N)^2 \rfloor-\lfloor s(2N)^2 \rfloor}}{2N}+\Big(\frac{\lfloor t(2N)^2 \rfloor-\lfloor s(2N)^2 \rfloor}{(2N)^2 N^3}\Big)^{\frac{1}{4}}+N^{-\frac{5}{4}}\Big)\\
\!\!\!&\leq&\!\!\!4\sqrt{2}\,a(2m)(t-s)^{\frac{1}{4}}+\tc(m)\Big(\sqrt{3(t-s)} + 3^\frac{1}{4}2^{\frac{3}{4}}(t-s)^{\frac{5}{8}}+2^{\frac{5}{4}}(t-s)^{\frac{5}{8}}\Big).
\end{eqnarray*}
To obtain the third line, we have used the simple inequality $|\lfloor t(2N)^2 \rfloor-\lfloor s(2N)^2 \rfloor| \leq 3(t-s)(2N)^2$ which holds since $t-s \geq (2N)^{-2}$. The asserted uniform bound follows from the fact that, as $s,t \in [0,T]$, we have $(t-s)^{\frac{5}{8}} \leq T^{\frac{3}{8}}(t-s)^{\frac{1}{4}}$ and $\sqrt{t-s} \leq T^{\frac{1}{4}}(t-s)^{\frac{1}{4}}$.\cqfd
\end{proof}
\textit{Proof of Proposition \ref{Prop:UnifBoundLyonsZheng}.} Fix an integer $p\geq 1$ and a real value $\alpha >\frac{1}{2}$. Using the Cauchy-Schwarz inequality $p-1$ times in the first line and Lemma \ref{Lemma:BoundUnifHatBarh} in the second line, one obtains that for all $0\leq s \leq t \leq T$ and all $N\geq 1$
\begin{eqnarray*}
\bbQN_{\bmuN}\bigg[\left\|\brmhi_t-\brmhi_s\right\|_{\cH_{-\alpha}}^{2p} \bigg]\!\!\!&\leq&\!\!\!\!\!\!\! \sum_{n_1,\ldots,n_p \geq 0}\!\!(1+n_1)^{-2\alpha}\ldots(1+n_p)^{-2\alpha}\prod_{j=1}^p\bbQN_{\bmuN}\Big[\big(\Hbrmhi_t(n_j)-\Hbrmhi_s(n_j)\big)^{2^{j+1}}\Big]^{\frac{1}{2^j}}\\
\!\!\!&\leq&\!\!\!\!\!\!\! \sum_{n_1,\ldots,n_p \geq 0}(1+n_1)^{-2\alpha}\ldots(1+n_p)^{-2\alpha}\prod_{j=1}^p\Big(\bar{\tc}(2^{j})(t-s)^{\frac{1}{4}}\Big)^2\\
\!\!\!&\leq&\!\!\! (t-s)^{\frac{p}{2}} \Big(\sum_{n\geq 0}(1+n)^{-2\alpha}\max_{j=1..p}\bar{\tc}^2(2^{j})\Big)^p.
\end{eqnarray*}
This completes the proof.\cqfd

\subsection{Tightness of \texorpdfstring{$\zeta$}{zeta}}\label{Subsection:TightnessZeta}

We equip the set $\bbM$ of Borel measures on $[0,\infty)\times(0,1)$ satisfying $\int_{[0,t]\times(0,1)}x(1-x)m(ds,dx)<\infty$ for every $t\geq 0$, with the smallest topology that makes
\[ m \mapsto \int_{[0,\infty)\times(0,1)}x(1-x)g(t,x)m(dt,dx)\]
continuous, for all maps $g$ that belong to $\ccC_{c}([0,\infty)\times[0,1],\bbR)$, taken to be the set of continuous maps that vanish outside a compact set of $[0,\infty)\times[0,1]$. Let us metrise this topology. Consider the countable set $\ccP$ of polynomials of two variables, $x$ and $t$, with rational coefficients. For each pair of rational values $p,q>0$, let $\rho_{p,q}:[0,\infty)\times[0,1]\rightarrow\bbR$ be a positive smooth function, equal to $1$ on $[0,p]\times[0,1]$, smaller than $1$ on $(p,p+q]\times[0,1]$ and that vanishes on $[p+q,\infty)\times[0,1]$. Let $g_k,k\geq 1$ be an enumeration of the product set $\{f\rho_{p,q}: f\in\ccP \mbox{ and } p,q \in (0,\infty)\cap\bbQ\}$. For every $k\geq 1$, we set $\varphi_k(t,x):=x(1-x)g_k(t,x)$. Then
\[ d(m,m'):=\sum_{k\geq 1}2^{-k}\Big(1\wedge\Big|\int\varphi_k dm-\int\varphi_k dm'\Big|\Big) \]
defines a distance on $\bbM$ that generates the above topology. Indeed, by the Stone-Weierstrass theorem, for every continuous function $g:[0,\infty)\times[0,1]\rightarrow\bbR$ that vanishes outside a compact set, $[0,T]\times[0,1]$ say, and for every $\epsilon > 0$, there exists $k\geq 1$ such that $\left\|g-g_k\right\|_{\infty}<\epsilon$ and $\mbox{supp } g_k \subset [0,T+\epsilon]\times[0,1]$.
\begin{lemma}
The metric space $(\bbM,d)$ is Polish.
\end{lemma}
\begin{proof}
The set of linear combinations of Dirac masses on $([0,\infty)\times(0,1))\cap(\bbQ\times\bbQ)$ with rational coefficients is dense in this metric space, so that it is separable. The completeness can be proved as follows. Let $m_n,n\geq 1$ be a Cauchy sequence for $d$. Define $\nu_n(dt,dx):=x(1-x)m_n(dt,dx)$. Then, by a diagonal argument there exists an increasing sequence $n_i,i\geq 1$ such that for every $k\geq 1$, $\nu_{n_i}(g_k)$ converges as $i\rightarrow\infty$ to a limit denoted by $\Lambda(g_k)$. Consider a continuous function $g$ on $[0,\infty)\times[0,1]$ that vanishes outside a compact set, say $[0,T]\times[0,1]$. Fix $\epsilon>0$ and consider $g_k$ such that $\left\|g-g_k\right\|_{\infty}<\epsilon$ and $\mbox{supp } g_k \subset [0,T+\epsilon]\times[0,1]$. There exists $g_p\geq 0$ such that $g_p \geq \tun_{[0,T+\epsilon]\times[0,1]}$. We write:
\begin{eqnarray*}
\big| \nu_{n_i}(g)- \nu_{n_j}(g)\big| &\leq& \big| \nu_{n_i}(g)- \nu_{n_i}(g_k)\big| + \big| \nu_{n_i}(g_k)- \nu_{n_j}(g_k)\big|+\big| \nu_{n_j}(g_k)- \nu_{n_j}(g)\big|\\
&\leq& \left\|g-g_k\right\|_{\infty}\big(\nu_{n_i}(g_p)+\nu_{n_j}(g_p)\big)+\big| \nu_{n_i}(g_k)- \nu_{n_j}(g_k)\big|
\end{eqnarray*}
Taking $i,j$ large enough, the left side becomes smaller than $\epsilon (3\Lambda(g_p)+ 1)$, so that $(\nu_{n_i}(g),i\geq 1)$ is a Cauchy sequence. We denote by $\Lambda(g)$ the limit. We have defined a positive linear map $\Lambda$ on the set of continuous functions on $[0,\infty)\times[0,1]$ with compact support. By the Riesz representation theorem, there exists a Borel measure $\nu$ on $[0,\infty)\times[0,1]$, finite on compact subsets, such that $\Lambda(g)=\nu(g)$. We then define $m(dt,dx):=\frac{1}{x(1-x)}\nu(dt,dx)$ on $[0,\infty)\times(0,1)$, which clearly belongs to $\bbM$. It is easily checked that $d(m_n,m)$ goes to $0$ as $n\rightarrow\infty$.\cqfd
\end{proof}

\noindent We work in \Mdeux, since the arguments are very similar in \Munw. For $A\subset\bbM$ to be relatively compact, it is necessary and sufficient that for all $k\geq 1$, $\sup_{m\in A} |\int \varphi_k dm| <\infty$. To show tightness of $\zetaun$ under $\bbQN_{\bmuN}$, it suffices to find for every $\epsilon > 0$ a sequence $\lambda_k > 0$ such that
\begin{equation}\label{Eq:Compactness}
\sup_{N\geq 1} \bbQN_{\bmuN}\Big(\Big|\int_{[0,\infty)\times(0,1)}\!\!\!\!\! \varphi_k(t,x)\zetaun(dt,dx)\Big| > \lambda_k\Big) < \epsilon\ 2^{-k}.
\end{equation}
For any two Riemann-integrable functions $g,h$ we define
\begin{equation}\label{Eq:DefIntegrals}
\langle g,h \rangle := \int_{[0,1]}g(x)h(x)dx\;\;\;,\;\;\;\langle g,h \rangle_{\tiN} := \frac{1}{2N}\sum_{k=0}^{2N}g(k_{\tiN})h(k_{\tiN}).
\end{equation}
Notice that $\llangle\cdot\rrangle$ denotes the bracket of a martingale. For every $k\geq 1$, the function $\varphi_k$ introduced at the beginning of this subsection, is compactly supported in $[0,\infty)\times[0,1]$ and vanishes for $x\in\{0,1\}$. Furthermore $\partial_t \varphi_k$ and $\partial^2_x\varphi_k$ exist and are continuous. Using (\ref{Eq:DefMde}), we see that for all $N\geq 1$ the process
\begin{equation}\label{Eq:Mvarphik}
\begin{split}
\MNun_t(\varphi_k)&:=\langle \rmhun_t,\varphi_k(t,\cdot) \rangle_{\tiN} - \langle \rmhun_0,\varphi_k(0,\cdot) \rangle_{\tiN} - \frac{(2N)^2}{2}\int_0^t \langle \Delta\rmhun_s,\varphi_k(s,\cdot) \rangle_{\tiN} ds\\
&\;\;\; - \int_0^t \langle \rmhun_s,\partial_s\varphi_k(s,\cdot) \rangle_{\tiN}ds- \frac{2}{\sqrt{2N}}\int_0^t \big\langle (\pN(\cdot)-\frac{(2N)^2}{2})\tun_{\{\Delta \rmhun_s(\cdot)\ne 0\}},\varphi_k(s,\cdot) \big\rangle_{\tiN}ds\\
&\;\;\;-\int_{[0,t]\times(0,1)}\!\!\!\!\!\varphi_k(s,x)\zetaun(ds,dx)
\end{split}
\end{equation}
is a martingale under $\bbQN_{\bmuN}$.
\begin{lemma}\label{LemmaCVMgales}
For every $k\geq 1$, the sequence of processes $(\MNun_t(\varphi_k),t\geq 0)$ is tight in $\bbD([0,\infty),\bbR)$ and any limit belongs to $\bbC([0,\infty),\bbR)$.
\end{lemma}
\begin{proof}
The bracket of the martingale is given by
\begin{eqnarray*}
\llangle \MNun(\varphi_k)\rrangle_t &=&\frac{4}{(2N)^2}\int_0^t \big\langle \pN(\cdot)\tun_{\{\Delta \rmhun_s(\cdot)>0\}}+\qN(\cdot)\tun_{\{\Delta \rmhun_s(\cdot)<0;\rmhun_s(\cdot)>\rmhde_s(\cdot)\}},\varphi_k^2(s,\cdot) \big\rangle_{\tiN} ds.
\end{eqnarray*}
Since $\bbQN_{\bmuN}$-a.s.~for every $t\geq 0$ we have $\llangle \MNun(\varphi_k)\rrangle_t\leq 4\left\|\varphi_k\right\|^2t$, we deduce that the sequence of bracket processes $(\llangle \MNun(\varphi_k)\rrangle_t,t\geq 0)$ is $\bbC$-tight. Theorem VI.4.13 in Jacod and Shiryaev~\cite{JacodShiryaev} thus implies that the sequence of martingales $(\MNun_t(\varphi_k),t\geq 0),N\geq 1$ is $\bbD$-tight. Since the jumps of these martingales are of vanishing magnitude - at most $2\left\|\varphi_k\right\|(2N)^{-\frac{3}{2}}$ - we deduce that they are actually $\bbC$-tight.\cqfd
\end{proof}
Fix $\epsilon>0$. For every $k\geq 1$, let $T_k >0$ be such that supp $\varphi_k \subset [0,T_k]\times(0,1)$. As a consequence of the tightness of $\rmhun$ and $\MNun(\varphi_k)$, we deduce that for every $k\geq 1$ there exists $\alpha_k>0$ such that
\[\sup_{N\geq 1}\bbQN_{\bmuN}\Big(\sup_{t\in[0,T_k],x\in[0,1]}|\rmhun_t(x)|>\alpha_k\Big) < \epsilon\ 2^{-k-1} \;\;\;,\;\;\;\sup_{N\geq 1}\bbQN_{\bmuN}\Big(|\MNun_{T_k}(\varphi_k)|>\alpha_k\Big) < \epsilon\ 2^{-k-1}.\]
Since $\varphi_k(s,0)=\varphi_k(s,1)=0$ we have
\[ \big|\langle \Delta\rmhun_s,\varphi_k(s,\cdot) \rangle_{\tiN}\big| =\big|\langle \rmhun_s,\Delta\varphi_k(s,\cdot) \rangle_{\tiN} \big| \leq (2N)^{-2}\sup_{[0,1]}|\partial^2_x\varphi(s,\cdot)|\sup_{[0,1]}|\rmhun_s|.\]
Using (\ref{Eq:Mvarphik}) at time $T_k$, we obtain that $\big|\int_{[0,\infty)\times(0,1)}\varphi_k(s,x)\zetaun(ds,dx)\big|$ is bounded by
\[ |\MNun_{T_k}(\varphi_k)|+\sup_{\substack{t\in[0,T_k]\\x\in[0,1]}}|\rmhun_t(x)|\big(2\left\|\varphi_k\right\|+\frac{T_k}{2}\left\|\partial_x^2\varphi_k\right\|+T_k\left\|\partial_s\varphi_k\right\|\big)+2\,T_k\!\sup_{\substack{N\geq 1\\x\in[0,1]}}\frac{|\pN(x)-\frac{(2N)^2}{2}|}{\sqrt{2N}}\left\|\varphi_k\right\|.\]
We deduce the existence of a sequence $\lambda_k$ satisfying (\ref{Eq:Compactness}). This ensures the tightness of $\zetaun$ under $\bbQN_{\bmuN}$. The proof works verbatim for $\zetade$.

\section{Identification of the limit}\label{Section:Limit}
We first give rigorous definitions of the stochastic PDEs of the statements then we complete the proofs of Theorems \ref{ThCVModel1}, \ref{ThCVModel1w} and \ref{ThCVModel2}. Recall Assumption \ref{Assumption} on the asymmetry $\sigma$. We start with the \textsc{RSHE}. Let $\ccC^{2}_c\big((0,1)\big)$ denote the space of compactly supported functions on $(0,1)$ with a continuous second derivative.
\begin{definition}(Nualart-Pardoux~\cite{NualartPardoux92})\label{Def:RSHE}
Consider a probability space $(\Omega,F,P)$ on which are defined a process $(\rmh_t,t\geq 0)$ in $\bbC([0,\infty),\cCMunw)$ and a random measure $\zeta\in\bbM$. We also assume that there exists a cylindrical Wiener process $(W_t,t\geq 0)$ on $\rL^2\big((0,1)\big)$ which is adapted to the natural filtration generated by $\rmh$ and $\zeta$. We say that $(\rmh,\zeta)$ is a solution to \textsc{RSHE} with initial condition $\nu$ if\begin{enumerate}[(i)]
\item The $\cCMunw$-valued random variable $\rmh_0$ has law $\nu$ and is independent of the cylindrical Wiener processes,
\item For any $\varphi\in\ccC^{2}_c\big((0,1)\big)$ we have $P$-a.s.:
\[\langle \rmh_t,\varphi \rangle = \langle \rmh_0,\varphi \rangle + t \langle \sigma,\varphi \rangle + \frac{1}{2}\int_0^t \langle \rmh_s,\varphi'' \rangle ds +\int_{[0,t]\times(0,1)}\!\!\! \varphi(x)\zeta(ds,dx)+\langle\varphi,W_t\rangle,\]
\item $P$-a.s., $\int_{[0,\infty)\times(0,1)} \rmh_t(x)\zeta(dt,dx)=0$.
\end{enumerate}
\end{definition}
The definition of the \textsc{SHE} is even simpler: it suffices to remove the random measure from this definition, so that we do not state it. It turns out that existence and uniqueness hold for these two stochastic PDEs, see Da Prato and Zabczyk~\cite{DaPratoZabczykBook} and Nualart and Pardoux~\cite{NualartPardoux92}. Let us now state our definition of \textsc{Pair \scriptsize{of}\normalsize\ RSHEs}.
\begin{definition}\label{Def:SHEs}
Consider a probability space $(\Omega,F,P)$ on which are defined a process $(\rmh_t,t\geq 0)$ in $\bbC([0,\infty),\cCMde)$ and two random measures $\zetaun,\zetade\in\bbM$. We also assume that there exist two independent cylindrical Wiener processes $(\Wun_t,t\geq 0)$, $(\Wde_t,t\geq 0)$ on $\rL^2\big((0,1)\big)$ which are adapted to the natural filtration generated by $\rmh$, $\zetaun$ and $\zetade$. We say that $(\rmh,\zetaun,\zetade)$ is a solution to \textsc{Pair \scriptsize{of}\normalsize\ RSHEs} with initial condition $\nu$ if\begin{enumerate}[(i)]
\item The $\cCMde$-valued random variable $\rmh_0$ has law $\nu$ and is independent of the cylindrical Wiener processes,
\item For any $\varphi\in\ccC^{2}_c\big((0,1)\big)$ we have for every $i\in\{1,2\}$
\[\langle \rmhi_t,\varphi \rangle = \langle \rmhi_0,\varphi \rangle -(-1)^i t \langle \sigma,\varphi \rangle + \frac{1}{2}\int_0^t \langle \rmhi_s,\varphi'' \rangle ds -(-1)^i\int_{[0,t]\times(0,1)}\!\!\!\!\!\!\!\! \varphi(x)\zetai(ds,dx)+\langle\varphi,\Wi_t\rangle,\]
\item $P$-a.s., $\int_{[0,\infty)\times(0,1)} \big(\rmhun_t(x)-\rmhde_t(x)\big)\big(\zetaun(dt,dx)+\zetade(dt,dx)\big)=0$.
\end{enumerate}
\end{definition}
\noindent We now work with the canonical process $(\rmh,\zetaun,\zetade)$ on $\bbD([0,\infty),\cCMde)\times\bbM\times\bbM$ endowed with the natural filtration $\ccF_t$ generated by $\rmh_s$ for all $s\in[0,t]$ and $\zetaun(B),\zetade(B)$ for all Borel sets $B\subset [0,t]\times(0,1)$. We state a martingale problem that allows to identify a solution to \textsc{Pair \scriptsize{of}\normalsize\ RSHEs}, a similar statement holds for the two other stochastic PDEs.
\begin{proposition}\label{Prop:MgalePb}
Let $\nu$ be a probability measure on $\cCMde$. Suppose that $\bbQ_{\nu}$ is a probability measure on $\bbC([0,\infty),\cCMde)\times\bbM\times\bbM$ under which:\begin{enumerate}[(a)]
\item $\rmh_0$ is distributed according to $\nu$,
\item For every $\varphi,\psi\in\ccC_c^2\big((0,1)\big)$ the processes
\begin{align*}
\Mi_t(\varphi)&:=\langle \rmhi_t,\varphi \rangle - \langle \rmhi_0,\varphi \rangle +(-1)^i t \langle \sigma,\varphi \rangle - \frac{1}{2}\int_0^t \langle \rmhi_s,\varphi'' \rangle ds +(-1)^i\int_{[0,t]\times(0,1)}\!\!\!\!\!\! \varphi(x)\zetai(ds,dx),\\
\Li_t(\varphi)&:=\Mi_t(\varphi)^2-t\langle \varphi,\varphi \rangle,\\
K_t(\varphi,\psi)&:= \Mun_t(\varphi)\Mde_t(\psi)
\end{align*}
are continuous $\ccF_t$-martingales,
\item We have $\int_{[0,\infty)\times(0,1)} \big(\rmhun_t(x)-\rmhde_t(x)\big)\big(\zetaun(dt,dx)+\zetade(dt,dx)\big)=0$.
\end{enumerate}
Then, $(\rmh,\zetaun,\zetade)$ is a solution to \textsc{Pair \scriptsize{of}\normalsize\ RSHEs} with initial condition $\nu$.
\end{proposition}
\begin{proof}
The arguments are standard. Property (iii) follows from (c). By density of $\ccC_c^2\big((0,1)\big)$ in $\rL^2\big((0,1)\big)$, for every $t>0$ we can extend the map $\varphi\mapsto t^{-\frac{1}{2}}\Mi_t(\varphi)$ into an isometry from $\rL^2\big((0,1),dx\big)$ into $\rL^2\big(\bbC\times\bbM\times\bbM,\bbQ_\nu\big)$. Then for every $\varphi\in\rL^2\big((0,1),dx\big)$, the process $(\Mi_t(\varphi),t\geq 0)$ is a Brownian motion with variance $t\left\|\varphi\right\|_{\rL^2}^2$ which is adapted to the filtration $(\ccF_t,t\geq 0)$ so that it is independent of $\ccF_0$. Consider the orthonormal basis $(\epsilon_n,n\geq 0)$ of $\rL^2\big((0,1)\big)$ introduced at the beginning of the proof of Proposition \ref{Prop:UnifBoundLyonsZheng}, and define $\Wi_t := \sum_{n\geq 0}\Mi_t(\epsilon_n)\epsilon_n$. This random variable takes values in a distribution space. For each $i\in\{1,2\}$, this is a cylindrical Wiener process on $\rL^2\big((0,1)\big)$. Property (ii) of Definition \ref{Def:SHEs} follows. The fact that $(K_t(\varphi,\psi),t\geq 0)$ is a martingale implies that the covariation of the Brownian motions $(\Mun_t(\varphi),t\geq 0)$ and $(\Mde_t(\psi),t\geq 0)$ vanishes so that they are independent. Consequently, the Gaussian processes $\Wun$ and $\Wde$ are independent. Finally, the independence of these Wiener processes from $\rmh_0$ follows from the independence of the $(M_t(\varphi),t\geq 0)$'s from $\ccF_0$. Property (i) follows.\cqfd 
\end{proof}

\subsection{Conclusion of the proof of Theorems \ref{ThCVModel1}, \ref{ThCVModel1w} and \ref{ThCVModel2}}
From now on, we restrict ourselves to \Mdeux\ as this is the most involved setting. The proof is very similar for the other two models. We have already obtained tightness of the sequence $\bbQN_\bmuN,N\geq 1$. Consider a convergent subsequence, which for simplicity we still denote $\bbQN_\bmuN,N\geq 1$, and let $\bbQ'$ be its limit which is supported by $\bbC([0,\infty),\cCMde)\times\bbM\times\bbM$. To complete the proof of Theorem \ref{ThCVModel2} we only need to show that the conditions of Proposition \ref{Prop:MgalePb} are fulfilled under $\bbQ'$ when the initial condition $\nu$ is taken to be $\bQMde$.\vspace{-10pt}
\paragraph{Martingale relations.} We start with the proofs of the martingale relations on $M^{(i)}_t(\varphi)$ and $L^{(i)}_t(\varphi)$. By symmetry, it suffices to consider $i=1$. The main idea of the proof is to consider discrete versions $\MNun_t(\varphi)$ and $\LNun_t(\varphi)$ of the martingales $\Mun_t(\varphi)$ and $\Lun_t(\varphi)$, and to show that the $\rL^2$ norms of the differences vanish as $N\rightarrow\infty$. Fix a map $\varphi$ in $\ccC_c^2\big((0,1)\big)$, by linearity we can assume that $\varphi\geq 0$. Recall the notation (\ref{Eq:DefIntegrals}). We define
\begin{eqnarray*}
\MNun_t(\varphi)&:=&\langle \rmhun_t,\varphi \rangle_{\tiN} - \langle \rmhun_0,\varphi \rangle_{\tiN}- \frac{2}{\sqrt{2N}}\int_0^t \Big\langle \Big(\pN(\cdot)-\frac{(2N)^2}{2}\Big)\tun_{\{\Delta \rmhun_s(\cdot)\ne 0\}},\varphi \Big\rangle_{\tiN}ds\\
&&\!\!\! - \frac{(2N)^2}{2}\int_0^t \langle \Delta\rmhun_s,\varphi \rangle_{\tiN} ds - \int_{[0,t]\times(0,1)}\!\!\!\!\!\!\varphi(x)\zetaun(ds,dx)
\end{eqnarray*}
and
\begin{eqnarray*}
\LNun_t(\varphi)\!\!\!&:=&\!\!\!\! \big(\MNun_t(\varphi)\big)^2\!-\frac{4}{(2N)^2}\int_0^t \langle \pN(\cdot)\tun_{\{\Delta \rmhun_s(.)>0\}}+\qN(\cdot)\tun_{\{\Delta \rmhun_s(\cdot)<0;\rmhun_s(\cdot)>\rmhde_s(\cdot)\}},\varphi^2 \rangle_{\tiN} ds.
\end{eqnarray*}
Using the stochastic differential equations (\ref{Eq:DefMde}), it is elementary to check that both processes are $\ccF_t$-martingales under $\bbQN_{\bmuN}$. Recall the definition of $\Mun_t(\varphi)$ and $\Lun_t(\varphi)$, which are well-defined random variables on the space $\bbD\times\bbM\times\bbM$.
\begin{lemma}\label{LemmaIdLimit}
For every $t\geq 0$, we have:\begin{itemize}
\item[(a)]$\sup\limits_{N\geq 1}\bbQN_{\bmuN}\Big[\big|\Mun_t(\varphi)\big|^4 \Big] <\infty \;\;,\;\;\sup\limits_{N\geq 1}\bbQN_{\bmuN}\big[\big|\Lun_t(\varphi)\big|^2\big] <\infty$,
\item[(b)]$\bbQN_{\bmuN}\Big[|\MNun_t(\varphi)-\Mun_t(\varphi)|^2\Big] \underset{N\rightarrow\infty}{\longrightarrow} 0\;\;,\;\;\bbQN_{\bmuN}\big[|\LNun_t(\varphi)-\Lun_t(\varphi)|^2\big] \underset{N\rightarrow\infty}{\longrightarrow} 0.$
\end{itemize}
\end{lemma}
\begin{proof}
The bound on the second moment of $\Lun_t(\varphi)$ follows from the bound on the fourth moment of $\Mun_t(\varphi)$, so we only need to bound this term uniformly to obtain (a).\\
The Burkholder-Davis-Gundy inequality (we refer to Appendix \ref{SectionProofTightnessMun} for notations) implies
\begin{eqnarray*}
\bbQN_\bmuN\Big[ \Big( \MNun_t(\varphi)\Big)^4 \Big] &\leq& \tcBDG(4)^4\bbQN_\bmuN\Big[\big[\MNun(\varphi)\big]_t^2\Big]\\
&\leq& \frac{16\tcBDG(4)^4}{(2N)^6} \left\|\varphi\right\|^4\bbQN_\bmuN\Big[\Big(\sum_{k=1}^{2N-1}\cLNun_t(k_{\tiN})+\cRNun_t(k_{\tiN})\Big)^2\Big].
\end{eqnarray*}
On the right we have the second moment of a Poisson random variable with mean $t(2N)^3$, this is equal to $t(2N)^3+t^2(2N)^6$. Consequently
\[ \sup_{N\geq 1}\bbQN_\bmuN\Big[\Big(\MNun_t(\varphi)\Big)^4\Big] < \infty.\]
To prove that the same holds for $\Mun_t(\varphi)$ instead of $\MNun_t(\varphi)$, it suffices to bound the fourth moment of $|\Mun_t(\varphi)-\MNun_t(\varphi)|$ uniformly. Actually, let us prove that the fourth moment of this quantity vanishes when $N$ goes to $\infty$. We have
\begin{eqnarray*}
\Mun_t(\varphi)-\MNun_t(\varphi)&=&\langle \rmhun_t,\varphi \rangle-\langle \rmhun_t,\varphi \rangle_{\tiN} - \langle \rmhun_0,\varphi \rangle+ \langle \rmhun_0,\varphi \rangle_{\tiN}\\
&& - \frac{1}{2}\int_0^t \langle \rmhun_s,\varphi'' \rangle ds + \frac{(2N)^2}{2}\int_0^t \langle \Delta\rmhun_s,\varphi \rangle_{\tiN} ds\\
&& - t \langle \sigma,\varphi \rangle + \frac{2}{\sqrt{2N}}\int_0^t \Big\langle \Big(\pN(\cdot)-\frac{(2N)^2}{2}\Big)\tun_{\{\Delta \rmhun_s(\cdot)\ne 0\}},\varphi \Big\rangle_{\tiN}ds.
\end{eqnarray*}
The fourth moments of the terms in the first two lines can be shown to vanish using standard arguments together with the uniform bound of the exponential moments of $\left\|h\right\|_{\cC}$ under the stationary measure $\bmuN$ obtained in Theorem \ref{ThCVInvMeasure}. The term on the third line is more involved and requires Theorem \ref{ThSuperExpo}. We do not provide the details since we will apply this theorem for a very similar term below. Consequently we have uniform bounds on the fourth moment of $\Mun_t(\varphi)$ so that statement (a) of the lemma follows. The above calculations also prove (b) for $\Mun_t(\varphi)$.\\ 
It remains to prove (b) for $\Lun_t(\varphi)$. First observe that the Cauchy-Schwarz inequality yields:
\[ \bbQN_{\bmuN}\big[|(\MNun_t(\varphi))^2-(\Mun_t(\varphi))^2|^2\big] \!\leq\! \bbQN_{\bmuN}\big[|\MNun_t(\varphi)-\Mun_t(\varphi)|^4\big]^{\frac{1}{2}}\bbQN_{\bmuN}\big[|\MNun_t(\varphi)+\Mun_t(\varphi)|^4\big]^{\frac{1}{2}}.\]
The arguments above show that the first term on the right hand side vanishes as $N\rightarrow\infty$ while the second is uniformly bounded. Consequently the left hand side vanishes as $N\rightarrow\infty$.
Let us define 
\[ A_{\tiN} := \Big|\frac{4}{(2N)^2}\int_0^t \langle \pN(\cdot)\tun_{\{\Delta \rmhun_s(\cdot)>0\}}+\qN(\cdot)\tun_{\{\Delta \rmhun_s(\cdot)<0;\rmhun_s(\cdot)>\rmhde_s(\cdot)\}},\varphi^2 \rangle_{\tiN} ds-t\langle \varphi,\varphi \rangle\Big|.\]
To complete the proof of (b) for $\Lun_t(\varphi)$, we only need to show that $\bbQN_\bmuN[A_{\tiN}^2]\rightarrow 0$ as $N\rightarrow\infty$. The random variable $A_{\tiN}$ is bounded by a deterministic constant uniformly in $N\geq 1$ so that it suffices to prove its convergence in probability to $0$. Observe that
\begin{eqnarray*}
A_{\tiN}&\leq&\Big|\int_0^t \Big\langle \frac{4}{(2N)^2}\Big(\pN(\cdot)\tun_{\{\Delta \rmhun_s(\cdot)>0\}}+\qN(\cdot)\tun_{\{\Delta \rmhun_s(\cdot)<0\}}\Big)-1,\varphi^2 \Big\rangle_{\tiN} ds\Big|\\
&&\!\!\!+ t\sum_{k=0}^{2N-1}\Big|\frac{1}{2N}\varphi^2(k_{\tiN})-\int_{\frac{k}{2N}}^{\frac{k+1}{2N}}\varphi^2(u)du\Big|+\frac{2}{(2N)^{\frac{3}{2}}}\int_{[0,t]\times(0,1)}\!\!\!\varphi^2(x)\zetaun(dt,dx).
\end{eqnarray*}
The second term corresponds to the approximation of the Riemann integral, it vanishes as $N\rightarrow\infty$. To show that the third term vanishes in $\bbQN_{\bmuN}$-probability as $N\rightarrow\infty$ we argue as follows: for all rational values $p,q$ such that $p > t$, the random variable $\int_{[0,t]\times(0,1)}\!\!\varphi^2(x)\zetaun(ds,dx)$ is smaller than
\[ \int_{[0,\infty)\times(0,1)}\rho_{p,q}(s)\varphi^2(x)\zetaun(ds,dx) \]
which converges in distribution, by the convergence of the measure $\zetaun$. To bound the first term we apply Theorem \ref{ThSuperExpo} as follows. Recall the notation of Section \ref{SubsectionSuperExpo}. Let $\Phi:\eta\mapsto2\etaun(1)(1-\etaun(2))+2(1-\etaun(1))\etaun(2)$ and observe that $\tilde{\Phi}(a)=4a(1-a)$. Recall that $\tau_k$ denotes the shift by $k$ introduced in Subsection \ref{SubsectionSuperExpo}. Then we write
\begin{eqnarray*}
&&\int_0^t \Big\langle \frac{4}{(2N)^2}\Big(\pN(\cdot)\tun_{\{\Delta \rmhun_s(\cdot)>0\}}+\qN(\cdot)\tun_{\{\Delta \rmhun_s(\cdot)<0\}}\Big)-1,\varphi^2 \Big\rangle_{\tiN} ds\\
&=&\int_0^t \Big\langle \Big(\frac{4\,\pN(\cdot)}{(2N)^2}-2\Big)\tun_{\{\Delta \rmhun_s(\cdot)>0\}}+\Big(\frac{4\,\qN(\cdot)}{(2N)^2}-2\Big)\tun_{\{\Delta \rmhun_s(\cdot)<0\}},\varphi^2 \Big\rangle_{\tiN} ds\\
&&\hspace{-12pt}+\!\int_0^t\sum_{k=0}^{2N-1}\Big(\Phi(\tau_k\upeta_s)-1\Big)\frac{1}{2N}\varphi^2\Big(\frac{k+1}{2N}\Big)ds.
\end{eqnarray*}
The hypotheses made on $\pN,\qN$ imply that the $\bbQN_{\bmuN}$ expectation of the absolute value of the first term on the right goes to $0$ as $N\rightarrow\infty$. To deal with the second term on the right, we introduce $\epsilon >0$ and we write
\begin{equation}\label{Eq:BoundVaradhan}\begin{split}
&\int_0^t\sum_{k=0}^{2N-1}\Big(\Phi(\tau_k\upeta_s)-1\Big)\frac{1}{2N}\varphi^2\Big(\frac{k+1}{2N}\Big)ds\\
&= \int_0^t\sum_{k=0}^{2N-1}\Big(\frac{1}{2\epsilon N+1}\sum_{j:|j-k|\leq \epsilon N}\Phi(\tau_j\upeta_s)-\tilde{\Phi}\Big(\sum_{j:|j-k|\leq \epsilon N}\frac{1}{2\epsilon N+1}\upetaun_s(j)\Big)\Big)\frac{1}{2N}\varphi^2\Big(\frac{k+1}{2N}\Big)ds\\
&+\int_0^t\sum_{k=0}^{2N-1}\bigg(\tilde{\Phi}\Big(\sum_{j:|j-k|\leq \epsilon N}\frac{1}{2\epsilon N+1}\,\upetaun_s(j)\Big)-1\bigg)\frac{1}{2N}\varphi^2\Big(\frac{k+1}{2N}\Big)ds.
\end{split}
\end{equation}
There is a slight abuse of notation in this formula: one should take the integer part of $\epsilon N$ everywhere this term appears. Notice also that all our indices are taken modulo $2N$. For $\epsilon$ small enough, Theorem \ref{ThSuperExpo} ensures that the first term on the right of (\ref{Eq:BoundVaradhan}) vanishes in $\bbQN_\bmuN$-probability as $N\rightarrow\infty$. Now observe that
$$\sum_{j:|j-k|\leq \epsilon N}\frac{1}{2\epsilon N+1}\,\upetaun_s(j)=\frac{1}{2} + \frac{\sqrt{2N}}{2}\,\frac{\rmhun_s\Big(\frac{k}{2N}+\frac{\epsilon N}{2N}\Big)-\rmhun_s\Big(\frac{k}{2N}-\frac{\epsilon N}{2N}\Big)}{2\epsilon N+1}.$$
Since $\tilde{\Phi}(a)=4a(1-a)$, the second term on the right of (\ref{Eq:BoundVaradhan}) can be bounded by
$$t\left\|\varphi^2\right\|\frac{2N}{(2\epsilon N +1)^2}\sup_{s\in[0,T],x\in[0,1]}\big|\rmhun_s(x+2\epsilon)-\rmhun_s(x)\big|^2.$$
For any fixed value $\epsilon$, the $\bbQN_\bmuN$-expectation of the supremum is uniformly bounded in $N\geq 1$ by Assertion (\ref{Tightness2}) of the proof of the tightness stated at the beginning of Section \ref{Section:Tightness}, consequently the whole quantity vanishes in $\bbQN_\bmuN$-probability as $N\rightarrow\infty$.\cqfd
\end{proof}
\noindent Fix $s\geq 0$ and let $\gamma_s:\bbD([0,\infty),\cCMde)\rightarrow\bbR$ be a bounded measurable function, measurable with respect to the $\sigma$-field generated by $\rmh_r,r\in[0,s]$ and continuous at any point in $\bbC([0,\infty),\cCMde)$. We then set $G_s$ to be the following bounded measurable map from $\bbD\times\bbM\times\bbM$ into $\bbR$:
\[ G_s(\rmh,\zetaun,\zetade) = \gamma_s(\rmh)\prod_{j=1}^{n}
\alpha_j\Big(\int_{[0,s]\times(0,1)}a_j(r,x)\zetaun(dr,dx)\Big)
\beta_j\Big(\int_{[0,s]\times(0,1)}b_j(r,x)\zetade(dr,dx)\Big),\]
where $n\geq 1$, $\alpha_j,\beta_j$ are bounded continuous functions on $\bbR$, and $a_j,b_j$ are non-negative compactly supported functions from $[0,\infty)\times(0,1)$ into $\bbR$ that admit a continuous derivative in time and a continuous second derivative in space.
\begin{lemma}\label{LemmaCVMtNMt}
For all $t\geq s$, the distribution of $\MNun_t(\varphi)G_s$ under $\bbQN_\bmuN$ converges to the distribution of $\Mun_t(\varphi)G_s$ under $\bbQ'$, and similarly for $\LNun_t(\varphi)G_s$ and $\Lun_t(\varphi)G_s$.
\end{lemma}
\noindent We postpone the proof of this lemma to the end of this subsection. Using Lemma \ref{LemmaIdLimit}, Lemma \ref{LemmaCVMtNMt}, and Lemma \ref{Lemma:CVUnifInteg}, we deduce that for all $t\geq s$:
\[\bbQN_{\bmuN}\Big[\MNun_t(\varphi)G_s \Big] \underset{N\rightarrow\infty}{\longrightarrow} \bbQ'\Big[\Mun_t(\varphi)G_s\Big]\;\;,\;\;\bbQN_{\bmuN}\big[\LNun_t(\varphi)G_s\big] \underset{N\rightarrow\infty}{\longrightarrow} \bbQ'\big[\Lun_t(\varphi)G_s\big].\]
Taking the limit as $N\rightarrow\infty$ in the following martingale identities:
$$ \bbQN_{\bmuN}\big[\MNun_t(\varphi)G_s\big]=\bbQN_{\bmuN}\big[\MNun_s(\varphi)G_s\big]\;\;,\;\;\bbQN_{\bmuN}\big[\LNun_t(\varphi)G_s\big]=\bbQN_{\bmuN}\big[\LNun_s(\varphi)G_s\big],$$
we therefore obtain
$$ \bbQ'\big[\Mun_t(\varphi)G_s\big]=\bbQ'\big[\Mun_s(\varphi)G_s\big]\;\;,\;\;\bbQ'\big[\Lun_t(\varphi)G_s\big]=\bbQ'\big[\Lun_s(\varphi)G_s\big].$$
Since the indicator of any closed set of the form $[u,v]\times[a,b] \subset [0,s]\times(0,1)$ can be approximated by functions of the type $a_j$ that appear in $G_s$, a classical argument based on the Monotone Class Theorem shows that $\Mun_t(\varphi)$ and $\Lun_t(\varphi)$ are $\ccF_t$-martingales under $\bbQ'$.

\medskip

\noindent We now prove that $K_t(\varphi,\psi)$ is an $\ccF_t$-martingale under $\bbQ'$. We know that the process $K_t(\varphi,\psi) - \llangle \Mun(\varphi),\Mde(\psi)\rrangle_t$ is an $\ccF_t$-martingale under $\bbQ'$. Since
\begin{equation}\label{EqBrackets}
\big\llangle \Mun(\varphi),\Mde(\psi)\big\rrangle_t = \frac{1}{4}\Big(\big\llangle \Mun(\varphi)+\Mde(\psi)\big\rrangle_t-\big\llangle \Mun(\varphi)-\Mde(\psi)\big\rrangle_t\Big),
\end{equation}
it suffices to show that the two brackets on the right are equal under $\bbQ'$. Using (\ref{Eq:DefMde}), we easily check that $(\MNun_t(\varphi)+\MNde_t(\psi))^2 - \llangle \MNun(\varphi)\rrangle_t-\llangle \MNde(\psi)\rrangle_t$ is an $\ccF_t$-martingale under $\bbQN_{\bmuN}$. Therefore, the same convergence arguments as above show that $(\Mun_t(\varphi)+\Mde_t(\psi))^2-t(\langle\varphi,\varphi\rangle+\langle\psi,\psi\rangle)$ is an $\ccF_t$-martingale under $\bbQ'$. Similarly, we obtain that $(\Mun_t(\varphi)-\Mde_t(\psi))^2-t(\langle\varphi,\varphi\rangle+\langle\psi,\psi\rangle)$ is an $\ccF_t$-martingale under $\bbQ'$ so that (\ref{EqBrackets}) vanishes under $\bbQ'$. This completes the proof of the martingale relations.
\paragraph{Support condition.} Let us show that for all $T>0$ and all $a<b \in (0,1)$ we have $\bbQ'$-a.s.
$$\int_{[0,T]\times(a,b)}\big(\rmhun_t(x)-\rmhde_t(x)\big)(\zetaun+\zetade)(dt,dx)=0.$$
Fix $a<b \in (0,1)$ and $T>0$. Let $\psi$ be a non-negative continuous function with compact support in $[0,\infty)\times(0,1)$ and such that $\psi(t,x)=1$ for all $(t,x)\in[0,T]\times[a,b]$. Then we introduce $F:\bbD\times\bbM\times\bbM\rightarrow\bbR$ as follows:
\[ F(\rmh,\zetaun,\zetade):=\int_{[0,\infty)\times(0,1)}\psi(t,x)\big(\rmhun_t(x)-\rmhde_t(x)\big)(\zetaun+\zetade)(dt,dx).\]
\begin{lemma}
The map $F$ is $\bbQ'$-a.s.~continuous.
\end{lemma}
\begin{proof}
Let $(\rmh^n,\zetaunn,\zetaden)$ be a sequence of elements of $\bbD\times\bbM\times\bbM$ that converges to an element $(\rmh,\zetaun,\zetade)$ in $\bbC\times\bbM\times\bbM$. We bound $\big| F(\rmh^n,\zetaunn,\zetaden)-F(\rmh,\zetaun,\zetade)\big|$ by:
\[ \big| F(\rmh^n,\zetaunn,\zetaden)-F(\rmh,\zetaunn,\zetaden)\big| + \big| F(\rmh,\zetaunn,\zetaden)-F(\rmh,\zetaun,\zetade)\big|.\]
The first term is bounded by $\sup_{t\in[0,T]}\left\|\rmh^n_t-\rmh_t\right\|_{\cC}\int\psi(t,x)(\zetaunn+\zetaden)(dt,dx)$. As $n\rightarrow\infty$, the integral converges to $\int\psi(t,x)(\zetaun+\zetade)(dt,dx)$ while the supremum vanishes since $\rmh^n\rightarrow\rmh$ in $\bbD$ and since $\rmh$ belongs to $\bbC$. We deduce that the first term vanishes as $n\rightarrow\infty$. The second term goes to $0$ as $n\rightarrow\infty$ by continuity of the map
\[(\zetaun,\zetade) \mapsto \int_{[0,\infty)\times(0,1)}\psi(t,x)\,(\rmhun_t(x)-\rmhde_t(x))(\zetaun+\zetade)(dt,dx).\]
Since $\bbQ'$ is supported by $\bbC\times\bbM\times\bbM$, this completes the proof.\cqfd
\end{proof}
As a consequence of this lemma, the pushforward of $\bbQN_{\bmuN}$ through $F$ converges weakly to the pushforward of $\bbQ'$ through $F$, and thus, for every $\delta > 0$
\[ \bbQ'\Big(F > \delta\Big) \leq \varliminf\limits_{N\rightarrow\infty}\bbQN_{\bmuN}\Big(F>\delta\Big)
=0.\]
The equality on the right follows from the fact that under $\bbQN_{\bmuN}$, the function $\rmhun-\rmhde$ vanishes on the support of $\zetaun+\zetade$. Finally observe that
\[ F(\rmh,\zetaun,\zetade)\geq \int_{[0,T]\times(a,b)}\big(\rmhun_t(x)-\rmhde_t(x)\big)(\zetaun+\zetade)(dt,dx),\]
so that $\bbQ'$-a.s.~$\int_{[0,T]\times(a,b)}(\rmhun_t(x)-\rmhde_t(x))(\zetaun+\zetade)(dt,dx)\leq\delta$. By taking sequences $T_n\uparrow\infty$, $\delta_n \downarrow 0$, $a_n \downarrow 0$ and $b_n\uparrow 1$ we conclude that $\bbQ'$-a.s.~$\int_{[0,\infty)\times(0,1)}(\rmhun_t(x)-\rmhde_t(x))(\zetaun+\zetade)(dt,dx)=0$. We have proved that $\bbQ'$ fulfills all the conditions of Proposition \ref{Prop:MgalePb}.\vspace{6pt}\\
\textit{Proof of Lemma \ref{LemmaCVMtNMt}.} The proof of Lemma \ref{LemmaCVMgales} ensures that $(\MNun_t(\varphi),t\geq 0)$ under $\bbQN_{\bmuN}$ is tight in $\bbD$ and that any limit is continuous. We first show that any limit has the same distribution as $(\Mun_t(\varphi),t\geq 0)$ under $\bbQ'$. Let us extract a subsequence from $\bbQN_\bmuN, N\geq 1$ such that the sequence of martingales converges, for simplicity we keep the same notation for the subsequence. By the Skorokhod Representation Theorem, there exists a probability space on which is defined a sequence $(\rmhN,\zetaNun,\zetaNde,\MNun(\varphi))$ that converges almost surely to $(\rmhinfty,\zetainftyun,\zetainftyde,\Minftyun(\varphi))$, and such that $(\rmhN,\zetaNun,\zetaNde,\MNun(\varphi))$ has the same distribution as $(\rmh,\zetaun,\zetade,\MNun(\varphi))$ under $\bbQN_\bmuN$. Recall that we have for every $t\geq 0$, 
\begin{align}\label{Eq:MeasureN}
\int_{[0,t]\times(0,1)}\!\!\!\!\!\!\varphi(x)\zetaNun(ds,dx) &= -\MNun_t(\varphi)+\langle \rmhNun_t,\varphi \rangle_{\tiN} - \langle \rmhNun_0,\varphi \rangle_{\tiN}\\\nonumber
&\;\;\;\;\;- \frac{2}{\sqrt{2N}}\int_0^t \Big\langle \Big(\pN(\cdot)-\frac{(2N)^2}{2}\Big)\tun_{\{\Delta \rmhNun_s(\cdot)\ne 0\}},\varphi \Big\rangle_{\tiN}ds\\\nonumber
&\;\;\;\;\;- \frac{(2N)^2}{2}\int_0^t \langle \Delta\rmhNun_s,\varphi \rangle_{\tiN} ds.
\end{align}
Using the arguments in the proof of Lemma \ref{LemmaIdLimit}, we deduce that the left hand side converges in probability to
\begin{equation}\label{Eq:DiffMeasure}
-\Minftyun_t(\varphi)+\langle \rmhinftyun_t,\varphi \rangle - \langle \rmhinftyun_0,\varphi \rangle - t \langle \sigma,\varphi \rangle - \frac{1}{2}\int_0^t \langle \rmhinftyun_s,\varphi'' \rangle ds.
\end{equation}
Up to an extraction, we can assume that the convergence is almost sure. We only need to show that (\ref{Eq:DiffMeasure}) coincides with $\int_{[0,t]\times(0,1)} \varphi(x)\zetainftyun(ds,dx)$. This is not obvious since our topology on $\bbM$ does not ensure continuity of the functional $m\mapsto \int_{[0,\infty)\times(0,1)}f(s,x)m(ds,dx)$ when $f$ is not continuous in time. However, the definition of our topology on $\bbM$ ensures that almost surely, for every pair of rational values $p,q$ the measure $\rho_{p,q}(s)\varphi(x)\zetaNun(ds,dx)$ on $[0,\infty)\times(0,1)$ converges weakly to $\rho_{p,q}(s)\varphi(x)\zetainftyun(ds,dx)$. Since $\varphi$ is non-negative, $[0,t]\times(0,1)$ is a closed subset of $[0,\infty)\times(0,1)$ and $[0,t+\epsilon)\times(0,1)$ is an open subset of $[0,\infty)\times(0,1)$, we obtain that almost surely for all $t,\epsilon > 0$,
\begin{align*}
\varlimsup_{N\rightarrow\infty}\int_{[0,t]\times(0,1)}\!\!\!\!\!\!\varphi(x)\zetaNun(ds,dx) \leq \int_{[0,t]\times(0,1)}\!\!\!\!\!\!\varphi(x)\zetainftyun(ds,dx)
\end{align*}
and
\begin{align*}
\int_{[0,t+\epsilon)\times(0,1)}\!\!\!\!\!\!\varphi(x)\zetainftyun(ds,dx)\leq \varliminf_{N\rightarrow\infty}\int_{[0,t+\epsilon)\times(0,1)}\!\!\!\!\!\!\varphi(x)\zetaNun(ds,dx).
\end{align*}
Consequently almost surely for all $t,\epsilon > 0$,
\[\varlimsup_{N\rightarrow\infty}\int_{[0,t]\times(0,1)}\!\!\!\!\!\!\varphi(x)\zetaNun(ds,dx) \leq \int_{[0,t]\times(0,1)}\!\!\!\!\!\!\varphi(x)\zetainftyun(ds,dx)\leq \varlimsup_{N\rightarrow\infty}\int_{[0,t+\epsilon]\times(0,1)}\!\!\!\!\!\!\varphi(x)\zetaNun(ds,dx).\]
The continuity in time of (\ref{Eq:DiffMeasure}) ensures that as $\epsilon\downarrow 0$, the difference between the rightmost and the leftmost terms in the above inequality tends to zero, so that (\ref{Eq:DiffMeasure}) coincides with $\int_{[0,t]\times(0,1)}\varphi(x)\zetainftyun(ds,dx)$.

\noindent This ensures that the distribution of $\MNun_t(\varphi)$ under $\bbQN_\bmuN$ converges to the distribution of $\Mun_t(\varphi)$ under $\bbQ'$. Recall the expression of $G_s$. To deal with $\MNun(\varphi)G_s$, it suffices to consider the martingales $\MNun(a_j)$ and $\MNde(b_j)$, and to repeat the above arguments in order to show the convergence in probability of
\[\int_{[0,t]\times(0,1)}a_j(r,x)\zetaNun(dr,dx) \;\;\;\;\mbox{ and }\;\;\;\;\int_{[0,t]\times(0,1)}b_j(r,x)\zetaNde(dr,dx)\]
towards
\[\int_{[0,t]\times(0,1)}a_j(r,x)\zetainftyun(dr,dx) \;\;\;\;\mbox{ and }\;\;\;\; \int_{[0,t]\times(0,1)}b_j(r,x)\zetainftyde(dr,dx).\]
Then, one multiplies both sides of (\ref{Eq:MeasureN}) by $G_s(\rmhN,\zetaNun,\zetaNde)$ and passes to the limit as $N\rightarrow\infty$, using the continuity of the map $\gamma_s$ together with the previous convergences. The second part of the statement on $\LNun_t(\varphi)$ and $\Lun_t(\varphi)$ follows from very similar arguments, so we do not provide the details.\cqfd

\appendix

\section{Proof of the large deviation result}\label{AppendixSuperExpo}
This is an adaptation of Kipnis, Olla and Varadhan~\cite{KOV89}.
\subsection{The symmetric case}
We consider \Munw\ in the symmetric case $\pN(\cdot)\!=\!\qN(\cdot)\!=\!(2N)^2/2$. From now on, $\cEN$ denotes $\cENMunw$ and $\bbPN_{\piN}$ is taken to be the measure on $\bbD([0,\infty),\cEN)$ of the process in this symmetric case starting from the invariant measure $\piN$. Recall the expression for $\VNe$. A simple calculation (almost the same as p.120 of~\cite{KOV89}) shows that for all $i \in \lbr 1, 2N\rbr$ and any given $k\geq 1$
\begin{eqnarray*}
&&\Big|\frac{1}{2\epsilon N +1}\!\!\sum_{j:|i-j|\leq \epsilon N}\!\!\!\Phi(\tau_j \eta) -\tilde{\Phi}\Big(\frac{1}{2\epsilon N + 1}\!\!\sum_{j:|i-j|\leq \epsilon N}\!\!\eta(j)\Big)\Big| \leq O\Big(\frac{k}{N}\Big)\\
&+&\frac{1}{2\epsilon N +1}\!\!\sum_{j:|i-j|\leq \epsilon N}\Big| \frac{1}{2k+1}\!\!\sum_{l:|j-l|\leq k}\!\!\Phi(\tau_l \eta) -\tilde{\Phi}\Big(\frac{1}{2k + 1}\!\!\sum_{l:|j-l|\leq k}\!\!\eta(l)\Big) \Big|\\
&+& \frac{\left\|\tilde{\Phi}'\right\|}{(2\epsilon N +1)^2}\sum_{j:|i-j|\leq \epsilon N}\sum_{j':|i-j'|\leq \epsilon N}\frac{1}{2k+1}\,\Big|\!\!\sum_{l:|j'-l|\leq k}\!\!\eta(l) -\!\!\sum_{l:|j-l|\leq k}\!\!\eta(l)\Big|,
\end{eqnarray*}
where the term $O\big(\frac{k}{N}\big)$ is uniform in $i,\eta$ so that its contribution to (\ref{EqObjectiveSuperExpo}) vanishes. The contributions of the second and third term above are dealt with by the following two lemmas. From now on, $f$ is implicitly taken to be non-negative and such that $\piN[f]=1$.
\begin{lemma}
For any $c>0$
$$\varlimsup\limits_{k\rightarrow\infty}\varlimsup\limits_{N\rightarrow\infty}\sup_{f:\DN(f)\leq cN}\frac{1}{N}\sum_{i=1}^{2N}\sum_{\eta\in\cEN}\Big|\frac{1}{2k+1}\!\sum_{j:|i-j|\leq k}\!\!\Phi(\tau_j \eta) -\tilde{\Phi}\Big(\frac{1}{2k + 1}\!\sum_{j:|i-j|\leq k}\!\!\eta(j)\Big) \Big|\piN(\eta)f(\eta)=0.$$
\end{lemma}
\begin{proof}
Fix $N\!\geq\! 1,k\!\in\!\lbr 1,N\rbr$. First observe that we can split the sum over $i$ into two sums: the first over $i\in\{1,\ldots,k\}\cup\{2N\!-\!k\!+\!1,\ldots,2N\}$ and the second over the remaining $i$'s. It is a simple matter to check that the first sum is bounded by a quantity of order $k/N$ so that it vanishes when $N$ goes to infinity, $k$ being fixed. To deal with the second sum we set $\cOk:=\{0,1\}^{2k+1}$ and write
\begin{eqnarray}\label{EqSplitSum}
&&\sup_{f:\DN(f)\leq cN}\frac{1}{N}\sum_{i=k+1}^{2N-k}\sum_{\eta\in\cEN}\Big|\frac{1}{2k+1}\!\sum_{j:|i-j|\leq k}\!\!\Phi(\tau_j \eta) -\tilde{\Phi}\Big(\frac{1}{2k + 1}\!\sum_{j:|i-j|\leq k}\!\!\eta(j)\Big) \Big|\piN(\eta)f(\eta)\nonumber\\
&\leq&\!\!\sup_{f:\DN(f)\leq cN}\frac{1}{N}\sum_{i=k+1}^{2N-k}\sum_{\xi\in\cOk}\sum_{\substack{\eta\in\cEN\\\eta_{|\lbr i-k,i+k\rbr}}=\xi}\Big|\frac{1}{2k+1}\sum_{j=0}^{2k}\Phi(\tau_j \xi) -\tilde{\Phi}\Big(\frac{1}{2k+1}\sum_{j=1}^{2k+1}\xi(j)\Big) \Big|\piN(\eta)f(\eta)\nonumber\\
&+&O\Big(\frac{1}{k}\Big).
\end{eqnarray}
The second term on the right bounds the error we make when we replace $\Phi(\tau_{i-k+j} \eta)$ by $\Phi(\tau_j \xi)$; it vanishes when $N$ and $k$ go to infinity. It remains to bound the first term on the right. To that end, we prove an inequality for the Dirichlet form.\\
Consider the symmetric simple exclusion process on $\cOk$ without wall. The uniform measure on $\cOk$ is reversible so that the Dirichlet form associated with this process is given by
\[ D^*(g) := \frac{1}{2}\sum_{\xi\in\cOk}2^{-(2k+1)}\sum_{j=1}^{2k}\Big(\sqrt{g(\xi^{j,j+1})}-\sqrt{g(\xi)}\Big)^2
\tun_{\{\nabla\xi(j)=1\}},\]
for all maps $g:\cOk\rightarrow\bbR_+$. We introduce, in particular, the map
$$\fk(\xi):=\frac{1}{(2N-2k)}\sum_{i=k+1}^{2N-k}\sum_{\substack{\eta\in\cEN\\\eta_{|\lbr i-k,i+k\rbr}}=\xi}\piN(\eta)f(\eta).$$
Recall that $\piN[f]=1$ and observe that $\sum_{\xi\in \cOk}\fk(\xi)=1$. For any two sequences $a_i,b_i \geq 0$ whose sums are finite, the triangle inequality implies $\big(\sqrt{\sum_i a_i}-\sqrt{\sum_i b_i}\big)^2 \leq \sum_i\big(\sqrt{a_i}-\sqrt{b_i}\big)^2$. This yields
\begin{eqnarray*}
D^*(\fk) &\leq& \frac{2^{-(2k+1)}}{2(2N-2k)}\sum_{\xi\in\cOk}\sum_{i=k+1}^{2N-k}\sum_{\substack{\eta\in\cEN\\\eta_{|\lbr i-k,i+k\rbr}}=\xi}\!\!\!\piN(\eta)\\
&&\times\sum_{j=1}^{2k}\Big(\sqrt{f(\eta^{i-k+j-1,i-k+j})}-\sqrt{f(\eta)}\Big)^2\tun_{\{\nabla\eta(i-k+j-1)=1\}}\\
&\leq& \frac{k 2^{-2k}}{2(2N-2k)}\sum_{\eta\in\cEN}\piN(\eta)\sum_{j=1}^{2N-1}\Big(\sqrt{f(\eta^{j,j+1})}-\sqrt{f(\eta)}\Big)^2\tun_{\{\nabla\eta(j)=1\}}\\
&\leq& \frac{k 2^{-2k}}{(2N-2k)(2N)^2}\DN(f).
\end{eqnarray*}
Now observe that the first term on the right of Equation (\ref{EqSplitSum}) can be written
\begin{eqnarray*}
&&\sup_{f:\DN(f)\leq cN}\frac{1}{N}\sum_{\xi\in\cOk}\Big|\frac{1}{2k+1}\sum_{j=0}^{2k}\Phi(\tau_j \xi) -\tilde{\Phi}\Big(\frac{1}{2k+1}\sum_{j=1}^{2k+1}\xi(j)\Big) \Big|(2N-2k)\fk(\xi)\\
&\leq& 2\sup_{\gk:D^*(\gk)\leq \frac{ck2^{-2k}}{2N(2N-2k)}}\sum_{\xi\in\cOk}\Big|\frac{1}{2k+1}\sum_{j=0}^{2k}\Phi(\tau_j \xi) -\tilde{\Phi}\Big(\frac{1}{2k+1}\sum_{j=1}^{2k+1}\xi(j)\Big) \Big|\gk(\xi),
\end{eqnarray*}
where the inequality comes from the bound on the Dirichlet form proved above and the supremum is implicitly taken over the compact set of non-negative maps $\gk$ such that $\sum_\xi\gk(\xi)=1$. Since the Dirichlet form is lower semi-continuous, we deduce that $\{\gk:D^*(\gk)\leq \frac{ck2^{-2k}}{2N(2N-2k)}\}$ is compact (as a closed subset of a compact set). Also if we write
$$F(\gk) := 2\sum_{\xi\in\cOk}\Big|\frac{1}{2k+1}\sum_{j=0}^{2k}\Phi(\tau_j \xi) -\tilde{\Phi}\Big(\frac{1}{2k+1}\sum_{j=1}^{2k+1}\xi(j)\Big) \Big|\gk(\xi),$$
then the map $F$ is continuous and we deduce that for each $N\geq 1$ there exists $\gNk$ realising the supremum. We stress that
$$ \varlimsup\limits_{N\rightarrow\infty}F(\gNk) \leq \sup_{\gk:D^*(\gk)=0}F(\gk).$$
Indeed, take any sub-sequence of $(\gNk,N\geq 1)$ whose image under $F$ converges to the $\varlimsup$ on the left. Then by compactness one can extract a sub-sub-sequence that converges to a limiting point $\gk^{\infty}$ such that $D^*(\gk^{\infty})=0$ and $\sum_\xi \gk^{\infty}(\xi)=1$. To complete the proof, observe that $\cOk$ can be decomposed into $2k+2$ irreducible classes, each corresponding to the subsets $\cOkl\subset\cOk$ with a constant number of particles $l\in\lbr 0,2k\!+\!1\rbr$. For each $l$, the uniform measure $m_l$ on $\cOkl$ is invariant so that $\{\gk:D^*(\gk)=0\}$ is the set of probability distributions on $\cOk$ obtained as convex combinations of the $m_l$'s. Consequently
$$\sup_{\gk:D^*(\gk)=0}F(\gk)=2\sup_{l\in\lbr 0,2k+1\rbr}\sum_{\xi\in\cOkl}{2k+1\choose l}^{-1}\Big|\frac{1}{2k+1}\sum_{j=0}^{2k}\Phi(\tau_j \xi) -\tilde{\Phi}\Big(\frac{l}{2k+1}\Big) \Big|.$$
Using the local central limit theorem (see for instance Step 6 in Chapter 5.4 of~\cite{KipnisLandimBook}) we deduce that $\varlimsup\limits_{k\rightarrow\infty}\sup_{\gk:D^*(\gk)=0}F(\gk)=0$.
\end{proof}
\begin{lemma}
For any $c >0$
$$ \varlimsup\limits_{k\rightarrow\infty}\varlimsup\limits_{\epsilon\downarrow 0}\varlimsup\limits_{N\rightarrow\infty}\!\sup_{f:D(f)\leq cN}\!\frac{1}{N}\sum_{i=1}^{2N}\frac{1}{(2\epsilon N+1)^2}\!\!\!\sum_{\textrm{\tiny$\substack{j:|j-i| \leq \epsilon N\\j':|j'-i|\leq\epsilon N}$}}\sum_{\eta\in\cEN}\frac{1}{2k+1}\bigg| \!\sum_{l=j'-k}^{j'+k}\!\eta(l) -\!\!\sum_{l=j-k}^{j+k}\!\eta(l)\bigg|\piN(\eta)f(\eta)=0 .$$
\end{lemma}
\begin{proof}
Fix $N\!\geq\! 1$. Observe that the sum over $i$ can be restricted to $\{\lfloor \epsilon N\rfloor +1,\ldots,\lfloor 2N-2\epsilon N \rfloor\}$ since the sum over the remaining $i$'s vanishes when $\epsilon$ goes to $0$. Similarly the sum over $j,j'$ can be restricted to the set
$$ J(i):=\{(j,j'):|j-i|\leq \epsilon N, |j'-i|\leq \epsilon N, j'-j > 2k\}$$
and the term $(2\epsilon N\!+\!1)^2$ can be replaced by $2\,\#J(i)$. Since $\#J(i)$ does not depend on $i$, we can write $\#J$. Consequently we obtain
\begin{equation}\label{EqSupDirichletForm2}
\sup_{f:D(f)\leq cN}\frac{1}{N}\sum_{i=\lfloor 2\epsilon N \rfloor +1}^{\lfloor 2N-2\epsilon N \rfloor}\frac{1}{\#J}\sum_{(j,j') \in J(i)}\sum_{\eta\in\cEN}\frac{1}{2k+1}\bigg| \!\sum_{l=j'-k}^{j'+k}\!\eta(l) -\sum_{l=j-k}^{j+k}\!\eta(l)\bigg|\,\piN(\eta)f(\eta).
\end{equation}
We consider three Dirichlet forms associated to three variants of the simple exclusion process on $\cOk\times\cOk$. From now on, $(\xi_1,\xi_2)$ will implicitly denote an element of the latter set while $\eta$ will designate an element of $\cON$. For all $\gk:\cOk\times\cOk\rightarrow\bbR_+$ we set
\begin{eqnarray*}
D^1(\gk) &:=& \frac{1}{2}\sum_{\xi_1,\xi_2}2^{-2(2k+1)}\sum_{n=1}^{2k}\Big(\sqrt{\gk(\xi_1^{n,n+1},\xi_2)}-\sqrt{\gk(\xi_1,\xi_2)}\Big)^2\tun_{\{\nabla\xi_1(n)=1\}},\\
D^2(\gk) &:=& \frac{1}{2}\sum_{\xi_1,\xi_2}2^{-2(2k+1)}\sum_{n=1}^{2k}\Big(\sqrt{\gk(\xi_1,\xi_2^{n,n+1})}-\sqrt{\gk(\xi_1,\xi_2)}\Big)^2\tun_{\{\nabla\xi_2(n)=1\}},\\
D^\circ(\gk) &:=& \frac{1}{2}\sum_{\xi_1,\xi_2}2^{-2(2k+1)}\Big(\sqrt{\gk\big((\xi_1,\xi_2)^\circ\big)}-\sqrt{\gk(\xi_1,\xi_2)}\Big)^2\tun_{\{\xi_1(k+1)=0,\xi_2(k+1)=1\}},
\end{eqnarray*}
where $(\xi_1,\xi_2)^\circ$ is the configuration obtained from $(\xi_1,\xi_2)$ by exchanging the values of $\xi_1(k+1)$ and $\xi_2(k+1)$. The Dirichlet form $D^1$ (resp. $D^2$) corresponds to a simple exclusion process only acting on $\xi_1$ (resp. $\xi_2$) while $D^\circ$ induces an interaction between $\xi_1$ and $\xi_2$. We now introduce the following map:
$$ \fk(\xi_1,\xi_2):=\sum_{i=\lfloor 2\epsilon N \rfloor +1}^{\lfloor 2N-2\epsilon N \rfloor}\frac{1}{(\lfloor 2N\!-\!2\epsilon N\rfloor \!-\!\lfloor 2\epsilon N\rfloor)\#J}\sum_{(j,j') \in J(i)}\sum_{\textrm{\tiny$\substack{\eta\in\cE_N\\ \eta_{|\lbr j-k,j+k\rbr}=\xi_1\\ \eta_{|\lbr j'-k,j'+k\rbr}=\xi_2}$}}\piN(\eta)f(\eta).$$
By symmetry, we have $D^1(\fk)=D^2(\fk)$ and
\begin{eqnarray*}
D^1(\fk) &\leq& \frac{1}{2}\sum_{\xi_1,\xi_2}\sum_{n=-k}^{k-1}\sum_{i=\lfloor 2\epsilon N \rfloor +1}^{\lfloor 2N-2\epsilon N \rfloor}\frac{2^{-2(2k+1)}}{(\lfloor 2N\!-\!2\epsilon N\rfloor \!-\!\lfloor 2\epsilon N\rfloor)\#J}\sum_{(j,j') \in J(i)}\\
&&\times\!\!\!\sum_{\textrm{\tiny$\substack{\eta\in\cE_N\\ \eta_{|\lbr j-k,j+k\rbr}=\xi_1\\ \eta_{|\lbr j'-k,j'+k\rbr}=\xi_2}$}}\piN(\eta)\Big(\sqrt{f\big(\eta^{j+n,j+n+1}\big)}-\sqrt{f(\eta)}\Big)^2\tun_{\{\nabla\eta(j+n)=1\}}\\
&\leq& \frac{8k(\epsilon N+k)\epsilon N}{\big(\lfloor 2N\!-\!2\epsilon N\rfloor \!-\!\lfloor 2\epsilon N\rfloor\big)(2N)^2\#J}\,\DN(f).
\end{eqnarray*}
Indeed, for a given flip appearing in the Dirichlet form, we have at most $2(\epsilon N+k)$ choices for $i$, $2k$ choices for $j$ and $2\epsilon N$ choices for $j'$. $D^\circ(\fk)$ can be bounded by
\begin{equation*}
\sum_{\xi_1,\xi_2}\sum_{i=\lfloor 2\epsilon N \rfloor +1}^{\lfloor 2N-2\epsilon N \rfloor}\!\frac{(\lfloor 2N\!-\!2\epsilon N\rfloor \!-\!\lfloor 2\epsilon N\rfloor)^{-1}}{2\#J 2^{-2(2k+1)}}\!\!\!\!\sum_{(j,j') \in J(i)}\!\sum_{\textrm{\tiny$\substack{\eta\in\cE_N\\ \eta_{|\lbr j\!-\!k,j\!+\!k\rbr}=\xi_1\\ \eta_{|\lbr j'\!-\!k,j'\!+\!k\rbr}=\xi_2}$}}\!\!\!\!\!\!\!\!\!\piN(\eta)\Big(\sqrt{f\big(\eta^{j,j'}\big)}-\sqrt{f(\eta)}\Big)^2\tun_{\{\eta(j)=0,\eta(j')=1\}},
\end{equation*}
where $\eta^{j,j'}$ is obtained from $\eta$ by exchanging the values $\eta(j)$ and $\eta(j')$. Observe that we have
$$ \eta^{j,j'} = \Big(\ldots\Big(\Big(\big(\ldots\big((\eta^{j,j+1})^{j+1,j+2}\big)\ldots\big)^{j'-1,j'}\Big)
^{j'-2,j'-1}\Big)\ldots\Big)^{j,j+1}.$$
We denote by $\eta_p$ the configuration obtained at the $p$-th step of the above formula, that is, $\eta_0:=\eta$, $\eta_1:=\eta^{j,j+1}$,$\ldots$, $\eta_{2(j'-j)-1}=\eta^{j,j'}$. We stress that all these configurations belong to $\cEN$, this is a consequence of our condition $\{\eta(j)\!=\!0,\eta(j')\!=\!1\}$. We thus have
$$\Big(\sqrt{f\big(\eta^{j,j'}\big)}-\sqrt{f(\eta)}\Big)^2 \leq \big(2(j'-j)-1\big)\sum_{p=1}^{2(j'-j)-1}\Big(\sqrt{f(\eta_p)}-\sqrt{f(\eta_{p-1})}\Big)^2.$$
One obtains $\eta_p$ from $\eta_{p-1}$ by exchanging the values of two consecutive sites. Then a simple calculation ensures the existence of a constant $r > 0$ such that when $k/\epsilon N$ is small enough
$$ D^\circ(\fk) \leq \frac{r\epsilon^2}{N}\,\DN(f).$$
We introduce the set $\GNk(\epsilon)$ of maps $\gk\!:\!\cOk\!\times\!\cOk\!\rightarrow\bbR$ such that $\sum_{\xi_1,\xi_2}\gk(\xi_1,\xi_2)=1$ and
\begin{equation*}
D^1(\gk),D^2(\gk)\leq\frac{4kc(\epsilon N+k)\epsilon N}{\big(\lfloor 2N\!-\!2\epsilon N\rfloor \!-\!\lfloor 2\epsilon N\rfloor\big)2N\#J}\;\;;\;\;D^\circ(\gk) \leq r\epsilon^2c.
\end{equation*}
Expression (\ref{EqSupDirichletForm2}) can be rewritten as follows:
\begin{eqnarray*}
&&\sup_{f:\DN(f)\leq cN}\frac{1}{N}\sum_{\xi_1,\xi_2}\sum_{i=\lfloor 2\epsilon N \rfloor +1}^{\lfloor 2N-2\epsilon N \rfloor}\frac{(\lfloor 2N\!-\!2\epsilon N\rfloor \!-\!\lfloor 2\epsilon N\rfloor)^{-1}}{\#J}\\
&&\times\sum_{(j,j') \in J(i)}\sum_{\textrm{\tiny$\substack{\eta\in\cE_N\\ \eta_{|\lbr j-k,j+k\rbr}=\xi_1\\ \eta_{|\lbr j'-k,j'+k\rbr}=\xi_2}$}}\!\!\!\frac{1}{2k+1}\,\Big| \sum_{l:|j'-l|\leq k}\!\!\!\eta(l) -\!\sum_{l:|j-l|\leq k}\!\!\eta(l)\Big|\piN(\eta)f(\eta)\\
&\leq& 2\sup_{\gk\in \GNk(\epsilon)}\sum_{\xi_1,\xi_2}\frac{1}{2k+1}\,\Big| \sum_{l=1}^{2k+1}\xi_1(l) -\sum_{l=1}^{2k+1}\xi_2(l)\Big|\gk(\xi_1,\xi_2).
\end{eqnarray*}
By the same compactness arguments as in the proof of the previous lemma, it suffices to show that
$$ \varlimsup\limits_{k\rightarrow\infty}\sup_{D^1(\gk)=D^2(\gk)=\Delta(\gk)=0}\sum_{\xi_1,\xi_2}\frac{1}{2k+1}\,\Big| \sum_{l=1}^{2k+1}\xi_1(l) -\sum_{l=1}^{2k+1}\xi_2(l)\Big|\,\gk(\xi_1,\xi_2)=0.$$
We now see $\gk$ as a probability measure on $\cOk\times\cOk$. The conditions $D^1(\gk)=D^2(\gk)=D^\circ(\gk)=0$ imply that $\gk$ is a convex combination of the uniform measures on $\cOk\times\cOk$ with a given number of particles. As at the end of the preceding lemma, the local central limit theorem completes the proof.\cqfd
\end{proof}

\subsection{The asymmetric case}
In the last subsection, we proved Theorem \ref{ThSuperExpo} under $\bbPN_{\piN}$. Let $\cVN(\delta)$ be the Borel subset of $\bbD([0,t],\cEN)$ defined by $\cVN(\delta):=\Big\{\upeta:\frac{1}{N}\int_0^t\VNe(\upeta_s)ds > \delta\Big\}$. Recall that we work implicitly in \Munw, so that we drop the superscript on the state-spaces. For any measure $\nuN$ on $\cCN$ we have
\[ \bbPN_{\nuN}\big(\cVN(\delta)\big) = \int_{\upeta\in\cEN}\tun_{\{\upeta\in\cVN(\delta)\}} \frac{\nuN(\upeta_0)}{\piN(\upeta_0)} d\bbPN_{\piN}(\upeta)\leq 2^{2N}\bbPN_{\piN}\big(\cVN(\delta)\big),\]
so that Theorem \ref{ThSuperExpo} also holds under $\bbPN_{\nuN}$. We now extend it to the asymmetric setting. To that end, we write
\begin{eqnarray*} 
\bbQN_{\nuN}\big(\cVN(\delta)\big) =\int\tun_{\{\upeta\in\cVN(\delta)\}}\frac{d\bbQN_{\nuN}}{d\bbPN_{\nuN}}d\bbPN_{\nuN}(\upeta) \leq \bigg(\bbPN_{\piN}\big(\cVN(\delta)\big)\bigg)^{\frac{1}{2}}\bigg(\int\Big(\frac{d\bbQN_{\nuN}}{d\bbPN_{\nuN}}\Big)^2 d\bbPN_{\nuN}(\upeta)\bigg)^{\frac{1}{2}}.
\end{eqnarray*}
Hence the result for $\bbQN_{\nuN}$ will follow if we can prove the existence of a constant $c > 0$ such that for all $N\geq 1$
\begin{equation}\label{EqBoundRadonNikodym}
\int\Big(\frac{d\bbQN_{\nuN}}{d\bbPN_{\piN}}\Big)^2 d\bbPN_{\piN}(\upeta)=\int\Big(\frac{d\bbQN_{\upeta_0}}{d\bbPN_{\upeta_0}}(\upeta)\Big)^2\Big(\frac{\nuN(\upeta_0)}{\piN(\upeta_0)}\Big)^2 d\bbPN_{\piN}(\upeta) \leq \exp(cN),
\end{equation}
where $\bbQN_{\upeta_0}$ denotes the distribution of the process starting from $\delta_{\upeta_0}$ at time $0$. The assumption on the asymmetry yields the following uniform estimates:
\begin{equation}\label{Eq:EstimatesPn}
\pN(\cdot) = \frac{(2N)^2}{2} + \sigma(\cdot)\sqrt{2N} + \cO\big((2N)^{-1}\big)\;\;,\;\;\qN(\cdot) = \frac{(2N)^2}{2} - \sigma(\cdot)\sqrt{2N} + \cO\big((2N)^{-1}\big).\end{equation}
For any initial condition $\upeta_0\in\cEN$, the measures $\bbQN_{\upeta_0}$ and $\bbPN_{\upeta_0}$ are equivalent and their Radon-Nikodym derivative up to time $t$ is given by (see for instance Appendix 1 - Proposition 2.6 in~\cite{KipnisLandimBook})
\begin{equation}\label{Eq:RadonNikodym}
\begin{split}
\log\frac{d\bbQN_{\upeta_0}}{d\bbPN_{\upeta_0}}(\upeta) &= \int_0^t\sum_{k=1}^{2N-1}\Big(\pN(k_{\tiN})-\frac{(2N)^2}{2}\Big)\big(\tun_{\{\nabla\upeta_s(k)=+1\}}-\tun_{\{\nabla\upeta_s(k)=-1\}}\big)ds\\
&\;\;\;\;-\sum_{k=1}^{2N-1}\Big(\log\frac{2\qN(k_{\tiN})}{(2N)^2} J_t^{k,k+1} + \log\frac{2\pN(k_{\tiN})}{(2N)^2} J_t^{k+1,k}\Big)
\end{split}
\end{equation}
where $J_t^{k,k+1}$ (resp. $J_t^{k+1,k}$) is the number of particles that have jumped from $k$ to $k+1$ (resp. from $k+1$ to $k$) up to time $t$. We rewrite the first term on the right of (\ref{Eq:RadonNikodym}) as follows:
\begin{eqnarray*}
\sum_{k=1}^{2N-1}\Big(\pN(k_{\tiN})-\frac{(2N)^2}{2}\Big)\big(\upeta_s(k+1)-\upeta_s(k)\big)\!\!\!&=&\!\!\!\sum_{k=2}^{2N-1}\Big(\pN((k-1)_{\tiN})-\pN(k_{\tiN})\Big)\,\upeta_s(k)\\
&&+\Big(\pN(1-\frac{1}{2N})-\frac{(2N)^2}{2}\Big)\,\upeta_s(2N)\\
&&-\Big(\pN(\frac{1}{2N})-\frac{(2N)^2}{2}\Big)\,\upeta_s(1)
\end{eqnarray*}
so that the uniform estimates (\ref{Eq:EstimatesPn}) together with the $1/2$-H\"older regularity of $\sigma$ ensures that this last expression is of order $N$ uniformly in $\upeta$. We now focus on the second term on the right of (\ref{Eq:RadonNikodym}) and write this as the sum of
\begin{equation*} A := \sum_{k=1}^{2N-1}\frac{2\sigma(k_{\tiN})}{(2N)^{\frac{3}{2}}}\Big(J_t^{k+1,k}-J_t^{k,k+1}\Big)\;\;,\;\;
B := \cO\big((2N)^{-3}\big)\sum_{k=1}^{2N-1}\Big(J_t^{k,k+1} + J_t^{k+1,k}\Big).
\end{equation*}
A simple calculation shows that $A=\frac{(2N)^{\frac{3}{2}}}{2}\big(\ccA_{\tiN}(\rmh_{t})-\ccA_{\tiN}(\rmh_{0})\big)$ where $\ccA_{\tiN}(\cdot)$ is the discrete weighted area under the interfaces as defined in Proposition \ref{PropInvMeasure}. Consequently we have $|A| \leq \sup|\sigma|\sqrt{2N}$ for every $\upeta$. Concerning $B$, observe that the sum is less than $\sum_k\cLN_t(k_{\tiN})+\cRN_t(k_{\tiN})$ which is a Poisson random variable with mean $t(2N)^3$ under $\bbPN_{\piN}$. Putting all these arguments together we deduce that (\ref{EqBoundRadonNikodym}) is fulfilled. This concludes the proof of Theorem \ref{ThSuperExpo}.\cqfd

\section{Proof of the tightness in \Modun}
\label{SectionProofTightnessMun}
We work in the natural filtration induced by the canonical process $(\rmh_t,t\geq 0)$: all the martingales will be considered w.r.t.~this filtration. Recall the notation $k_{\tiN}=\frac{k}{2N}$. First we rewrite the system of stochastic differential equations (\ref{Eq:DefMun}) in the following semimartingale form
\begin{equation}\label{EqrmhM1}
d\rmh_t(k_{\tiN}) = \frac{(2N)^2}{2}\,\Delta\rmh_t(k_{\tiN})dt 
+\frac{2}{\sqrt{2N}}\Big(\pN(k_{\tiN})-\frac{(2N)^2}{2}\Big)\tun_{\{\Delta \rmh_t(k_{\tiN})\ne 0\}}dt + d M_t(k_{\tiN}),
\end{equation}
where
\[ M_t(k_{\tiN}) := \frac{2}{\sqrt{2N}}\Big(\int_0^t(d\cLN_s(k_{\tiN})-\pN(k_{\tiN}) ds)\tun_{\{\Delta \rmh_s(k_{\tiN}) > 0\}} - \int_0^t(d\cRN_s(k_{\tiN})-\qN(k_{\tiN}) ds)\tun_{\{\Delta \rmh_s(k_{\tiN}) < 0\}}\Big)\]
is a martingale. We introduce the fundamental solution $\frg^{\tiN}\!=\frg$ that solves for all $k,l \in \lbr 0,2N\rbr$
\begin{eqnarray*}
\begin{cases}\partial_t \frg_t(k_{\tiN},l_{\tiN}) \!\!\!\!&= \frac{(2N)^2}{2}\,\Delta \frg_t(k_{\tiN},l_{\tiN}),\\
\frg_0(k_{\tiN},l_{\tiN}) \!\!\!\!&= \delta_{k_{\tiN}}(l_{\tiN}),\\
\frg_t(k_{\tiN},0) \!\!\!\!&= \frg_t(k_{\tiN},1) = \frg_t(0,l_{\tiN}) = \frg_t(1,l_{\tiN}) = 0.\end{cases}
\end{eqnarray*}
Notice that the discrete Laplacian on the first line acts on the map $l_{\tiN}\mapsto\frg_t(k_{\tiN},l_{\tiN})$ for any given $k_{\tiN}$. Classical arguments (see for instance Chapter V p.237 in the book of Spitzer~\cite{Spitzer76}) ensure that for all $t\geq 0$ and all $k,l \in \lbr 0,2N\rbr$ we have
\begin{equation}\label{EqDiscreteHeatKernel}
\frg_t(k_{\tiN},l_{\tiN}) = \frac{1}{N} \sum_{n=1}^{2N-1}\sin\big(n\pi k_{\tiN}\big)\sin\big(n\pi l_{\tiN}\big)e^{(2N)^2 t\big(\cos(\frac{n}{2N}\pi)-1\big)}.
\end{equation}
\begin{remark}
The function $(t,k,l)\mapsto\frg_{t}(k_{\tiN},l_{\tiN})$ is the Green function associated to the differential operator $\partial_t-\frac{(2N)^2}{2}\Delta$. It corresponds to the transition kernel of a continuous-time simple random walk on $\{0,\frac{1}{2N},\ldots,1\}$ sped up by $(2N)^2/2$ and killed at $0$ and $1$.
\end{remark}
\noindent From the fundamental solution, one can write the mild formulation of the semimartingale:
\begin{equation}\label{EqMildDiscrete}
\begin{split}
\rmh_t(k_{\tiN}) &= \sum_{k=0}^{2N}\frg_{t}(k_{\tiN},l_{\tiN})\,\rmh_0(k_{\tiN}) + N_t^t(l_{\tiN})\\
&\;\;\;\;+ \frac{2}{\sqrt{2N}}\sum_{k=0}^{2N}\int_0^t \frg_{t-r}(k_{\tiN},l_{\tiN})\Big(\pN(k_{\tiN})-\frac{(2N)^2}{2}\Big)\tun_{\{\Delta\rmh_r(k_{\tiN})\ne 0\}}dr,
\end{split}
\end{equation}
where we have introduced the collection of martingales $(N_s^t(l_{\tiN}),0 \leq s \leq t),l\in\lbr 0,2N\rbr$ as follows:
$$ N_s^t(l_{\tiN}) := \sum_{k=0}^{2N}\int_0^s \frg_{t-r}(k_{\tiN},l_{\tiN})dM_r(k_{\tiN}). $$
This mild formulation is valid since (\ref{EqMildDiscrete}) defines a process satisfying the stochastic differential equations (\ref{EqrmhM1}) for which pathwise uniqueness is known.\\
Let us introduce some notation. For every $p\in[1,\infty)$, $\left\|F\right\|_{p}$ will denote the $\rL^p$ norm of a real-valued random variable $F$. For any square integrable c\`adl\`ag martingale $(X_t,t\geq 0)$, $[X]$ will denote its quadratic variation. In the particular case of purely discontinuous martingales, we have
\[ \forall t\geq 0,\;\;[X]_t := \sum_{\tau\leq t}\big(X_\tau-X_{\tau-}\big)^2.\]
We also denote by $\left\llangle X\right\rrangle$ the bracket of $X$, defined as the unique predictable process such that $(X^2_t - \left\llangle X\right\rrangle_t,t\geq 0)$ is a martingale. We recall the Burkholder-Davis-Gundy inequality~\cite{LLP80} that ensures, for any $p\in[1,\infty)$, the existence of a constant $\tcBDG(p) > 0$ such that for all $t\geq 0$
\[\left\|X_t\right\|_{p} \leq \tcBDG(p)\sqrt{\left\|\,[X]_t\,\right\|_{p/2}}.\]
Since the process $D_t:=[X]_t-\left\llangle X\right\rrangle_t$ is itself a martingale, for any $p\geq 2$ we have the following inequality
\begin{equation}\label{Eq:BDGTwice}
\left\| X_t \right\|_p \leq \tcBDG(p)\sqrt{\left\|\left\llangle X \right\rrangle_t \right\|_{p/2}} + \tcBDG(p)\sqrt{\tcBDG(p/2)}\left\|\big[D\big]_t \right\|_{p/4}^{\frac{1}{4}}.
\end{equation}
The proof of Proposition \ref{PropHolderSpaceTimeBar} requires a series of lemmas that we now present. From the hypothesis on $\pN,\qN$, we know that there exists $\tilde{\sigma} > 0$ such that for all $N\geq 1$
\begin{equation}\label{EqAsym}
\sup_{k\in\lbr 1,2N-1\rbr}|\qN(k_{\tiN})-\pN(k_{\tiN})| \leq \tilde{\sigma}\sqrt{2N}.
\end{equation}
Fix $T>0$ until the end of the section.
\begin{lemma}\label{LemmaBoundsHeatKernel}The following properties hold true
\begin{enumerate}
\item[\rm{(i)}] For all $N\geq 1$, $t\geq 0$, $k,l \in \llbracket 0,2N\rrbracket$,\hspace{5pt} $\frg_t(k_{\tiN},l_{\tiN}) \leq 1 \wedge \sqrt{\frac{2\pi}{(2N)^2t}}$ .
\item[\rm{(ii)}] Fix $\gamma \in (0,1]$. There exists a constant $\tc_{\textup{\tiny kernel}}'(\gamma,T) > 0$ such that for all $N\in\bbN$, $l \in \llbracket 0,2N\rrbracket$ and all $t\leq t' \in [0,T]$
$$\sup_{k\in\llbracket 0,2N\rrbracket}\big|\frg_{t'}(k_{\tiN},l_{\tiN})-\frg_t(k_{\tiN},l_{\tiN})\big| \leq \frac{\tc_{\textup{\tiny kernel}}'(\gamma,T)}{2N\sqrt{t}}\Big(\frac{t'-t}{t}\Big)^{\gamma}.$$
\item[\rm{(iii)}] For all $N\in\bbN$, $0 \leq s \leq t$ and all $k,l \in \llbracket 0,2N\rrbracket$, $\sup_{r\in[s,t]}\frg_r(k_{\tiN},l_{\tiN}) \leq e^{(2N)^2(t-s)}\frg_t(k_{\tiN},l_{\tiN})$.
\end{enumerate}
\end{lemma}
\begin{proof}
First, observe that (\ref{EqDiscreteHeatKernel}) can be rewritten
\begin{equation*}
\frg_t(k_{\tiN},l_{\tiN}) = 2 \int_0^1\sin\big(\left\lfloor 2Nu\right\rfloor\pi k_{\tiN}\big)\sin\big(\left\lfloor 2Nu\right\rfloor\pi l_{\tiN}\big)e^{(2N)^2 t\big(\cos(\frac{\left\lfloor 2Nu\right\rfloor}{2N}\pi)-1\big)}du.
\end{equation*}
Second, recall that for all $a \in [0,\pi]$, $1-a^2/2 \leq \cos(a) \leq 1 - 2a^2/\pi^2$. To prove (i), we write
\begin{eqnarray*}
\frg_t(k_{\tiN},l_{\tiN}) &\leq& 2 \int_0^1 e^{-2(2N)^2 t(u-\frac{1}{2N})^2}du \leq \sqrt{\frac{2\pi}{(2N)^2 t}} \int_\bbR\sqrt{\frac{2(2N)^2 t}{\pi}}e^{-2(2N)^2 tu^2}du = \sqrt{\frac{2\pi}{(2N)^2 t}},
\end{eqnarray*}
where we use the bound on the cosine in the first inequality and we recognise the Gaussian distribution in the second step. Bound (i) follows.\\
We turn to (ii).  Fix $N$ and $\gamma$ as in the statement. For all $k,l\in\llbracket 0,2N\rrbracket$ and all $0 \leq t \leq t' \leq T$, we bound $|\frg_{t'}(k_{\tiN},l_{\tiN})-\frg_t(k_{\tiN},l_{\tiN})|$ by 
\begin{eqnarray*}
&&\frac{2}{\sqrt{(2N)^2 t}}\int_0^{\sqrt{(2N)^2 t}}\Big(1-e^{(2N)^2(t'-t)\big(\cos(\left\lfloor \frac{2Nv}{\sqrt{(2N)^2 t}}\right\rfloor\frac{\pi}{2N})-1\big)} \Big)e^{(2N)^2 t\big(\cos(\left\lfloor \frac{2Nv}{\sqrt{(2N)^2 t}}\right\rfloor\frac{\pi}{2N})-1\big)}dv\\
&\leq&\frac{2}{\sqrt{(2N)^2 t}}\int_0^{\sqrt{(2N)^2 t}}\Big(1-e^{-\frac{(t'-t)\pi^2}{2t}v^2}\Big)e^{-2(v-\sqrt{t})^2}dv.
\end{eqnarray*}
Since $\gamma$ belongs to $(0,1]$, we have $1-e^{-a} \leq a^\gamma$ for all $a\geq 0$, and we deduce that
\begin{eqnarray*}
|\frg_{t'}(k_{\tiN},l_{\tiN})-\frg_t(k_{\tiN},l_{\tiN})|&\leq&\frac{2}{\sqrt{(2N)^2 t}}\Big(\frac{t'-t}{t}\Big)^\gamma\Big(\frac{\pi^2}{2}\Big)^\gamma\int_{-\sqrt{t}}^{\sqrt{(2N)^2 t}}\big|v+\sqrt{t}\big|^{2\gamma}e^{-2v^2}dv.
\end{eqnarray*}
Setting $\tc_{\textup{\tiny kernel}}'(\gamma,T):=2(\pi^2/2)^\gamma\int_\bbR|v+\sqrt{2T}|^{2\gamma}e^{-2v^2}dv$, the asserted bound follows.\\
Finally we observe that
$$ \frg_t(k_{\tiN},l_{\tiN}) = e^{-(2N)^2 t}\sum_{n=0}^{\infty}\frac{((2N)^2 t)^n}{n!}g_n(k,l),$$
where $g_n(k,l)$ is the probability that a discrete time symmetric random walk on $\llbracket 0,2N\rrbracket$, starting from $k$ and killed at $0$ and $2N$, reaches $l$ after $n$ jumps. Bound (iii) then easily follows.\cqfd
\end{proof}
From now on the $\rL^p$ norm is always implicitly taken under the measure $\bbQN_{\nuN}$. 
\begin{lemma}\label{LemmaHolderTime}
Fix $p\in [4,\infty)$ and $\gamma \in (0,\frac{1}{4}\wedge\frac{\upbeta_{\textup{\tiny init}}}{2})$. There exists $\tk_{\textup{\tiny time}}=\tk_{\textup{\tiny time}}(p,T,\gamma) > 0$ such that for all $0 \leq t \leq t' \leq T$ and all $N\geq 1$, under $\bbQN_{\nuN}$ we have
$$ \sup\limits_{l\in\llbracket 0,2N\rrbracket}\left\| \rmh_{t'}\big(l_{\tiN}\big)-\rmh_t\big(l_{\tiN}\big) \right\|_{p} \leq \tk_{\textup{\tiny time}}\Big( (t'-t)^\gamma + \frac{1}{(2N)^{\frac{1}{2}\wedge\upbeta_{\textup{\tiny init}}}}\Big).$$
\end{lemma}
\begin{proof}
Fix $N,l,t,t',\gamma$ as in the statement. Using (\ref{EqMildDiscrete}), we treat separately the initial condition, the asymmetric terms and the martingale term by writing
$$\left\| \rmh_{t'}\big(l_{\tiN}\big)-\rmh_t\big(l_{\tiN}\big) \right\|_{p} \leq \cI(l,t,t') + \cA(l,t,t') + \cN(l,t,t') $$
where
\begin{eqnarray*}
\cI(l,t,t') \!\!\!&:=&\!\!\! \left\| \sum_{k=0}^{2N}\big(\frg_{t'}(k_{\tiN},l_{\tiN})-\frg_{t}(k_{\tiN},l_{\tiN})\big)\rmh_0(k_{\tiN}) \right\|_p,\\
\cA(l,t,t') \!\!\!&:=&\!\!\! 2\,\sum_{k=0}^{2N}\frac{|\pN(k_{\tiN})\!-\frac{(2N)^2}{2}|}{\sqrt{2N}}\Bigg\|\Big(\!\int_0^{t'}\!\!\!\!\! \frg_{t'-r}(k_{\tiN},l_{\tiN})\tun_{\{\Delta\rmh_r(k_{\tiN})\ne 0\}}dr\\
&&\hspace{3.8cm}-\!\!\int_0^{t}\!\!\!\!\frg_{t-r}(k_{\tiN},l_{\tiN})\tun_{\{\Delta\rmh_r(k_{\tiN})\ne 0\}}dr\Big) \Bigg\|_p,\\
\cN(l,t,t') \!\!\!&:=&\!\!\! \left\|N_{t'}^{t'}(l_{\tiN})-N_{t}^{t}(l_{\tiN})\right\|_p.
\end{eqnarray*}
Below, we prove that each of the three terms separately satisfies the bound of the statement. We start with the initial condition. Observe that one can extend $\rmh_0$ into a $2$-periodic asymmetric function on the whole of $\bbR$. The solution to the discrete heat equation on $[0,1]$ starting from $\rmh_0$ is then the restriction of the solution to the discrete heat equation on $\bbR$ starting from $\rmh_0$. The fundamental solution $\frf$ for this discrete heat equation on $\bbR$ is translation invariant, and we can write
\begin{eqnarray*}
\cI(l,t,t') \!\!&:=&\!\! \left\| \sum_{k\in\bbZ}\big(\frf_{t'}(l_{\tiN}-k_{\tiN})-\frf_{t}(l_{\tiN}-k_{\tiN})\big)\rmh_0(k_{\tiN}) \right\|_p\\
\!\!&=&\!\! \left\| \sum_{i\in\bbZ}\frf_{t'-t}(l_{\tiN}-i_{\tiN})\big(\sum_{k\in\bbZ}\frf_{t}(i_{\tiN}-k_{\tiN})\rmh_0(k_{\tiN})-\sum_{k\in\bbZ}\frf_{t}(l_{\tiN}-k_{\tiN})\rmh_0(k_{\tiN}) \big)\right\|_p\\
\!\!&\leq&\!\! \tc_{\textup{\tiny init}}\sum_{i\in\bbZ}\frf_{t'-t}((l-i)_{\tiN})\big|(l-i)_{\tiN}\big|^{\upbeta_{\textup{\tiny init}}},
\end{eqnarray*}
where we have used the semigroup property in the second line. A simple calculation (or Proposition A.1 in~\cite{DemboTsai13}) ensures that $\sup_{r\geq 0}\sum_{i\in\bbZ}\frf_{r}((l-i)_{\tiN})\exp(\frac{|l-i|_{\tiN}}{\sqrt{r}\vee\frac{1}{2N}}) < \infty$.
Since $x^{\upbeta_{\textup{\tiny init}}} e^{-|x|}$ is bounded on $\bbR$, we deduce that
$$\sum_{i\in\bbZ}\frf_{t'-t}((l-i)_{\tiN})\big|(l-i)_{\tiN}\big|^{\upbeta_{\textup{\tiny init}}} \apprle |t'-t|^{\frac{\upbeta_{\textup{\tiny init}}}{2}} + \frac{1}{(2N)^{\upbeta_{\textup{\tiny init}}}}.$$
This implies the bound for the initial condition. The asymmetric term can be handled using Equation (\ref{EqAsym}) and Lemma \ref{LemmaBoundsHeatKernel} \rm{(i)} and \rm{(ii)}:
\begin{eqnarray*}
\cA(l,t,t') &\leq& 2\tilde{\sigma}\sum_{k=0}^{2N}\int_0^{t}\!\! \big|\frg_{t'-r}(k_{\tiN},l_{\tiN})-\frg_{t-r}(k_{\tiN},l_{\tiN})\big|dr\!+\,2\tilde{\sigma}\sum_{k=0}^{2N}\int_{t}^{t'}\!\!\!\frg_{t'-r}(k_{\tiN},l_{\tiN})dr\\
&\leq& 2\tilde{\sigma}\int_0^{t}\!\! \frac{\tc_{\textup{\tiny kernel}}'(\gamma,T)}{\sqrt{t-r}}\bigg(\frac{t'-t}{t-r}\bigg)^{\gamma}dr+2\tilde{\sigma}\int_{t}^{t'}\!\!\!\sqrt{\frac{2\pi}{t'-r}}\,dr\\
&\leq& \frac{4\tilde{\sigma}\, T\tc_{\textup{\tiny kernel}}'(\gamma,T)}{1-2\gamma}(t'-t)^{\gamma} + 4\tilde{\sigma}\sqrt{\pi(t'-t)}.
\end{eqnarray*}
This ensures the expected bound for the asymmetric term since $(t'-t)^{\frac{1}{2}} \leq (t'-t)^{\gamma}T^{\frac{1}{2}-\gamma}$.\\
We turn to the martingale term. We want to bound, for all $0 \leq t \leq t+\delta \leq T$, the $\rL^p$-norm of $N_{t+\delta}^{t+\delta}(l_{\tiN})-N_t^{t}(l_{\tiN})$. To that end, we split it into $N_{t+\delta}^{t+\delta}(l_{\tiN})-N_t^{t+\delta}(l_{\tiN})$ and $N_{t}^{t+\delta}(l_{\tiN})-N_t^{t}(l_{\tiN})$. To deal with the first term, we introduce
$$ \forall u\in[0,\delta],\;\; A_u^{t+\delta}(l) := \sum_{k=0}^{2N}\int_t^{t+u}\frg_{t+\delta-r}(k_{\tiN},l_{\tiN})dM_r(k_{\tiN}),$$
which is an $\cF_{t+u}$-martingale. We easily see that $A_\delta^{t+\delta}(l)=N_{t+\delta}^{t+\delta}(l_{\tiN})-N_t^{t+\delta}(l_{\tiN})$. We introduce $D_u^{t+\delta}(l):=\big[A^{t+\delta}(l)\big]_u-\left\llangle A^{t+\delta}(l)\right\rrangle_u$. Let us partition the interval $[0,\delta]$ into the subintervals $I_i:=[i(2N)^{-2},(i+1)(2N)^{-2}]$ for all $i=0,\ldots,\left\lfloor \delta(2N)^2\right\rfloor-1$ and $I_{\left\lfloor \delta(2N)^2\right\rfloor}:=[\left\lfloor \delta(2N)^2\right\rfloor(2N)^{-2},\delta]$. Observe that
\[ \left\|\cL^{\tiN}_{(i+1)(2N)^2}(k_{\tiN})-\cL^{\tiN}_{i(2N)^2}(k_{\tiN})+\cR^{\tiN}_{(i+1)(2N)^2}(k_{\tiN})-\cR^{\tiN}_{i(2N)^2}(k_{\tiN}) \right\|_p\]
is the $\rL^p$-norm of a Poisson random variable with mean $1$; let us denote this quantity by $a(p)$. Then we obtain
\begin{eqnarray*}
&&\left\| \big[D^{t+\delta}(l)\big]_\delta \right\|_p = \left\|\sum_{\tau \in (0,\delta]}\sum_{k=0}^{2N}\frg_{t+\delta-\tau}(k_{\tiN},l_{\tiN})^4\, \big(\rmh_{t+\tau}(k)-\rmh_{t+\tau-}(k_{\tiN})\big)^4\right\|_p\\
&&\leq\frac{2^4}{(2N)^2}\sum_{k=0}^{2N}\sum_{i=0}^{\left\lfloor \delta(2N)^2\right\rfloor}
\!\!\!\!\sup_{r\in I_i}\frg_{t+\delta-r}(k_{\tiN},l_{\tiN})^4\, a(p).
\end{eqnarray*}
Additionally, for every $i$ in the above sum we bound $\sum_k\sup_{r\in I_i}\frg_{t+\delta-r}(k_{\tiN},l_{\tiN})^4$ by
\[e^4 \sup_{k}\frg_{t+\delta-i(2N)^{-2}}(k_{\tiN},l_{\tiN})^3 \sum_k\frg_{t+\delta-i(2N)^{-2}}(k_{\tiN},l_{\tiN})\leq e^4 (1 \wedge\frac{2\pi}{(2N)^2(t+\delta-i(2N)^{-2})})^{3/2},\]
using Lemma \ref{LemmaBoundsHeatKernel} \rm{(i)} and \rm{(iii)}. Consequently
\[ \left\| \big[D^{t+\delta}(l)\big]_\delta \right\|_p \leq \frac{2^4}{(2N)^2} e^4 a(p) \bigg(1+\sum_{i=0}^{\left\lfloor \delta(2N)^2\right\rfloor-1} \Big(\frac{2\pi}{(2N)^2(t+\delta-i(2N)^{-2})}\Big)^{\frac{3}{2}}\bigg).\]
The r.h.s.~can be bounded by $(2N)^{-2}$ times a constant $d(p)$ that does not depend on $t,\delta$. Moreover, we use the $\bbQN_{\nuN}$-a.s.~bound $d\left\llangle M(k_{\tiN}) \right\rrangle_r \leq 8Ndr$ to write
\begin{eqnarray*}
\left\| \left\llangle A^{t+\delta}(l) \right\rrangle_\delta \right\|_p &=& \left\|\int_t^{t+\delta}\sum_{k=0}^{2N}\frg_{t+\delta-r}(k_{\tiN},l_{\tiN})^2\, d\left\llangle M(k_{\tiN}) \right\rrangle_r \right\|_p\\
&\leq& 8N \int_t^{t+\delta}\sqrt{\frac{2\pi}{(2N)^2(t+\delta-r)}}dr \leq 8\sqrt{2\pi\delta}.
\end{eqnarray*}
Thus, using (\ref{Eq:BDGTwice}), we obtain
\[ \left\| N_{t+\delta}^{t+\delta}(l_{\tiN})-N_t^{t+\delta}(l_{\tiN}) \right\|_p \leq 2\sqrt{2}\,\tcBDG(p)(2\pi\delta)^{\frac{1}{4}}+\frac{\tcBDG(p)\tcBDG(p/2)^{\frac{1}{2}}d(p/4)^{\frac{1}{4}}}{\sqrt{2N}}.\]
Consequently for all $0 \leq t \leq t' \leq T$ we have
\begin{eqnarray*}
\left\|N_{t'}^{t'}(l_{\tiN})-N_{t}^{t'}(l_{\tiN})\right\|_p\!\!\! &\leq&\!\!\!2\sqrt{2}\tcBDG(p)\big(2\pi(t'-t)\big)^{\frac{1}{4}} + \frac{\tcBDG(p)\tcBDG(p/2)^{\frac{1}{2}}d(p/4)^{\frac{1}{4}}}{\sqrt{2N}},
\end{eqnarray*}
which proves the bound for the first part of the martingale term. We now turn to the second part, $N_{t}^{t+\delta}(l_{\tiN})-N_t^{t}(l_{\tiN})$, of the martingale term. For every $0 \leq t \leq t +\delta \leq T$, we introduce $[0,t]\ni s\mapsto B_s^t(\delta,l):=N_{s}^{t+\delta}(l_{\tiN})-N_s^{t}(l_{\tiN})$ and $[0,t]\ni s\mapsto E_s^t(\delta,l):=\big[B^t(\delta,l)\big]_s-\left\langle B^t(\delta,l)\right\rangle_s$. Both are $\cF_s$-martingales. As above we subdivide $[0,t]$ into subintervals of length $(2N)^{-2}$ and we obtain
\begin{eqnarray*}
&&\left\| \big[E^t(\delta,l)\big]_t \right\|_p = \left\| \sum_{\tau\in[0,t]}\sum_{k=0}^{2N}\big(\frg_{t+\delta-\tau}(k_{\tiN},l_{\tiN})-\frg_{t-\tau}(k_{\tiN},l_{\tiN})\big)^4\big(\rmh_\tau(k_{\tiN})-\rmh_{\tau-}(k_{\tiN})\big)^4\right\|_p\\
&\leq& \frac{2^8}{(2N)^2} e^4 a(p)\Big( 2+\sum_{i=0}^{\lfloor t(2N)^2\rfloor -1}\Big(\frac{2\pi}{(2N)^2(t+\delta-i(2N)^{-2})}\Big)^{\frac{3}{2}}+\Big(\frac{2\pi}{(2N)^2(t-i(2N)^{-2})}\Big)^{\frac{3}{2}}\Big),
\end{eqnarray*}
where the r.h.s.~can be bounded by $(2N)^2$ times a constant, say $d'(p)$, that does not depend on $t,\delta$. Concerning the bracket of $B^t$, Lemma \ref{LemmaBoundsHeatKernel} \rm{(ii)} with $2\gamma\in(0,1/2)$ instead of $\gamma$ yields
\begin{eqnarray*}
\left\|\left\llangle B^t(\delta,l) \right\rrangle_t \right\|_p &=& \left\| \int_0^t\sum_{k=0}^{2N}\big(\frg_{t+\delta-r}(k_{\tiN},l_{\tiN})-\frg_{t-r}(k_{\tiN},l_{\tiN})\big)^2 d\left\llangle M(k_{\tiN}) \right\rrangle_r\right\|_p\\
&\leq& 16N\int_0^t\sup_{k}|\frg_{t+\delta-r}(k_{\tiN},l_{\tiN})-\frg_{t-r}(k_{\tiN},l_{\tiN})| dr\\
&\leq& 8\tc_{\textup{\tiny kernel}}'(2\gamma,T)\delta^{2\gamma} \int_0^t\frac{dr}{(t-r)^{\frac{1}{2}+2\gamma}} \leq \frac{16\,T^{\frac{1}{2}-2\gamma}\tc_{\textup{\tiny kernel}}'(2\gamma,T)\delta^{2\gamma}}{1-4\gamma}.
\end{eqnarray*}
Arguing as above, for all $0 \leq t \leq t' \leq T$ we have
\begin{eqnarray*}
\left\|N_{t}^{t'}(l)-N_{t}^{t}(l)\right\|_p\!\!\! &\leq&\!\!\! 4\tcBDG(p)(t'-t)^{\gamma}\sqrt{\frac{\tc_{\textup{\tiny kernel}}'(2\gamma,T)T^{\frac{1}{2}-2\gamma}}{1-4\gamma}}+\frac{\tcBDG(p)\tcBDG(p/2)^{\frac{1}{2}}d'(p/4)^{\frac{1}{4}}}{\sqrt{2N}},
\end{eqnarray*}
which proves the bound for the second part of the martingale term.\cqfd
\end{proof}
We now state a similar result for the space increments; the proof rests on the same arguments.
\begin{lemma}\label{LemmaHolderSpace}
Fix $p\in [4,\infty)$ and $\beta \in (0,\frac{1}{2}\wedge\upbeta_{\textup{\tiny init}})$. There exists $\tk_{\textup{\tiny space}}=\tk_{\textup{\tiny space}}(p,T,\beta) > 0$ such that for all $N\geq 1$ and all $x, y\in [0,1]$, under $\bbQN_{\nuN}$ we have
$$ \sup\limits_{t\in[0,T]}\left\| \rmh_{t}(x)-\rmh_{t}(y) \right\|_{p} \leq \tk_{\textup{\tiny space}}|x-y|^\beta .$$
\end{lemma}
\noindent Since $\rmh$ is not continuous in time, we consider its interpolation $\brmh$ defined in (\ref{EqInterpol}). The proof of the next result is an easy adaptation of that of Lemma \ref{Lemma:BoundUnifHatBarh}.
\begin{lemma}\label{LemmaBarTime}
Fix $p\in [4,\infty)$ and $\gamma \in (0,\frac{1}{4}\wedge\frac{\upbeta_{\textup{\tiny init}}}{2})$. There exists $\bar{\tk}_{\textup{\tiny time}}=\bar{\tk}_{\textup{\tiny time}}(p,T,\gamma) > 0$ such that for all $0 \leq t \leq t' \leq T$ and all $N\geq 1$, under $\bbQN_{\nuN}$ we have
$$ \sup\limits_{x\in[0,1]}\left\| \brmh_{t'}(x)-\brmh_t(x) \right\|_{p} \leq \bar{\tk}_{\textup{\tiny time}}(t'-t)^\gamma.$$
\end{lemma}
\noindent\textit{Proof of Proposition \ref{PropHolderSpaceTimeBar}.} Fix $N\geq 1$, $\gamma\in(0,\frac{1}{4}\wedge\frac{\upbeta_{\textup{\tiny init}}}{2})$ and $p\in(\frac{2}{\gamma},\infty)$. For all $0 \leq t \leq t' \leq T$ and $x,y \in [0,1]$ we have
\begin{eqnarray*}
\left\|\brmh_t(x)-\brmh_{t'}(y) \right\|_p \leq \left\|\brmh_t(x)-\brmh_t(y) \right\|_p+\left\|\brmh_t(y)-\brmh_{t'}(y) \right\|_p.
\end{eqnarray*}
Using Lemma \ref{LemmaHolderSpace}, we bound the first term on the right
\begin{eqnarray*}
\left\|\brmh_t(x)-\brmh_t(y) \right\|_p &\leq& \left\|\rmh_{\frac{\lfloor t(2N)^2\rfloor}{(2N)^2}}(x)-\rmh_{\frac{\lfloor t(2N)^2\rfloor}{(2N)^2}}(y) \right\|_p + \left\|\rmh_{\frac{\lfloor t(2N)^2\rfloor+1}{(2N)^2}}(x)-\rmh_{\frac{\lfloor t(2N)^2\rfloor+1}{(2N)^2}}(y) \right\|_p\\
&\leq& 2\,\tk_{\textup{\tiny space}}|x-y|^\gamma
\end{eqnarray*}
while the second term on the right can be dealt with using Lemma \ref{LemmaBarTime}. Consequently we obtain
\begin{eqnarray*}
\forall x,y\in[0,1], \forall t,t'\in[0,T],\;\;\;\left\|\brmh_t(x)-\brmh_{t'}(y) \right\|_p &\leq& 2\,\tk_{\textup{\tiny space}}|x-y|^\gamma + \bar{\tk}_{\textup{\tiny time}}|t'-t|^\gamma\\
&\leq& \big(2\,\tk_{\textup{\tiny space}} + \bar{\tk}_{\textup{\tiny time}}\big)\big(|x-y|+|t'-t|\big)^\gamma.
\end{eqnarray*}
Using Kolmogorov's Continuity Theorem, we obtain the existence of a modification of $\brmh$ satisfying the statement of the proposition for all $a \in (0,\frac{p\gamma-2}{p})=(0,\gamma -\frac{2}{p})$. Notice that $\brmh$ is already continuous in the variable $(x,t)$ (it is the interpolation of $\rmh$ taken at the values $x=l/2N,t=k/(2N)^2$ for all integers $l,k$), so it coincides with its modification.\cqfd

\paragraph{Acknowledgements.} We are grateful to Jon Warren for mentioning the existence of the watermelon, to Jeremy Quastel for pointing out Varadhan's lecture notes and to two anonymous referees for their careful reading of this article. C.L.~would also like to thank Lorenzo Zambotti for a fruitful discussion on this work and Julien Reygner for several stimulating discussions. This work was initiated during a visit of C.L.~to the Department of Statistics in Oxford in 2012, this visit was partially funded by the Fondation Sciences Math\'ematiques de Paris. A.M.E.~was supported in part by EPSRC grant EP/I01361X/1.


\begin{thebibliography}{37}

\bibitem{AmbrosioSavareZambotti09}
L.~Ambrosio, G.~Savar{\'e} and L.~Zambotti.
\newblock Existence and stability for {F}okker-{P}lanck equations with
              log-concave reference measure.
\newblock {\em Probab. Theory Related Fields}, 145(3-4):517--564, 2009.

\bibitem{BertiniGiacomin97}
L.~Bertini and G.~Giacomin.
\newblock Stochastic {B}urgers and {KPZ} equations from particle systems.
\newblock {\em Comm. Math. Phys.}, 183(3):571--607, 1997.

\bibitem{Billingsley_CVPM}
P.~Billingsley.
\newblock {\em Convergence of probability measures}.
\newblock Wiley Series in Probability and Statistics: Probability and
  Statistics. John Wiley \& Sons Inc., New York, second edition, 1999.

\bibitem{CaputoMartinelliToninelli08}
P.~Caputo, F.~Martinelli and F.~Toninelli.
\newblock On the approach to equilibrium for a polymer with adsorption and repulsion.
\newblock {\em Electronic Journal of Probability}, vol.~13:213--258, 2008.

\bibitem{CaravennaDeuschel08}
F.~Caravenna and J.~D.~Deuschel.
\newblock Pinning and wetting transition for {$(1+1)$}-dimensional fields with Laplacian interaction.
\newblock {\em Ann. Probab.}, 36(6):2388--2433, 2008.

\bibitem{CaravennaDeuschel09}
F.~Caravenna and J.~D.~Deuschel.
\newblock Scaling limits of {$(1+1)$}-dimensional pinning models with
              {L}aplacian interaction.
\newblock {\em Ann. Probab.}, 37(3):903--945, 2009.

\bibitem{DalMueZam06}
R.~C.~Dalang, C.~Mueller, and L.~Zambotti.
\newblock {Hitting properties of parabolic s.p.d.e.'s with reflection}.
\newblock {\em Ann. Probab.}, 34(4):1423--1450, 2006.

\bibitem{DaPratoZabczykBook}
G.~Da Prato and J.~Zabczyk.
\newblock {\em Stochastic equations in infinite dimensions}, volume 44 of
  {\em Encyclopedia of Mathematics and its Applications}.
\newblock Cambridge University Press, Cambridge, 1992.

\bibitem{DebusscheZambotti07}
A.~Debussche and L.~Zambotti.
\newblock Conservative stochastic {C}ahn-{H}illiard equation with reflection.
\newblock {\em Ann. Probab.}, 35(5):1706--1739, 2007.

\bibitem{DeMasiPresuttiScacciatelli89}
A.~De~Masi, E.~Presutti, and E.~Scacciatelli.
\newblock The weakly asymmetric simple exclusion process.
\newblock {\em Ann. Inst. H. Poincar\'e Probab. Statist.}, 25(1):1--38, 1989.

\bibitem{DemboTsai13}
A.~Dembo and L.~-C.~Tsai.
\newblock Weakly asymmetric non-simple exlcusion process and the KPZ equation.
\newblock {\em arXiv:1302.5760}, 2013.

\bibitem{Hitchhiker12}
E.~Di~Nezza, G.~Palatucci, and E.~Valdinoci.
\newblock Hitchhiker's guide to the fractional {S}obolev spaces.
\newblock {\em Bull. Sci. Math.}, 136(5):521--573, 2012.

\bibitem{DunlopFerrariFontes02}
F.~M.~Dunlop, P.~A.~Ferrari and L.~R.~G.~Fontes.
\newblock A dynamic one-dimensional interface interacting with a wall.
\newblock {\em J. Statist. Phys.}, 107(3-4):705--727, 2002.

\bibitem{Dyson62}
F.~J.~Dyson.
\newblock A {B}rownian-motion model for the eigenvalues of a random matrix.
\newblock {\em J. Mathematical Phys.}, 3:1191--1198, 1962.

\bibitem{Funaki05}
T.~Funaki.
\newblock {\em Stochastic interface models}, volume 1869 of
  {\em Lecture Notes in Mathematics}.
\newblock Springer, Berlin, 2005.

\bibitem{FunaOll01}
T.~Funaki and S.~Olla.
\newblock Fluctuations for {$\nabla\varphi$} interface model on a wall.
\newblock {\em Stochastic Process. Appl.}, 94(1):1--27, 2001.

\bibitem{FunakiSasada10}
T.~Funaki and M.~Sasada.
\newblock Hydrodynamic limit for an evolutional model of two-dimensional
  {Y}oung diagrams.
\newblock {\em Comm. Math. Phys.}, 299(2):335--363, 2010.

\bibitem{FunakiSasadaSauerXie13}
T.~Funaki, M.~Sasada, M.~Sauer, and B.~Xie.
\newblock Fluctuations in an evolutional model of two-dimensional {Y}oung
  diagrams.
\newblock {\em Stochastic Process. Appl.}, 123(4):1229--1275, 2013.

\bibitem{Gartner88}
J.~G{\"a}rtner.
\newblock Convergence towards {B}urgers' equation and propagation of chaos for
  weakly asymmetric exclusion processes.
\newblock {\em Stochastic Process. Appl.}, 27(2):233--260, 1988.

\bibitem{GiacominOllaSpohn01}
G.~Giacomin, S.~Olla and H.~Spohn.
\newblock Equilibrium fluctuations for {$\nabla\varphi$} interface model.
\newblock {\em Ann. Probab.}, 29(3):1138--1172, 2001.

\bibitem{Gillet03}
F.~Gillet.
\newblock Asymptotic behaviour of watermelons.
\newblock {\em arXiv:0307204}, 2003.

\bibitem{JanowskyLebowitz92}
S.~Janowsky and J.~Lebowitz.
\newblock Finite-size effects and shock fluctuations in the asymmetric
  simple-exclusion process.
\newblock {\em Physical Review A}, 45(2):618--625, 1992.

\bibitem{JacodShiryaev}
J.~Jacod and A.~N.~Shiryaev.
\newblock {\em Limit theorems for stochastic processes}, volume 288 of {\em
  Grundlehren der Mathematischen Wissenschaften [Fundamental Principles of
  Mathematical Sciences]}.
\newblock Springer-Verlag, Berlin, second edition, 2003.

\bibitem{Kaigh76}
W.~D.~Kaigh.
\newblock An invariance principle for random walk conditioned by a late return
  to zero.
\newblock {\em Ann. Probab.}, 4(1):115--121, 1976.

\bibitem{KipnisLandimBook}
C.~Kipnis and C.~Landim.
\newblock {\em Scaling limits of interacting particle systems}, volume 320 of
  {\em Grundlehren der Mathematischen Wissenschaften [Fundamental Principles of
  Mathematical Sciences]}.
\newblock Springer-Verlag, Berlin, 1999.

\bibitem{KhorunzhiyMarckert09}
O.~Khorunzhiy and J.~-F.~Marckert.
\newblock Uniform bounds for exponential moment of maximum of a {D}yck path.
\newblock {\em Electron. Commun. Probab.}, 14:327--333, 2009.

\bibitem{KOV89}
C.~Kipnis, S.~Olla, and S.~R.~S. Varadhan.
\newblock Hydrodynamics and large deviation for simple exclusion processes.
\newblock {\em Comm. Pure Appl. Math.}, 42(2):115--137, 1989.

\bibitem{Lacoin13}
H.~Lacoin.
\newblock The scaling limit of polymer pinning dynamics and a one dimensional
  Stefan freezing problem.
\newblock {\em Comm. Math. Phys.}, 331(1):21--66, 2014.

\bibitem{Liggett68}
T.~M.~Liggett.
\newblock An invariance principle for conditioned sums of independent random
  variables.
\newblock {\em J. Math. Mech.}, 18:559--570, 1968.

\bibitem{LLP80}
E.~Lenglart, D.~L{\'e}pingle and M.~Pratelli.
\newblock Pr\'esentation unifi\'ee de certaines in\'egalit\'es de la th\'eorie
  des martingales.
\newblock In {\em Seminar on {P}robability, {XIV} ({P}aris, 1978/1979)
  ({F}rench)}, volume 784 of {\em Lecture Notes in Math.}, pages 26--52.
  Springer, Berlin, 1980.

\bibitem{NualartPardoux92}
D.~Nualart and {\'E}.~Pardoux.
\newblock White noise driven quasilinear {SPDE}s with reflection.
\newblock {\em Probab. Theory Related Fields}, 93(1):77--89, 1992.

\bibitem{RevuzYor}
D.~Revuz and M.~Yor.
\newblock {\em Continuous martingales and Brownian motion}, volume 293 of {\em
  Grundlehren der Mathematischen Wissenschaften [Fundamental Principles of
  Mathematical Sciences]}.
\newblock Springer-Verlag, Berlin, 3rd edition, 1999.

\bibitem{Spitzer76}
F.~Spitzer.
\newblock {\em Principles of random walk}.
\newblock Springer-Verlag, New York, second edition, 1976.
\newblock Graduate Texts in Mathematics, Vol. 34.

\bibitem{Triebel78}
H.~Triebel.
\newblock {\em Interpolation theory, function spaces, differential operators},
  volume~18 of {\em North-Holland Mathematical Library}.
\newblock North-Holland Publishing Co., Amsterdam, 1978.

\bibitem{VaradhanLectureNotes}
S.~R.~S. Varadhan.
\newblock Lectures on hydrodynamic scaling.
\newblock In {\em Hydrodynamic limits and related topics ({T}oronto, {ON},
  1998)}, volume~27 of {\em Fields Inst. Commun.}, pages 3--40. Amer. Math.
  Soc., Providence, RI, 2000.

\bibitem{Vervaat79}
W.~Vervaat.
\newblock A relation between {B}rownian bridge and {B}rownian excursion.
\newblock {\em Ann. Probab.}, 7(1):143--149, 1979.

\bibitem{XuZhang09}
T.~Xu and T.~Zhang.
\newblock {White noise driven SPDEs with reflection: existence, uniqueness and
  large deviation principles}.
\newblock {\em Stochastic Process. Appl.},
  119(10):3453--3470, 2009.

\bibitem{Zambotti01}
L.~Zambotti.
\newblock A reflected stochastic heat equation as symmetric dynamics with
  respect to the 3-d {B}essel bridge.
\newblock {\em J. Funct. Anal.}, 180(1):195--209, 2001.

\end{thebibliography}
\end{document}